\documentclass[a4paper, 12pt]{amsart}

\usepackage{verbatim}
\usepackage{amssymb}
\usepackage{amsbsy}
\usepackage{amscd}
\usepackage{amsmath}
\usepackage{amsthm}
\usepackage[mathscr]{eucal}
\usepackage{mathrsfs}  
\usepackage{bm}

\newtheorem{theorem}{Theorem}\numberwithin{theorem}{section}

\newtheorem{defn}[theorem]{Definition}

\newtheorem{ex}[theorem]{Example}

\newtheorem{prop}[theorem]{Proposition}

\newtheorem{cor}[theorem]{Corollary}

\newtheorem{rem}[theorem]{Remark}

\numberwithin{equation}{section}

\def\Ker{\operatorname{Ker}}
\def\Im{\operatorname{Im}}
\def\Hom{\operatorname{Hom}}  
 
\def\rank{\operatorname{rank}} 
\def\Res{\operatorname{Res}}  

\def\L{\mathbb L}

\def\b{{\bf b}}
\def\x{{\bf x}}
\def\y{{\bf y}}

\newcommand{\Z}{\mathbb{Z}}
\newcommand{\R}{\mathbb{R}}
\newcommand{\C}{\mathbb{C}}

\newcommand{\F}{\mathcal{F}}  

\newcommand{\s}{\mathbb{S}}

\newcommand{\G}{\mathfrak{g}}  
\newcommand{\B}{\mathfrak{b}}
\newcommand{\T}{\mathfrak{t}} 
\newcommand{\p}{\mathfrak{p}}

\newcommand{\hooklongrightarrow}{\lhook\joinrel\longrightarrow}
\newcommand{\twoheadlongrightarrow}{\relbar\joinrel\twoheadrightarrow}

\title
{
   Universal Gysin formulas for the universal Hall-Littlewood functions   
}

\author{Masaki Nakagawa   and Hiroshi Naruse}

\thanks{The  first author is partially supported by the Grant-in-Aid for Scientific  Research 
          (C) 15K04876, Japan Society for  the Promotion of Science}
\thanks{The second author is partially supported by the Grant-in-Aid for Scientific  Research 
          (B) 16H03921, Japan Society for  the Promotion of Science}
\address{Graduate School of  Education \endgraf
                       Okayama University \endgraf
                       Okayama  700-8530 \\ Japan 
}

\address{Graduate School of Education \endgraf 
                        University of Yamanashi   \endgraf
                       Kofu  400-8510 \\ Japan
}

\email{nakagawa@okayama-u.ac.jp}
\email{hnaruse@yamanashi.ac.jp} 



\begin{document}

\begin{abstract}
It is known that the usual Schur  $S$- and $P$-polynomials can be described via the Gysin 
homomorphisms for  flag bundles in the ordinary cohomology theory. Recently, 
P. Pragacz generalized these {\it Gysin formulas} to the Hall-Littlewood polynomials.  
 In this paper,  
we introduce a  {\it universal} analogue  of the Hall-Littlewood polynomials, which 
we call the {\it universal Hall-Littlewood functions}, and give  Gysin formulas 
for various flag bundles in the complex cobordism theory. 
Furthermore, we give two kinds of  the {\it universal} analogue of the schur polynomials, 
and some Gysin formulas for these functions are  established.  

\end{abstract}

\maketitle
\tableofcontents

\section{Introduction}     \label{sec:Introduction}  
\subsection{Gysin formulas}    \label{subsec:GysinFormulas}  
Any  continuous map $f: X  \longrightarrow Y$ between topological spaces 
defines {\it pull-back} homomorphism 
   $f^{*}:  H^{i}(Y)  \longrightarrow H^{i}(X)$ 
in cohomology,  and {\it push-forward} homomorphism 
   $f_{*}:  H_{i}(X)  \longrightarrow H_{i}(Y)$ 
in homology for all $i \in \Z$.  If $X$ is   a compact oriented smooth manifold
of dimension $m$,  the $m$-th homology group $H_{m}(X)$ is isomorphic to 
$\Z$,  with a generator $[X]$ the  {\it fundamental class} of $X$, 
and the following {\it Poincar\'{e} duality map} 
\begin{equation*}
    \mathcal{P}_{X}:  H^{i}(X)   \longrightarrow H_{m-i}(X), \;  \alpha \longmapsto 
\alpha \cap [X] 
\end{equation*} 
is an isomorphism.  If $f: X  \longrightarrow Y$ is a smooth map, with 
$X$ and $Y$  compact oriented smooth manifolds of dimensions $m$ and $n$ 
respectively,  then we get a push-forward homomorphism in cohomology: 
\begin{equation*} 
   f_{*}:  H^{i}(X)  \overset{\sim}{\longrightarrow} H_{m-i}(X)  \overset{f_{*}}{\longrightarrow} 
           H_{m-i}(Y)  \overset{\sim}{\longleftarrow} H^{i-(m-n)}(Y).  
\end{equation*} 
This map is called a  {\it Gysin map} ({\it homomorphism}), 
 {\it push-forward}, or {\it Umkehr map}.  
Intuitively the intersection of  cycles in homology is turned into the product 
of cohomology classes  by means of Poincar\'{e} duality. 
This conversion, together with the computation of the Gysin maps, 
enables us to make some geometric problems into algebraic computations. 
For example, the computation of the Gysin map for various 
flag bundles  is used to determine the cohomology class corresponding to 
 a Schubert variety  (see e.g., Akyildiz \cite{Akyildiz1984}, 
Damon \cite[Theorem 3 (Chern's formula)]{Damon1974}). 
Furthermore,  the similar computation is applied to determining the cohomology 
class of {\it degeneracy locus} (see e.g., Damon \cite[Corollary 3]{Damon1974}, 
Fulton \cite[Chapter 14]{Fulton1998},   Porteous \cite[p.298]{Porteous1971}).  
Thus the computation of  various Gysin maps has a lot of applications in 
geometry, and  there are many formulas describing Gysin maps.  
These formulas are called {\it Gysin formulas} or {\it push-forward formulas} 
in general. 
Although we do not intend  to survey these formulas thoroughly here, 
we shall quote some results related to our work: 
\begin{itemize} 
\item Gysin formulas for flag bundles, Grassmann bundles, or projective bundles 
        are described in  Borel-Hirzebruch \cite[\S 8]{Borel-Hirzebruch1958}, 
        \cite[\S 20]{Borel-Hirzebruch1959},  Buch \cite[\S 7]{Buch2002(Duke)}, 
        Damon \cite{Damon1973}, \cite{Damon1974}, 
        Darondeau-Pragacz \cite{Darondeau-Pragacz2015},  
        Fel'dman \cite[\S 4]{Fel'dman2003},  
        Fulton \cite[\S 14.2]{Fulton1998}, 
       Fulton-Pragacz \cite[Chapter IV, Appendices E, F]{Fulton-Pragacz1998},
       Harris-Tu \cite[\S 2]{Harris-Tu1984},  Ilori \cite{Ilori1978}, 
       Kajimoto-Sugawara \cite{Kajimoto-Sugawara1998},  
       Pragacz  \cite[\S 2]{Pragacz1988},  \cite[\S 4]{Pragacz1996},  \cite{Pragacz2015}, 
       Quillen \cite{Quillen1969}, 
       Sugawara \cite{Sugawara1988},  Tu \cite{Tu2010}, \cite{Tu2015}, 
       Vishik  \cite[\S 5.7]{Vishik2007}.  

\item For a connected complex (semi-simple) Lie group $G_{\C}$ 
         with a Borel subgroup $B$ and a parabolic subgroup $P$ containing $B$, 
         Gysin formulas for the natural projection $G_{\C}/B \longrightarrow G_{\C}/P$ 
         are described in  Akyildiz \cite{Akyildiz1984}, 
         Akyildiz-Carrell \cite{Akyildiz-Carrell1983}, \cite{Akyildiz-Carrell1987}, 
         Fulton-Pragacz \cite[Appendix E]{Fulton-Pragacz1998},  
         Brion \cite{Brion1996}, 
         Kajimoto \cite{Kajimoto1997}.  
\end{itemize} 
As we mentioned above,  most of these formulas are formulated in  the ordinary cohomolgy (or Chow) theory, and many different approaches such as   the {\it residue symbol}  
(Damon \cite{Damon1973}),  
the {\it zeros of holomorphic vector fields on flag varieties} (Akyildiz-Carrell \cite{Akyildiz-Carrell1983}, \cite{Akyildiz-Carrell1987}),  {\it representation theory}  (Brion \cite{Brion1996}), and  the {\it equivariant localization formula of Atiyah-Bott-Berline-Vergne for a torus action} (Tu \cite{Tu2010}, \cite{Tu2015})
 have been used to prove them.

On the other hand,   it is known that the ordinary cohomology theory is
a special case of  {\it complex-oriented} 
generalized cohomology theories which  corresponds  to an  {\it additive} formal group law
(see Example \ref{ex:FGL} (1)). Therefore it is natural to ask 
if the above Gysin formulas in the ordinary cohomology theory can be 
generalized to any complex-oriented generalized cohomology theory. 
For example, Buch \cite[\S 7]{Buch2002(Duke)}  proved a $K$-theoretic  analogue of Pragacz's  Gysin formula \cite{Pragacz1988}.   Quillen \cite{Quillen1969} stated 
(without proof)  the Gysin formula for a projective bundle in the complex cobordism theory.  
As  shown by Quillen \cite{Quillen1969},  the complex cobordism theory $MU^{*}(-)$ 
is {\it universal} among  complex-oriented generalized cohomology theories, 
and therefore it is desirable to formulate these Gysin formulas in the complex cobordism theory. 
The first main purpose of this paper
 is to establish the Gysin formulas for 
general flag bundles  
 in the complex cobordism theory.\footnote{  
{\it Universal Gysin formulas}   
in the title of this paper  signifies  Gysin formulas in the complex cobordism theory 
which is {\it universal} among all complex-oriented generalized cohomology theories.
}   One of our main results, Theorem \ref{thm:Nakagawa-Naruse(GeneralFlagBundles)}, 
is  a generalization of  Pragacz's result  \cite[Proposition 5]{Pragacz2015}   
to the complex cobordism theory.  
Note that Pragacz derived his formula from Brion's result \cite[Proposition 2.1]{Brion1996}
which is proved by the Weyl character formula and the Grothendieck-Riemann-Roch 
theorem.   Here we adopt a topological approach which was developed by 
Bressler-Evens \cite[\S 1]{Bressler-Evens1990}.   Main tools are 
the {\it Becker-Gottlieb transfer} \cite[Theorem 4.3]{Becker-Gottlieb1975} 
 and the {\it Brumfiel-Madsen formula} \cite[Theorem 3.5]{Brumfiel-Madsen1976}. 
 It should be remarked that most of our results in this paper seems to be  valid 
in the {\it algebraic cobordism theory} due to Levine-Morel \cite{Levine-Morel2007}.

\subsection{Gysin formulas for Schur functions}   \label{subsec:GysinFormulasSchurFunctions}
In the previous subsection \S \ref{subsec:GysinFormulas}, we collected 
various Gysin formulas related to our work.    Some formulas involve 
symmetric functions such as {\it Schur} $S$- and  $P$-{\it functions}
 (see e.g.,  Fulton-Pragacz \cite[Chapter IV]{Fulton-Pragacz1998}, 
Pragacz \cite[\S 2]{Pragacz1988}, \cite[\S 4]{Pragacz1996}). 
In order to clarify what we are considering, we shall give one typical 
example: 
Let $E \overset{p}{\longrightarrow} X$ be a complex vector bundle  of rank $n$ 
over a  variety $X$. Let $\pi: G^{1}(E)  \longrightarrow X$ be the associated 
Grassmann bundle of hyperplanes in  $E$.\footnote{
One can naturally identify $G^{1}(E)$ with the associated projective 
bundle $P(E^{\vee})$ of lines in the dual bundle $E^{\vee}$.   
}    On $G^{1}(E)$, we have the {\it tautological exact sequence} of 
vector bundles: 
\begin{equation*} 
       0 \longrightarrow  S  \hooklongrightarrow \pi^{*}(E)  \twoheadlongrightarrow 
     Q  \longrightarrow 0. 
\end{equation*} 
Let $\xi := c_{1}(Q) \in H^{2}(G^{1}(E))$ be the first Chern class 
of the line bundle $Q$, and $x_{1} := \xi, x_{2}, \ldots, x_{n}$ be the {\it Chern roots} 
of $E$.\footnote{
By the {\it splitting principle},  the vector bundle $E$ splits into 
the sum of line bundles when pulled-back to the {\it full flag bundle} 
$\F \ell (E)$ via the projection $\tau: \F \ell (E) \longrightarrow X$.  
The Chern roots  of $E$ are the first Chern classes of these 
line bundles on $\F \ell (E)$.  
}  Then as for the Gysin homomorphism $\pi_{*}:  H^{*}(G^{1}(E)) \longrightarrow 
H^{*}(X)$, it is well-known that 
\begin{equation}   \label{eqn:GysinFormulaProjectiveBundleI}  
   \pi_{*}(\xi^{k})  = s_{k - n  +1}(E) \quad (k \geq 0), 
\end{equation} 
where $s_{i}(E)$ is the $i$-th {\it Segre class} of $E$ (see e.g., Fulton-Pragacz \cite[\S 4.1]{Fulton-Pragacz1998}).  Since the Chern classes $c_{i}(E)$ can be identified with 
the $i$-th elementary symmetric polynomial $e_{i}(\x_{n})$ in 
$\x_{n} = (x_{1}, \ldots, x_{n})$,\footnote{
by means of $\tau^{*}:  H^{*}(X) \longrightarrow H^{*}(\F \ell (E))$, which is known to be 
injective. 
}  
the Segre class $s_{j}(E)$ can be identified with the $j$-th homogeneous 
complete symmetric polynomial $h_{j}(\x_{n})$, which is nothing but the  Schur $S$-polynomial
$s_{(j)}(\x_{n})$ corresponding to the ``one-row'' $(j)$.  Therefore the formula (\ref{eqn:GysinFormulaProjectiveBundleI}) can be interpreted as\footnote{
Strictly speaking, this formula should be considered in $H^{*}(\F \ell (E))$.  
}  
\begin{equation}    \label{eqn:GysinFormulaProjectivBundleII}  
       \pi_{*} (x_{1}^{k})   =   s_{(k - n + 1)}(\x_{n}). 
\end{equation} 
The formula (\ref{eqn:GysinFormulaProjectivBundleII}) can be generalized to the full flag 
bundle $\tau: \F \ell (E) \longrightarrow X$ as follows: 
Let $\lambda = (\lambda_{1}, \lambda_{2}, \ldots, \lambda_{n}) \; (\lambda_{1} \geq \lambda_{2} \geq \cdots \geq \lambda_{n} \geq 0)$ be a partition of length $\leq n$.  
Then Fulton-Pragacz \cite[(4.1)]{Fulton-Pragacz1998} (see also Pragacz \cite[Lemma 2.3]{Pragacz1988}, \cite[Proposition 4.4]{Pragacz1996}, \cite[Example 8]{Pragacz2015})
 showed the following 
formula\footnote{
This formula is also considered in $H^{*}(\F \ell (E))$.  
}: 
\begin{equation}   \label{eqn:Jacobi-TrudiIdentityI}  
     \tau_{*} (x_{1}^{\lambda_{1} + n-1} x_{2}^{\lambda_{2} + n-2} \cdots x_{n}^{\lambda_{n}}) 
  =  s_{\lambda}(\x_{n}). 
\end{equation} 
The formula (\ref{eqn:Jacobi-TrudiIdentityI}) gives the usual Schur $S$-polynomial 
as the push-forward image  of  the  Gysin map from the full flag bundle,   
and is called the {\it Jacobi-Trudi identity} 
in Fulton-Pragacz \cite[p.42]{Fulton-Pragacz1998}.  The essentially identical formula 
has been  obtained by many authors (Damon \cite[Corollary 2]{Damon1974}, 
Harris-Tu \cite[Proposition 2.3]{Harris-Tu1984}, Manivel  \cite[Exercise 3.8.3]{Manivel2001}, 
Pragacz \cite[Proposition 4.4]{Pragacz1996}, \cite[Example 8]{Pragacz2015},  Sugawara \cite[Theorem A]{Sugawara1988}).  
An analogous formula for Schur $P$-polynomial is also given by Pragacz \cite[Corollary 2.7]{Pragacz1988},  \cite[Example 11]{Pragacz2015}.  Recently, Pragacz \cite{Pragacz2015} 
succeeded in generalizing the above Gysin formulas for Schur $S$- and $P$-polynomials 
to the {\it Hall-Littlewood polynomials}, which interpolate between Schur $S$- 
and $P$-polynomials (see Macdonald \cite[Chapter III, \S 2]{Macdonald1995}).

On the other hand,  it is well-known that the usual Schur $S$-polylnomials  
(resp. $P$-polynomials)   {\it represent} 
the {\it Schubert classes} of the complex Grassmannian (resp. Lagrangian 
Grassmannian)   in the ordinary cohomology 
(see e.g., Fulton \cite[\S 9.4]{Fulton1997},  Pragacz \cite[\S 6]{Pragacz1991}).  
In order to generalize the above facts to other cohomology theories  such as 
$K$-theory,  complex cobordism theory (or algebraic cobordism theory) 
and their torus equivariant versions,  various generalizations, analogues, and 
deformations of Schur functions have been introduced by many authors.  
We shall quote some of them for convenience of the readers
(in the following, $\lambda$ (resp. $\nu$) is understood to be a 
partition (resp. strict partition) of length $\leq n$, and $\x_{n} = (x_{1}, \ldots, x_{n})$
is a sequence of  $n$ independent variables, and $\b = (b_{1}, b_{2}, \ldots)$ 
is a sequence of ``deformation paremeters''): 
\begin{itemize} 
\item The {\it factorial Schur polynomials}  $s_{\lambda}(\x_{n}|\b)$ 
          (see Ikeda-Naruse \cite[\S 5.1]{Ikeda-Naruse2009}, 
          Macdonald \cite[Chapter 1, Examples 20]{Macdonald1995}, 
         \cite[6th variation]{Macdonald1992}, 
       Molev-Sagan \cite{Molev-Sagan1999})   represent 
         the Schubert classes of the torus equivariant  cohomology of the 
        complex Grassmannian (see Knutson-Tao \cite[\S 6]{Knutson-Tao2003}, 
        Ikeda-Naruse \cite[Theorem 5.4]{Ikeda-Naruse2009}).

\item The {\it factorial Schur} $P$- and $Q$-{\it polynomials}  $P_{\nu}(\x_{n}|\b)$, 
          $Q_{\nu}(\x_{n}|\b)$  introduced by 
         Ivanov \cite[Definitions 2.10 and 2.13]{Ivanov2004} (see also Ikeda-Mihalcea-Naruse 
         \cite[\S 4.2]{IMN2011}) represent the Schubert classes of the torus equivariant 
         cohomology of the orthogonal and Lagrangian Grassmannians
          (see Ikeda \cite[Theorem 6.2]{Ikeda2007}, Ikeda-Naruse \cite[Theorem 8.7]{Ikeda-Naruse2009}). 

\item The {\it factorial Grothendieck polynomials}  $G_{\lambda}(\x_{n}|\b)$  
         introduced by McNamara \cite[Definition 4.1]{McNamara2006}  (see also Ikeda-Naruse \cite[(2.13), (2.14)]{Ikeda-Naruse2013}) 
represent the Schubert classes of the torus equivariant $K$-theory of the complex 
Grassmannian (see Ikeda-Naruse \cite{Ikeda-Naruse2013}). 

\item The {\it $K$-theoretic factorial} $P$- and $Q$-{\it polynomials}
          $GP_{\nu}(\x_{n}|\b)$, $GQ_{\nu}(\x_{n}|\b)$ introduced by 
       Ikeda-Naruse \cite[Definition 2.1]{Ikeda-Naruse2013} represent 
        the Schubert classes of the torus equivariant $K$-theory of the 
       orthogonal and Lagrangian Grassmannians
       (see Ikeda-Naruse \cite[Theorem 8.3]{Ikeda-Naruse2013}).  
 
\item   The {\it universal factorial Schur} ($S$-) {\it functions}\footnote{
           Notice that $s_{\lambda}(\x_{n}|\b)$ and $G_{\lambda}(\x_{n}|\b)$ are 
           polynomials in $\x_{n}$,  whereas $s^{\L}_{\lambda}(\x_{n}|\b)$ is a 
           formal power series in $\x_{n}$.  
} $s^{\L}_{\lambda}(\x_{n}|\b)$, 
            $P$- and $Q$-{\it functions} $P^{\L}_{\nu}(\x_{n}|\b)$, $Q^{\L}_{\nu}(\x_{n}|\b)$
           were introduced by the authors 
            (see \cite[Definitions 4.1 and  4.10]{Nakagawa-Naruse2016} or (\ref{eqn:Definitions^L(x_n|b)}) and 
           (\ref{eqn:DefinitionP^L(x_n|b)Q^L(x_n|b)}) in this paper).   These functions are 
           {\it universal} analogue of the above polynomials.  
\end{itemize}  
Thus 
our second main purpose of this paper
is the introduction of the 
{\it universal} analogue of the Hall-Littlewood polynomials, 
and to establish the Gysin formulas for various Schur functions
in the complex cobordism theory.     The universal analogue of the Hall-Littlewood polynomials 
denoted by $H^{\L}_{\lambda}(\x_{n};t)$ are defined in Definition \ref{df:DefinitionUHLF}, and  
we call  them  the {\it universal Hall-Littlewood functions}.   Then our main result in this direction is 
 Corollary \ref{cor:Nakagawa-Naruse(CharacterizationUHLF)} which 
gives the universal Hall-Littlewood function as the push-forward image of
the Gysin map from a  partial flag bundle, thus generalizing  Pragacz's result 
(see Corollary \ref{cor:GysinFormulasH-LPolynomials}).  
The Jacobi-Trudi identity (\ref{eqn:Jacobi-TrudiIdentityI}) can also be 
formulated in the complex cobordism theory (see Corollary \ref{cor:Nakagawa-Naruse(CharacterizationUSF)}).  With regard to the universal Schur functions 
$s^{\L}_{\lambda}(\x_{n})$ and the universal Hall-Littlewood functions 
$H^{\L}_{\lambda}(\x_{n}; t)$, a comment is in order: 
The usual Hall-Littlewood polynomial denoted by $P_{\lambda}(x_{1}, \ldots, x_{n}; t)$
in Macdonald's book \cite[Chapter III, \S 2]{Macdonald1995}  
reduces to the usual Schur polynomial  $s_{\lambda}(x_{1}, \ldots, x_{n})$ under 
the specialization $t = 0$.   However,  we found that our $H^{\L}_{\lambda}(\x_{n}; t)$ 
does not necessarily reduce to the universal  Schur function $s^{\L}_{\lambda}(\x_{n})$ 
when $t = 0$.  Thus the specialization $H^{\L}_{\lambda}(\x_{n}; 0)$ gives 
another  universal analogue of Schur functions, which we call 
the {\it new universal Schur functions}.\footnote{
It is desirable to give these functions more specific name which 
characterize them (a poposal will be given in \S \ref{subsec:Damon'sResolution}).    
}   We shall discuss these new functions 
denoted $\s^{\L}_{\lambda}(\x_{n})$ (and their factorial version denoted 
$\s^{\L}_{\lambda}(\x|\b)$) separately in the final section  \S \ref{sec:NUFSF}. 
Among our results in this section are Theorem \ref{thm:GysinFormulasNUFSF}
which reveals the difference between the ``old''  and the ``new'' universal 
factorial Schur functions.   As an application of the Gysin formulas for 
the new universal factorial Schur functions, we formulate 
the {\it Thom-Porteous fromula} in  the complex cobordism theory
(see Theoren \ref{thm:Thom-PorteousFormulaComplexCobordism}).  
Moreover, we shall show that the new universal factorial Schur functions 
{\it represent}  the Schubert classes of the complex Grassmannian 
in the complex cobordism theory (see Theorem \ref{thm:DamonClass}).

\subsection{Organization of the paper}
The paper is organized as follows:  In Section 2, we shall give 
topological preliminaries   needed to develop our work. 
The key concepts are {\it complex-oriented generalized cohomology 
theory},  {\it Gysin maps}, 
{\it Becker-Gottlieb transfer}, {\it Brumfiel-Madsen formula}, 
{\it Bressler-Evens formula}.  Various types of  {\it Gysin formulas} 
are reviewed at the end of this section.   
Section 3 is devoted to the introduction of the  {\it universal} analogue 
of the usual Schur $S$-,  $P$-, $Q$-, and Hall-Littlewood polynomials.  
Especially, the {\it universal Hall-Littlewood functions}, which are 
the central theme of this paper, are introduced. 
In Section 4, after reviewing the Gysin formulas for various Schur functions, 
we shall give the {\it universal} analogues of these formulas. 
In order to establish these formulas, the Bressler-Evens formula 
plays the crucial role.   
In Section 5, we introduce the {\it new universal factorial Schur functions}. 
If we set all the deformation parameters  to be $0$, the {\it new universal Schur functions} 
can be obtained, and these functions coincide with the universal Hall-Littlewood 
functions under the specialization $t = 0$.  We also give some Gysin formulas 
for these  new universal Schur functions.  

\vspace{0.3cm}  

\textbf{Acknowledgments.} \quad    
We would like to thank Takeshi Ikeda,  Thomas Hudson, Tomoo Matsumura  for helpful comments and valuable conversations.  Especially, Tomoo Matsumura kindly explained 
his recent work (Hudson-Matsumura \cite{Hudson-Matsumura2016}) to us, 
that is closely related to  our current work.  

\section{Topological preliminaries}  \label{sec:TopologicalPreliminaries}  
\subsection{Complex-oriented generalized cohomology theory}   \label{subsec:COGCT}  
A {\it generalized cohomology theory}  $h^{*}(-) =  \bigoplus_{n \in \Z} h^{n} (-)$ is a contravariant functor from the category  $\mathcal{C W}$ of  CW complexes  to the category of graded 
abelian  groups which satisfies  all the axioms of Eilenberg-Steenrod \cite[Chapter I, 3c]{Eilenberg-Steenrod1952}    except the {\it dimension 
axiom}.   Thus  $h^{n}(\mathrm{pt})$ is not necessarily zero even if $n \neq 0$.  Here   $\mathrm{pt}$ means a space consisting of a single point.  The cohomology  $h^{*} := h^{*}(\mathrm{pt})$ of a point 
is called the {\it coefficient group}.     
 In what follows, we assume that the theory $h^{*}(-)$ is {\it multiplicative}, that is, for 
CW pairs $(X, A)$  and $(Y, B)$, there exists an {\it external product} 
\begin{equation*} 
    h^{k}(X, A)  \otimes h^{l}(Y, B)  \longrightarrow h^{k + l}((X, A) \times (Y, B))  \quad k, l \in \Z,  
\end{equation*} 
that satisfies certain axioms (see e.g., Dold \cite[\S 4]{Dold1962}). 
Under this assumption,   $h^{*}$ becomes a graded-commutative ring, 
and for a space $X$,  the cohomology  ring $h^{*}(X)$ has  an $h^{*}$-module structure. 
Furthermore, if we consider the {\it infinite} CW complexes, we need suitable  axioms about limits such as the {\it  additivity axiom}  or {\it wedge axiom}  due to Milnor \cite{Milnor1962}.  
In what follows,  when we refer to a {\it generalized cohomology theory},
it means  a multiplicative generalized cohomology theory 
defined on $\mathcal{C W}$ satisfying the additivity axiom unless otherwise stated.  
Let $\tilde{h}^{*}(-)$ denote the corresponding {\it reduced} cohomology theory,\footnote{
For a CW-complex $X$ with base point $x_{0}$, 
the reduced cohomology  $\tilde{h}^{*}(X)$ is defined  to be $h^{*}(X, x_{0})$.
}      and let $j: \C P^{1} \hooklongrightarrow 
\C P^{\infty}$ be the canonical inclusion of $\C P^{1} \approx  S^{2}$ into the infinite 
complex projective space. 

\begin{defn} [Adams \cite{Adams1974}, Part II, p.37; Switzer \cite{Switzer1975}, \S 16.27] 
   A generalized cohomology theory $h^{*}(-)$ is called complex-orientable  if  there exists an element $x^{h}  \in 
  \tilde{h}^{2}(\C P^{\infty})$ such that $j^{*}(x^{h})$ is a generator of $\tilde{h}^{2}(\C P^{1}) 
\cong \tilde{h}^{2}(S^{2}) \cong \tilde{h}^{0}(S^{0}) \cong h^{0}(\mathrm{pt}) = h^{0}$. 
\end{defn} 
If this element $x^{h}$ is specified, then $h^{*}(-)$ is said to be {\it complex-oriented}, and 
$x^{h}$ is called the {\it orientation class}.  
Then it is known that the cohomology  ring of  the infinite projective space $\C P^{\infty}$ is $h^{*}(\C P^{\infty})
= h^{*}[[x^{h}]]$, a formal power series ring with   the  given   generator $x^{h}  \in \tilde{h}^{2}(\C P^{\infty})$ (see Adams \cite[Part II, Lemma 2.5]{Adams1974}).    

Complex-orientability implies a lot of useful properties. 
We recall here some of them.

\subsubsection{Chern classes}  
For a complex vector bundle $E  \overset{p}{\longrightarrow}  X$, one can define 
the {\it $h^{*}$-theory Chern classes}  $c^{h}_{i}(E)  \in h^{2i}(X) \; (i = 0, 1, 2, \ldots, n = \rank E)$
by the {usual} Grothendieck's method (see Conner-Floyd \cite[Theorem 7.6]  
{Conner-Floyd1966}, Grothendieck \cite[\S 3]{Grothendieck1958}, Switzer \cite[Theorem 16.2]{Switzer1975}).  The total Chern class of $E$ is given by 
 $c^{h}(E) := \sum_{i = 0} ^{n} c^{h}_{i}(E)$.   Then the usual Whitney product formula 
is given by $c^{h}(E \oplus F) = c^{h}(E) \cdot c^{h}(F)$ for two complex vector 
bundles $E$, $F$.

\subsubsection{Formal group law}     \label{subsubsec:FGL}   
If $L$, $M$ are complex line bundles over $X$, then 
\begin{equation*} 
     c_{1}^{h}(L  \otimes M)  =  F_{h} (c_{1}^{h}(L),  c_{1}^{h}(M)), 
\end{equation*} 
where   
\begin{equation*} 
      F_{h} (X, Y)    =  X + Y + \sum_{i, j \geq 1} a_{i, j}^{h}  X^{i} Y^{j}   \in  h^{*}[[X, Y]]  \quad 
     (a_{i, j}^{h} \in  h^{2(1 - i - j)}) 
\end{equation*} 
 is   a  (one dimensional commutative) formal group law  over the graded ring $h^{*}$ 
associated  with the cohomology theory $h^{*}(-)$.  
Then  the formal power series $F_{h}(X, Y)$ satisfies the conditions
\begin{enumerate} 
\item [(i)]  $F_{h}(X, 0) = X$, $F_{h}(0, Y) = Y$, 
\item [(ii)] $F_{h}(X, Y) = F_{h}(Y, X)$, 
\item [(iii)]  $F_{h}(X, F_{h}(Y, Z)) =  F_{h} (F_{h}(X, Y), Z)$. 
\end{enumerate} 
We shall use this formal group law to define the {\it formal sum}, 
{\it formal inverse}, and {\it formal subtraction}.  
For two  indeterminates $X$, $Y$, the formal sum $X +_{F} Y$ is defined 
as 
\begin{equation*} 
  X +_{F} Y := F_{h}(X, Y) = X + Y + \sum_{i, j \geq 1} a^{h}_{i, j} X^{i}Y^{j}  
\in h^{*}[[X, Y]]. 
\end{equation*} 
Denote by 
\begin{equation*}  
     [-1]_{F} (X)   = \iota_{F}(X)  = \overline{X} =   -X + \sum_{j \geq 2} c^{h}_{j} X^{j}   \in h^{*}[[X]]
\end{equation*} 
the formal inverse series.  
To be precise,  $[-1]_{F}(X)$ is the unique formal power series satisfying the condition  
$F_{h}  (X,  [-1]_{F}(X))  \equiv 0$, or equivalently $X +_{F}   [-1]_{F}(X) = 0$. 
This formal inverse allows us to define the formal subtraction: 
\begin{equation*} 
   X -_{F} Y :=  X +_{F}  [-1]_{F}(Y) = X +_{F}  \overline{Y}. 
\end{equation*}  
Finally,  we define $[0]_{F}(X) := 0$, and inductively,  
\begin{equation*}
   [n]_{F}(X)  :=  [n-1]_{F}(X) +_{F} X  =  F_{h}([n-1]_{F}(X), X)  =   \underbrace{X +_{F} X +_{F}  \cdots +_{F} X}_{n}  
\end{equation*}
for $n \geq 1$.  We also define    $[-n]_{F}(X) := [n]_{F}([-1]_{F}(X)) = [-1]_{F}([n]_{F}(X))$ 
for $n \geq 1$.     
We call $[n]_{F}(X)$ the {\it $n$-series} in the following.  

\begin{ex}   \label{ex:FGL}  
\quad 
\begin{enumerate} 
\item For the ordinary cohomology theory $($with integer coefficients$)$ $h = H$, 
the coefficient ring  is $H^{*} = H^{*}(\mathrm{pt}) = \Z$ $(H^{0} = \Z$, $H^{k} = 0 \; (k \neq 0))$.  
     We choose  the standard orientation, namely the class of a hyperplane $x^{H}  \in \tilde{H}^{2}(\C P^{\infty})$.  Then the associated formal group law is  the {\it additive} formal group law 
       $F_{H}(X, Y) = F_{a} (X, Y) =  X + Y$, and 
       the formal inverse is given by $[-1]_{H}(X) = -X$.

\item  For the $($topological$)$ $K$-theory  
   $h = K$,  the coefficient ring  is $K^{*} = K^{*}(\mathrm{pt}) = \Z[\beta, \beta^{-1}]$, 
with $\beta  := 1 -  \eta_{1}^{\vee}    \in K^{-2}(\mathrm{pt}) \cong  \tilde{K}(S^{2})$, 
where $\eta_{1}$ stands for the {\it tautological} $($or Hopf$)$ line bunlde over $\C P^{1} \cong S^{2}$, and $\eta_{1}^{\vee}$ its dual. 
We choose  the standard orientation 
 $x^{K} :=    \beta^{-1}(1 - \eta_{\infty}^{\vee}) \in \tilde{K}^{2}(\C P^{\infty})$, where $\eta_{\infty}$ 
stands for the tautological line bundle   over $\C P^{\infty}$.\footnote{
We adopt the  convention due to Bott \cite[Theorem 7.1]{Bott1969}, 
Levine-Morel \cite[Example 1.1.5]{Levine-Morel2007}
so that the $K$-theory first Chern class of a line bundle $L$ (over a space $X$) 
is given by $c^{K}_{1}(L) = \beta^{-1}(1 - L^{\vee})$, where $L^{\vee}$ denotes the dual 
bundle of $L$.  In this convention,  the orientation class $x^{K}$ is equal to 
the $K$-theory first Chern class of  the  bundle  $\eta_{\infty}$, 
namely $c^{K}_{1}(\eta_{\infty}) = \beta^{-1}(1  - \eta_{\infty}^{\vee})$.  
}   
Then the associated formal group law is the {\it multiplicative} formal group law
 $F_{K}(X, Y) = F_{m}(X, Y) = X + Y -  \beta XY$,  
and the formal inverse is given by 
  \begin{equation*} 
    [-1]_{K}(X)  =  -\dfrac{X}{1 -  \beta X}  =  -X - \beta X^2 - \beta^{2} X^3 - \beta^{3} X^4 - \cdots.  
  \end{equation*}  
     
\item For the complex cobordism theory $h  = MU$,   
        the coefficient ring $MU^{*} = MU^{*}(\mathrm{pt})$ is a polynomial algebra over $\Z$ 
      on generators of degrees $-2, -4, \ldots $ $($see e.g., Adams \cite[Part II, Theorem 8.1] {Adams1974}$)$. 
       As in Adams \cite[Part II, Examples (2.4)]{Adams1974}, Ravenel \cite[Example 4.1.3]{Ravenel2004},
     we take the orientation class $x^{MU}  
       \in \tilde{MU}^{2}(\C P^{\infty})$ to be the $($stable$)$ homotopy class of the map 
     $\C P^{\infty} \simeq BU(1) \overset{\sim}{\longrightarrow}  MU(1)$, 
     where $MU(1)$ denotes the Thom space of the universal line bundle over $BU(1)$. 
     Then the associated formal group law 
       \begin{equation*} 
          F_{MU}(X, Y) = X + Y + \sum_{i, j \geq 1} a^{MU}_{i, j} X^{i} Y^{j}, 
          \quad a^{MU}_{i, j}  \in MU^{2(1 -i - j)},  
       \end{equation*} 
       is a {\it universal formal group law} first shown by 
       Quillen \cite[Theorem 2]{Quillen1969}.   To be precise,  
      for any formal group law $F$ over a commutative ring 
       $R$ with unit, there exists a unique ring homomorphism 
      $\theta: MU^{*} \longrightarrow R$ 
       such that $F (X, Y) =  (\theta_{*} F_{MU})(X, Y) := 
      X  + Y + \sum_{i, j \geq 1} \theta (a^{MU}_{i, j}) X^{i}Y^{j}$.  
       Quillen also showed that the coefficient ring $MU^{*}$ is 
       isomorphic to the {\it Lazard ring} $\L$ $($see \S $\ref{subsec:UFGL}$ in 
this paper$)$.  
\end{enumerate} 
\end{ex} 

\subsection{Gysin maps} 
For a certain kind of map $f: X \longrightarrow Y$ between spaces, 
 the so-called {\it Gysin map}, {\it push-forward}, or {\it Umkehr map},  usually denoted by  
$f_{*}:  h^{*}(X) \longrightarrow h^{*}(Y)$ can be defined. 
Here are some examples of Gysin maps:  
\begin{itemize}  
\item (Classical Gysin map in the ordinary cohomology theory): 
    For a smooth map 
\begin{equation*} 
    f:  M \longrightarrow N,  
\end{equation*}     
between compact oriented smooth manifolds,   the {\it Gysin map} 
\begin{equation*} 
    f_{*}:  H^{q}(M)  \longrightarrow H^{q  - (\dim M - \dim N)}(N) 
\end{equation*} 
can be defined by $f_{*} :=  \mathcal{P}_{N}^{-1} \circ f_{*} \circ \mathcal{P}_{M}$, 
where $\mathcal{P}_{M}$ (resp. $\mathcal{P}_{N}$) denotes the Poincar\'{e} duality 
isomorphism from cohomology to homology. 

\item (Integration along (over)  the fiber): 
         For a fibration  $F \hooklongrightarrow E  \overset{\pi}{\longrightarrow} B$ with 
         the base $B$ simply-connected and  the fiber $F$  a compact connected manifold, 
         Borel-Hirzebruch \cite[\S 8]{Borel-Hirzebruch1958}  
         defined a push-forward map called the {\it integration along the fiber}:  
         \begin{equation*} 
                  \pi_{*}  =  \natural :  H^{q}(E)  \longrightarrow  H^{q - \dim F} (B).  
        \end{equation*} 
        
\item  (Gysin map in the $K$-theory $K(-)$):  
           For a proper morphism  $f: X \longrightarrow Y$ of non-singular varieties, 
           Grothendieck constructed an additive map $f_{!}:  K_{\circ} (X)  \longrightarrow K_{\circ} (Y)$  defined by 
          \begin{equation*} 
                    f_{!}([\mathcal{F}]) :=  \sum_{i \geq 0}  (-1)^{i} [R^{i}f_{*} (\mathcal{F})], 
          \end{equation*} 
           for $\mathcal{F}$ a coherent sheaf.  Here $R^{i} f_{*}(\mathcal{F})$ is {\it Grothendieck's higher direct image sheaf} (see e.g., Fulton \cite[\S 15.1]{Fulton1998}).

\item (Gysin map in the complex cobordism theory $MU^{*}(-)$):  
          In \cite{Quillen1971}, Quillen gave a geometric interpretation of the 
         complex cobordism theory $MU^{*}(X)$, where $X$ is assumed to be 
        a manifold.   In his interpretation, an element of $MU^{*}(X)$ is given 
       by a {\it cobordism class} of a {\it proper} and {\it complex-oriented} map  $f: Z \longrightarrow  X$.  
     A proper complex-oriented map $g: X \longrightarrow Y$ 
      of dimension $d$ induces a map 
      \begin{equation*} 
            g_{*}:   MU^{q}(X)    \longrightarrow MU^{q-d}(Y) 
      \end{equation*} 
      which sends the cobordism class of $f: Z  \longrightarrow X$ into the cobordism class 
     of $g \circ f:  Z  \longrightarrow Y$.  
\end{itemize}

All these Gysin maps have the common properties which can be  axiomatized as follows (see Bressler-Evens \cite[p.801, (E)]{Bressler-Evens1990}, 
Levine-Morel \cite[Definition 1.1.2]{Levine-Morel2007}, Quillen  \cite[\S 1]{Quillen1971}): 
Let $h^{*}(-)$ be  a  complex-oriented generalized cohomology theory  defined on 
a suitable category of spaces.  For a morphism $f: X  \longrightarrow Y$, one has 
a Gysin  map 
$f_{*}:  h^{*}(X)  \longrightarrow h^{*}(Y)$ having the following basic properties: 
\begin{enumerate}  
\item (Naturality):   For a composite $g \circ f:  X  \overset{f}{\longrightarrow}  Y  \overset{g}{\longrightarrow} 
Z$,  one has $(g \circ f)_{*}   = g_{*} \circ f_{*}$.  

\item (Projection formula):   For $x   \in h^{*}(X)$ and $y \in h^{*}(Y)$, one has 
         $f_{*}(f^{*}(y) \cdot x)  = y \cdot f_{*}(x)$.  

\item (Base-change): For the following commutative diagram
\begin{equation*} 
 \begin{CD} 
            X \times_{Z} Y    @>{\tilde{f}}>>         Y \\
                @V{\tilde{g}}VV                                 @VV{g}V \\
            X                        @>{f}>>                Z,       \\
\end{CD} 
\end{equation*} 
one has $f^{*} \circ g_{*} = \tilde{g}_{*} \circ \tilde{f}^{*}$. Here $X \times_{Z} Y$ denotes 
the fiber product of $X$ and $Y$ over $Z$, namely $X \times_{Z} Y  = \{ (x, y)  \in X \times Y \;   | \; f(x) = g(y) \}$.  
\end{enumerate}

\subsection{Becker-Gottlieb transfer}  
Let $F   \overset{\iota}{\hooklongrightarrow}  E  \overset{\pi}{\longrightarrow}  B$ be a fiber bundle
 whose fiber $F$ is a compact smooth manifold, whose structure group $G$ is a compact 
Lie group acting smoothly on $F$, and whose base space $B$ is a finite complex. 
In this setting, Becker-Gottlieb \cite[\S 3]{Becker-Gottlieb1975} constructed 
a {\it stable} map\footnote{
For a suitable positive integer $m$, we have a map 
\begin{equation*} 
    \tau (\pi):     S^{m} (B^{+})  \longrightarrow S^{m}  (E^{+}),  
\end{equation*} 
where $S$ denotes the reduced suspension.  
}
\begin{equation*} 
    \tau (\pi)  : B^{+} \longrightarrow E^{+}, 
\end{equation*} 
 called the {\it Becker-Gottlieb transfer}. 
Here $B^{+}$ means  the union of $B$ with a point. 
  For any generalized cohomology theory $h^{*}(-)$, the map 
$\tau  (\pi): B^{+} \longrightarrow E^{+}$ induces a  ``wrong-way''  degree-preserving  homomorphism 
\begin{equation*} 
   \tau (\pi)^{*}:  h^{*}(E)  \longrightarrow h^{*}(B),  
\end{equation*}    
which is also called the Becker-Gottlieb transfer.  $\tau (\pi)^{*}$  is {\it not} a ring homomorphism, 
but is an $h^{*}(B)$-module homomorphism (Becker-Gottlieb \cite[(5.3)]{Becker-Gottlieb1975}), that is, 
 the following formula holds: 
\begin{equation*} 
    \tau (\pi)^{*}(  \pi^{*}(x) \cdot y)  =  x \cdot \tau (\pi)^{*} (y),  \quad 
    x  \in h^{*}(B), \;   y  \in  h^{*}(E). 
\end{equation*} 
Furthermore the Gysin map and the Becker-Gottlieb transfer are related as follows
(see Becker-Gottlieb \cite[Theorem 4.3]{Becker-Gottlieb1975}): 
\begin{equation}     \label{eqn:Becker-GottliebTransferGysinMap}  
    \tau (\pi)^{*} (x)   =  \pi_{*} (\chi^{h}  (T_{\pi}) \cdot  x)   \quad   (x  \in h^{*}(E)). 
\end{equation} 
Here $T_{\pi}$ is 
the ({\it tangent}) {\it bundle along the fibers} ({\it of $\pi$}) (see e.g., Borel-Hirzebruch \cite[\S 7.4]{Borel-Hirzebruch1958}),\footnote{ 
If we assume that $E$, $B$ are also smooth manifolds,
$T_{\pi}$ is a subbundle of the tangent bundle $T(E)$ of $E$.  
The fiber of $T_{\pi}$   over a point $y \in E$ consists of  all tangent vectors 
at the point $y$ which are tangent to the fiber $(\cong F)$ through $y$.  
If we denote by $\iota: F \hooklongrightarrow E$ the fiber inclusion, then 
$\iota^{*}(T_{\pi})  \cong T(F)$, the tangent bundle of $F$.  
} 
and $\chi^{h}  (T_{\pi})$ denotes the  $h^{*}$-theory Euler class of $T_{\pi}$.  
Here $T_{\pi}$ is regarded as a real vector bundle, and assumed to be {\it $h^{*}$-oriented}  (For the notion of $h^{*}$-orientablility, see Atiyah-Bott-Shapiro \cite[\S 12]{Atiyah-Bott-Shapiro1964},  Dold \cite[\S 4]{Dold1962}). 
The $h^{*}$-theory Euler class $\chi^{h} (T_{\pi})$ is defined with respect to this orientation. 
In practice (see Bressler-Evens \cite[\S 1]{Bressler-Evens1990} and 
\S \ref{subsec:Bressler-EvensFormula}),  we require that the fiber $F$
is smooth and {\it almost complex},  and the structure group of $F$ preserves the almost 
complex structure.  Hence the tangent bundle $T (F)$ has a complex vector bundle structure, and so does the bundle along the fibers $T_{\pi}$.  
If the cohomology theory $h^{*}(-)$ is  complex-oriented, then 
the $h^{*}$-theory Euler class $\chi^{h} (T_{\pi})$ is nothing but the
$h^{*}$-theory {\it top}  Chern class $c^{h}_{f}(T_{\pi})$, where $2f$
is the {\it real} dimension of $F$.

\subsection{Brumfiel-Madsen formula}   \label{subsec:Brumfiel-Madsen}  
Let $G$ be a compact connected (semi-simple) Lie group (of rank $\ell$)  
 with a maximal torus $T$ ($\cong (U(1))^{\ell}$). 
Let $H$ be a closed connected subgroup of $G$ of maximal rank, i.e., $T \subset H$. 
Denote by $W_{G}$ and $W_{H}$ the Weyl group of $G$ and $H$ respectively. 
We then have a natural inclusion $W_{H}  \subset W_{G}$.  
Suppose that $G \hooklongrightarrow P   {\longrightarrow}  B$ is a principal $G$-bundle, 
and consider the following associated bundles: 
\begin{equation*} 
\begin{array}{ccccc} 
    & G/T     \hooklongrightarrow    & E_{1} := P \times_{G} (G/T)  
                  \overset{\pi_{1}}{\longrightarrow}   &  B,   \medskip \\
    & G/H   \hooklongrightarrow  &   E_{2} := P \times_{G} (G/H)  
                 \overset{\pi_{2}}{\longrightarrow}   &  B.   \medskip 
\end{array} 
\end{equation*} 
Then there is a fiber bundle $H/T \hooklongrightarrow E_{1}  \overset{\pi}{\longrightarrow}  E_{2}$, 
where the projection $\pi$ is induced from the natural projection (also denoted by the same symbol)
$\pi: G/T \longrightarrow G/H$,  
and  we have the following commutative diagram: 
\begin{equation}    \label{eqn:CD(Brumfiel-Madsen)}    
\begin{CD} 
      E_{1} = P \times_{G} (G/T)      @>{\pi}>>       E_{2} = P \times_{G} (G/H)  \\
        @V{\pi_{1}}VV                                              @VV{\pi_{2}}V  \\
      B                                      @>{=}>>                B.    
\end{CD} 
\end{equation} 
The usual right action of the Weyl group $W_{G}$ on $G/T$ induces a right  action on
$E_{1} = P \times_{G} (G/T)$ over $B$, 
i.e., a bundle map over $B$.   
As a subgroup of $W_{G}$, the Weyl group $W_{H}$ of $H$ also acts on $E_{1}$, which 
is a bundle map over $E_{2}$.   Therefore the coset $\overline{w} = wW_{H}  \in W_{G}/W_{H}$
defines a well-defined map $\pi \circ w:  E_{1}  \overset{w}{\longrightarrow} E_{1} 
\overset{\pi}{\longrightarrow}  E_{2}$, which induces a homomorphism in 
cohomology: $w \circ  \pi^{*}:  h^{*}(E_{2})  \overset{\pi^{*}}{\longrightarrow} h^{*}(E_{1})  \overset{w}{\longrightarrow}    h^{*}(E_{1})$.   
Then Brumfiel-Madsen  established the following useful formula:  
\begin{theorem} [Brumfiel-Madsen \cite{Brumfiel-Madsen1976}, Theorem 3.5; Bressler-Evens \cite{Bressler-Evens1990}, Theorem 1.3]  \label{thm:Brumfiel-MadsenFormula} 
In the above setting, we have 
\begin{equation*} 
       \pi_{1}^{*} \circ  \tau (\pi_{2})^{*}  =  \sum_{\overline{w}  \in W_{G}/W_{H}}  w \circ  \pi^{*}. 
\end{equation*} 
\end{theorem}
\noindent
As a special case where  $H = T$, we have 
\begin{cor} [Bressler-Evens \cite{Bressler-Evens1990}, Corollary 1.4]   \label{cor:Brumfiel-MadsenFormula}  
\begin{equation*} 
       \pi_{1}^{*}  \circ \tau (\pi_{1})^{*}   =   \sum_{w \in W_{G}}  w. 
\end{equation*} 
\end{cor} 
We apply   Corollary  \ref{cor:Brumfiel-MadsenFormula} 
to the case where  $P = EG$, the universal space\footnote{
$EG$ is a contractible space on which $G$ acts freely.  
} for $G$,  so that $B = BG$, the classifying space of $G$,  and $E_{1}  =  EG  \times_{G} (G/T)  \simeq BT$. 
In this case, the fibration $G/T \hooklongrightarrow E_{1}  \overset{\pi_{1}}{\longrightarrow}  B$
becomes  the following classical Borel fibration: 
\begin{equation*}   
  G/T    \overset{\iota}{\hooklongrightarrow}   BT  \overset{\rho \, =  \, \rho (T, G)}{\longrightarrow}  BG. 
\end{equation*} 
Thus we have the following: 
\begin{equation}  \label{eqn:Brumfiel-MadsenFormula(BorelFibrationG/TBTBG)}   
     \rho^{*} \circ  \tau (\rho)^{*}  = \sum_{w \in W_{G}}  w. 
\end{equation} 
Combining (\ref{eqn:Becker-GottliebTransferGysinMap}) and (\ref{eqn:Brumfiel-MadsenFormula(BorelFibrationG/TBTBG)}), we have 
\begin{equation}    \label{eqn:Becker-Gottlieb-Brumfiel-Madsen(FullFlag)}  
      \rho^{*} \circ \rho_{*}(\chi^{h}  (T_{\rho}) \cdot f)  =   \sum_{w \in W_{G}} w \cdot f  \quad 
    \text{for} \; f  \in  h^{*}(BT). 
\end{equation} 
From this formula, Bressler-Evens \cite{Bressler-Evens1990} derived a 
useful formula which is explained briefly in the next subsection.

\subsection{Bressler-Evens formula}   \label{subsec:Bressler-EvensFormula}   
Before stating their result, we shall recall  some facts from Lie theory.      
Let $T \subset G$ be as above and $G/T  \overset{\iota}{\hooklongrightarrow}   BT \overset{\rho}{\longrightarrow} BG$
the Borel fibration.  Let $G_{\C}$  and $T_{\C}$ be the complexification of $G$ and 
$T$ respectively.  Thus $G_{\C}$ is  a connected complex (semi-simple) Lie group 
with maximal compact subgroup $G$, and $T_{\C}  \cong (\C^{*})^{\ell}$.   
 Denote by $B$ a Borel subgroup  of $G_{\C}$ containing $T_{\C}$. Then the natural inclusion $G \hookrightarrow 
G_{\C}$ induces a diffeomorphism $G/T  \overset{\sim}{\rightarrow} G_{\C}/B$. 
By this identification, the {\it full flag manifold}  $G/T$ is equipped with a complex structure.  
Let $\G$ and $\T$  be the Lie algebras of $G$ and $T$, and 
$\G_{\C} = \G \otimes_{\R}  \C$ and $\T_{\C} = \T \otimes_{\R}  \C$ their complexification. 
Then  $\G_{\C}$ and $\T_{\C}$  are the Lie algebras of $G_{\C}$ and $T_{\C}$
respectively.  
  Then we have the root space decomposition 
\begin{equation}  \label{eqn:RootSpaceDecompositionG}   
  \G_{\C}  =  \T_{\C}  \oplus 
                   \bigoplus_{\alpha \in \Delta^{+}}  (\G_{\alpha} \oplus \G_{-\alpha}), 
\end{equation} 
where $\G_{\alpha} = \{ x \in \G_{\C} \, | \, [h, x]   = \alpha (h) x \; (\forall h  \in \T_{\C}) \}$, 
and the system of {\it positive}  roots $\Delta^{+}  \subset \Hom_{\C}(\T_{\C}, \C)$ corresponds to the Lie algebra 
$\B$ of $B$. Thus $\B = \T_{\C} \oplus \bigoplus_{\alpha \in \Delta^{+}}  \G_{\alpha}$. 
We set $\Delta^{-} := - \Delta^{+}$, the system of {\it negative} roots, 
and $\Delta := \Delta^{+}  \sqcup  \Delta^{-}$, the system of roots.  
It is well-known that the tangent bundle $T(G/T)$ to  $G/T$ is
isomorphic to the vector bundle $G \times_{T}  (\G/\T)$ associated with the principal $T$-bundle
$T \hooklongrightarrow G \longrightarrow G/T$ and $T$-module $\G/\T$.  
Since we have the natural identification $\G/\T \cong \G_{\C}/\B$, we have the 
following isomorphism as {\it complex}  vector bundles: 
\begin{equation*} 
   T(G/T) \cong  G \times_{T}  (\G_{\C}/\B). 
\end{equation*} 
Hence the tangent bundle along the fibers $T_{\rho}$ is isomorphic to the {\it complex} vector bundle 
$ET \times_{T} (\G_{\C}/\B)$ associated with the universal $T$-bundle $T \hookrightarrow ET \rightarrow BT$.  Thus 
\begin{equation}   \label{eqn:TangentBundleAlongFibersT_rho}  
      T_{\rho}   \cong  ET \times_{T}  (\G_{\C}/\B). 
\end{equation}  
For each character $\chi  \in \Hom \, (T, U(1)) \cong \Hom \, (T_{\C}, \C^{*})  = \hat{T}_{\C}$,   
we have the associated complex 
line bundle $L_{\chi}$ over $BT$ defined by $L_{\chi} := ET  \times_{T}  \C 
=  (ET \times \C)/(y, v)  \sim (y \cdot t,  \chi(t)^{-1} v)$.   
Each root $\alpha  \in  \Delta \subset \Hom_{\C} (\T_{\C}, \C)$
defines a character 
$\chi_{\alpha}  \in \Hom \, (T_{\C}, \C^{*})$, and we have the associated complex 
line bundle $L_{\chi_{\alpha}}$, which is also denoted by $L_{\alpha}$ for simplicity.  
By (\ref{eqn:RootSpaceDecompositionG}) and (\ref{eqn:TangentBundleAlongFibersT_rho}), 
we have 
\begin{equation*} 
    T_{\rho}  \cong ET  \times_{T}  \bigoplus_{\alpha \in \Delta^{+}}  \G_{-\alpha}  
   \cong  \bigoplus_{\alpha  \in \Delta^{+}}   L_{-\alpha}.  
\end{equation*} 
Therefore the $h^{*}$-theory top Chern class (Euler class) of $T_{\rho}$ is given 
by 
\begin{equation}   \label{eqn:c^h_top(T_rho)}   
    c^{h}_{{\rm top}} (T_{\rho})   =  c^{h}_{{\rm top}}  \left  (\bigoplus_{\alpha \in \Delta^{+}}  L_{-\alpha} \right )
                              = \prod_{\alpha  \in \Delta^{+}}   c^{h}_{1}(L_{-\alpha}).  
\end{equation} 
From (\ref{eqn:c^h_top(T_rho)})  and the formula  (\ref{eqn:Becker-Gottlieb-Brumfiel-Madsen(FullFlag)}),  
Bressler-Evens deduced the following:\footnote{
For the ordinary cohomology theory $h = H$, this formula was already 
proved by Borel-Hirzebruch \cite[Theorem 20.3]{Borel-Hirzebruch1959}
(see also Tu \cite[\S 11.1]{Tu2015}).  
}
\begin{theorem}  [Bresser-Evens \cite{Bressler-Evens1990},  Theorem 1.8]   \label{thm:Bressler-EvensThm1.8} 
Let  $h^{*}(-)$ be a  complex-oriented generalized cohomology theory. We assume that 
the coefficient ring $h^{*}$ is torsion-free.\footnote{
We made this assumption for simplicity.  See Bressler-Evens \cite[Remark 1.10]{Bressler-Evens1992} for less restrictive assumptions. 
} 
Then for $f  \in h^{*}(BT)$,  we have 
\begin{equation*}  
        \rho^{*} \circ \rho_{*} (f)  =   \sum_{w \in W_{G}}   w \cdot 
                          \left [ \dfrac{f}{\prod_{\alpha \in \Delta^{+}} c_{1}^{h}   (L_{-\alpha})}     \right ] . 
\end{equation*} 
\end{theorem} 

One can easily extend Theorem \ref{thm:Bressler-EvensThm1.8}  to the case of 
{\it partial flag manifolds}.  Thus let  $\Theta \subset \Pi := \{ \text{simple roots} \}  
\subset \Delta^{+}$ be a subset of the set of simple roots,  
and $P = P_{\Theta}$ be the corresponding {\it parabolic subgroup},  and 
put  $H = H_{\Theta} :=  G  \cap  P$.  Then it is known that  $H$ is 
 the centralizer of the toral subgroup  defined by ``$\alpha  = 0$'' 
for $\forall \alpha \in \Theta$, and hence a closed  connected subgroup of $G$
of maximal rank, i.e., $T \subset H$.  
The homogeneous manifold  $G/H$ has a complex structure (see e.g., Borel-Hirzebruch  \cite[\S 13.5]{Borel-Hirzebruch1958}), and we have the following Borel fibration
\begin{equation*} 
 G/H    \overset{\iota_{H}}{\hooklongrightarrow}  BH  \overset{\sigma = \rho (H, G)}{\longrightarrow}  BG.   
\end{equation*}    
There exists a natural identification  
 $G/H  \overset{\sim}{\rightarrow} G_{\C}/P$, and this homogeneous manifold 
is called a partial flag manifold (see  e.g., Borel-Hirzebruch  \cite[\S 14.3]{Borel-Hirzebruch1958}).  
 Denote by $\Delta_{H}  = \Delta^{+}_{H}  \sqcup \Delta^{-}_{H}$
the system of roots of $H$ with respect to $T$.   The Lie algebra $\p = \p_{\Theta}$ of $P$ is given by 
\begin{equation}   \label{eqn:RootSpaceDecompositionP}  
    \p  =  \B  \oplus \bigoplus_{\alpha \in \Delta_{H}^{+}} \G_{-\alpha} 
         = \T_{\C} \oplus  \bigoplus_{\alpha  \in \Delta^{+}}  \G_{\alpha}  \oplus  \bigoplus_{\alpha \in \Delta_{H}^{+}} \G_{-\alpha}.  
\end{equation} 
The  tangent bundle of $G/H \overset{\sim}{\rightarrow}  G_{\C}/P$ is given by 
\begin{equation*} 
         T(G/H)  \cong G \times_{H}  (\G_{\C}/\p). 
\end{equation*} 
Hence the tangent bundle along the fibers $T_{\sigma}$ is isomorphic to 
the complex vector bundle $EH \times_{H} (\G_{\C}/\p)$ associated with 
the universal $H$-bundle $H  \hooklongrightarrow EH  \longrightarrow BH$.  
Thus 
\begin{equation}    \label{eqn:TangentBundleAlongFibersT_sigma}  
     T_{\sigma} \cong EH  \times_{H}   (\G_{\C}/\p). 
\end{equation}

We apply Theorem \ref{thm:Brumfiel-MadsenFormula} to the case where $P = EG$ and 
$H = H_{\Theta}$, 
so that $B = BG$, and $E_{2} =  EG \times_{G} (G/H) \simeq BH$. 
Then the fibration $G/H  \hooklongrightarrow E_{2}  \overset{\pi_{2}}{\longrightarrow}  B$
becomes the  Borel fibration
   $G/H    \overset{\iota}{\hooklongrightarrow}  BH  \overset{\sigma}{\longrightarrow}  BG$, 
and the commutative diagram (\ref{eqn:CD(Brumfiel-Madsen)}) yields the following commutative 
diagram: 
\begin{equation}    \label{eqn:CD(BorelFibration)}    
\begin{CD} 
        BT      @>{\pi  = \rho(T, H)}>>                     BH \\
        @V{\rho = \rho(T, G)}VV                                              @VV{\sigma = \rho(H, G)}V  \\
      BG                                      @>{=}>>                BG.    
\end{CD} 
\end{equation} 
Then by Theorem \ref{thm:Brumfiel-MadsenFormula}, we have 
\begin{equation}   \label{eqn:Brumfiel-MadsenFormula(BorelFibrationG/HBHBG)} 
     \rho^{*} \circ  \tau (\sigma)^{*}  =  \sum_{\overline{w} \in W_{G}/W_{H}}  w \circ \pi^{*}. 
\end{equation} 
Combining (\ref{eqn:Becker-GottliebTransferGysinMap}) and (\ref{eqn:Brumfiel-MadsenFormula(BorelFibrationG/HBHBG)}),  we have 
\begin{equation}     \label{eqn:Becker-Gottlieb-Brumfiel-Madsen(PartialFlag)}  
      \rho^{*} \circ \sigma_{*}  (c^{h}_{{\rm top}} (T_{\sigma}) \cdot f)  =  \sum_{\overline{w} \in W_{G}/W_{H}}  w \cdot \pi^{*} (f)  \quad \text{for} \; f   \in h^{*}(BH). 
\end{equation}   

On the other hand, pulling back the tangent bundle along the fibers $T_{\sigma}$ to $BT$  via tha map $\pi = \rho (T, H)$, 
we have from (\ref{eqn:TangentBundleAlongFibersT_sigma}),  (\ref{eqn:RootSpaceDecompositionG}), and (\ref{eqn:RootSpaceDecompositionP}),  
\begin{equation*}
     \pi^{*}  (T_{\sigma})   \cong ET \times_{T}   (\G_{\C}/\p)  \cong ET \times_{T}  \bigoplus_{\alpha \in \Delta^{+} \setminus \Delta_{H}^{+}}  \G_{-\alpha}  \cong  \bigoplus_{\alpha  \in \Delta^{+} \setminus \Delta_{H}^{+}} L_{-\alpha}. 
\end{equation*}  
Therefore  the $h^{*}$-theory top  Chern class (Euler class) of $\pi^{*}(T_{\sigma})$ is given by 
\begin{equation}  \label{eqn:c^h_top(T_sigma)}   
     \pi^{*}c^{h}_{{\rm top}}(T_{\sigma})  = c^{h}_{{\rm top}} \left (\bigoplus_{\alpha  \in \Delta^{+} \setminus \Delta_{H}^{+}}  L_{-\alpha}  \right )  = \prod_{\alpha \in  \Delta^{+} \setminus \Delta_{H}^{+}}  c^{h}_{1}(L_{-\alpha}). 
\end{equation} 
Then by the analogous argument to that of Bressler-Evens \cite[Theorem 1.8]{Bressler-Evens1990},   (\ref{eqn:c^h_top(T_sigma)}), and the formula (\ref{eqn:Becker-Gottlieb-Brumfiel-Madsen(PartialFlag)}),   one obtains the following 
(Note that by the commutativity of the diagram (\ref{eqn:CD(BorelFibration)}), 
one has $\rho^{*} \circ \sigma_{*}  =  \pi^{*} \circ \sigma^{*} \circ \sigma_{*}$): 
\begin{cor}    \label{cor:Bressler-EvensThm1.8Cor}  
For  $f  \in h^{*}(BH)$, we have 
\begin{equation} 
  \pi^{*} \circ   \sigma^{*} \circ \sigma_{*}  (f)  = \sum_{\overline{w} \in W_{G}/W_{H}} w \cdot 
\left [  
          \dfrac{\pi^{*}(f) }
                  {\prod_{\alpha \in \Delta^{+} \setminus \Delta^{+}_{H}}  c^{h}_{1} (L_{-\alpha})} 
\right ].  
\end{equation} 
\end{cor} 

\subsection{Various Gysin formulas}   \label{subsec:VariousGysinFormulas}  
As mentioned in the introduction,  various  types of  {\it Gysin formulas}  
related to the Gysin maps 
are known (see e.g.,  Akyildiz \cite{Akyildiz1984}, 
Akyildiz-Carrell  \cite{Akyildiz-Carrell1987},  
Buch \cite{Buch2002(Duke)},  
Damon \cite{Damon1973}, \cite{Damon1974},  Darondeau-Pragacz \cite{Darondeau-Pragacz2015},   Fel'dman \cite{Fel'dman2003}, 
Fulton \cite{Fulton1998},  
Fulton-Pragacz \cite{Fulton-Pragacz1998},  Harris-Tu \cite{Harris-Tu1984}, 
Ilori \cite{Ilori1978},  
Jozefiak-Lascoux-Pragacz \cite{JLP1982},  Kajimoto \cite{Kajimoto1997},   Quillen \cite{Quillen1969},  Pragacz \cite{Pragacz1988}, \cite{Pragacz1996},  \cite{Pragacz2015},  Sugawara \cite{Sugawara1988},  Tu \cite{Tu2010}, \cite{Tu2015}). 
In this subsection, we shall take up  typical examples of these formulas.

\subsubsection{Gysin formulas of type ``$G_{\C}/B  \longrightarrow G_{\C}/P$''} 
First recall the result due to Akyildiz-Carrell \cite{Akyildiz-Carrell1987}. 
In order to state their result, we shall use the same notation as in \S \ref{subsec:Bressler-EvensFormula} with  a slightly minor change.  So Let $G_{\C}  \supset B  \supset T_{\C}$  be as in \S \ref{subsec:Bressler-EvensFormula}.   Consider  the  parabolic  subgroup $P = P_{\Theta}$  corresponding 
to a subset $\Theta \subset \Pi = \{ \text{simple roots} \}
        \subset \Delta^{+}$.     Thus the homogeneous variety 
$G_{\C}/P$ is a partial  flag variety. 
Denote by $W_{\Theta}$ (resp. $\Delta_{\Theta}$)
the  Weyl group (resp. root system)   corresponding  to $P_{\Theta}$.  
Let  $\chi  \in \hat{T}_{\C} = \Hom \, (T_{\C}, \C^{*})$ be a character.  
By composing the natural projection\footnote{
Recall that $B$ is the semi-direct product of $T_{\C}$ and its unipotent part.  
} $B   \twoheadlongrightarrow T_{\C}$ with 
$\chi:  T_{\C}  \longrightarrow \C^{*}$,   we have a character 
$\chi_{B} = \chi:  B  \longrightarrow \C^{*}$.  
Then one can define a complex line bundle $M_{\chi}$ over $G_{\C}/B$ 
in the usual manner.  
By assigning each character $\chi \in \hat{T}_{\C}$
the first Chern class $c_{1}(M_{\chi})  \in H^{2}(G_{\C}/B ; \C)$, 
the {\it characteristic homomorphism}\footnote{
In topology, it is customary that the character group $\hat{T}_{\C} \cong \Hom \, (T, U(1))$
is identified with $H^{1}(T) \cong H^{2}(BT)$ (this latter identification is given by the 
{\it negative transgression} (see Borel-Hirzebruch \cite[\S 10.1]{Borel-Hirzebruch1958})). 
Under this identification, one has the isomorphism $\mathrm{Sym} \, (\hat{T}_{\C}) \cong 
H^{*}(BT;\C)$ as algebras, and the characteristic homomorphism 
$c$ can be identified with the induced homomorphism $\iota^{*}:  H^{*}(BT;\C)  \longrightarrow H^{*}(G/T; \C)$.  
}  
\begin{equation*} 
               c:  R := \mathrm{Sym} \, (\hat{T}_{\C})  \rightarrow H^{*}(G_{\C}/B; \C)       
\end{equation*}
is defined.  Here $\mathrm{Sym} \, (\hat{T}_{\C})$ means the symmetric algebra 
of  $\hat{T}_{\C}$ over $\C$.  
Let  $\pi: G_{\C}/B  \longrightarrow G_{\C}/P$  be the  natural projection.  
Then Akyildiz-Carrell showed the following formula (see also Brion \cite{Brion1996}): 
\begin{theorem}[Akyildiz-Carrell  
\cite{Akyildiz-Carrell1987}, Theorem 1; Brion \cite{Brion1996}, Proposition 1.1]     \label{thm:Akyildiz-Carrell}  
The Gysin homomorphism  
    $\pi_{*}:  H^{*}(G_{\C}/B; \C)  \longrightarrow H^{*}(G_{\C}/P; \C)$   
is  given by   
\begin{equation}  \label{eqn:Akyildiz-CarrellFormula}   
     \pi^{*} \circ  \pi_{*}  \, (c(f)) =  c \left ( \sum_{w \in W_{\Theta}}  \dfrac{\det (w) \; w \cdot f }{ \prod_{\alpha  \in \Delta^{+}_{\Theta}} \alpha } \right )   \quad \text{for}  \; f \in R.   
\end{equation} 
\end{theorem}  
\noindent
Here $\det \, (w)$ means $(-1)^{\ell (w)}$, where $\ell (w)$ denotes the 
length of the Weyl group element $w$.  
Since $w \cdot \prod_{\alpha \in \Delta_{\Theta}^{+}}     \alpha = (-1)^{\ell (w)} \prod_{\alpha \in \Delta_{\Theta}^{+} }  \alpha$ for any $w \in W_{\Theta}$, 
the above formula (\ref{eqn:Akyildiz-CarrellFormula})  can also be written as follows: 
\begin{equation*} 
        \pi^{*} \circ  \pi_{*}  \, (c(f)) 
    =  c \left ( \sum_{w \in W_{\Theta}}  
                      w \cdot \left   [   \dfrac{f}{ \prod_{\alpha  \in \Delta^{+}_{\Theta}} \alpha }
                                  \right ]  
            \right )   \quad \text{for}  \; f \in R. 
\end{equation*} 
Akyildiz-Carrell   proved  this  formula by the method based on the zeros of holomorphic 
vector fields on  relevant flag varieties.   Brion  proved this  formula by the Weyl character formula and Grothendieck-Riemann-Roch theorem.  

\subsubsection{Gysin formulas of type  ``$\mathcal{F}  \ell (E)  \longrightarrow X$''} \label{subsubsec:GysinFormulasFl(E)->X} 
Next recall the result due to Fulton-Pragacz \cite[Chapter IV]{Fulton-Pragacz1998}. 
Let $E  \overset{p}{\longrightarrow}   X$ be a  complex vector bundle of rank $n$ over 
a {\it variety}.\footnote{
Actually it is enough to assume that the base space is some  {\it nice  space}, 
say, a {\it paracompact space},  so that  
the classification theorem of vector bundles holds.    
} 
Denote by $\tau = \tau_{E}: \mathcal{F} \ell (E)  \longrightarrow X$ the 
associated flag bundle parametrizing   successive   flags of   {\it quotients} 
 of $E$ of  ranks $n-1, \ldots, 2, 1$.   Thus we have the {\it tautological 
sequence of flag of quotient bundles}  
\begin{equation*} 
           \tau^{*}E = Q^{n} \twoheadlongrightarrow Q^{n-1} \twoheadlongrightarrow \cdots \twoheadlongrightarrow Q^{2} \twoheadlongrightarrow Q^{1} \twoheadlongrightarrow Q^{0} = 0, 
\end{equation*} 
where $\rank \, (Q^{i}) = i  \; (i = 0, 1, 2, \ldots, n)$.\footnote{
Here we followed the convention as in e.g., Pragacz \cite[\S 2]{Pragacz1988}. 
On $\F \ell (E)$, we also have the {\it tautological sequence of flag of subbundles} 
\begin{equation*} 
    0 = S_{0} \subset S_{1} \subset S_{2} \subset \cdots \subset S_{n-1} \subset S_{n} = \tau^{*}E, 
\end{equation*} 
where $\rank \, (S_{i})  = i \; (i = 0, 1, 2, \ldots, n)$.  These two tautological sequences 
are related by $Q^{i}  = \tau^{*}E/S_{n-i}$. Therefore if we define the line bundles 
$L_{i} := S_{i}/S_{i-1} \; (i = 1, 2, \ldots, n)$, then we have $L^{i} = L_{n+1-i}
\; (i = 1, 2, \ldots, n)$. 
}   
Define the line bundles 
 $L^{i} := \Ker \,  (Q^{i} \twoheadlongrightarrow Q^{i-1}) \; (i = 1, 2, \ldots, n)$ 
 over  $\F \ell (E)$.  
Put     $x_{i} := c_{1}(L^{i})  \in H^{2}(\F \ell (E)) \; (i = 1, 2, \ldots, n)$ (the  {\it Chern roots}  of $E$).  
Then Fulton-Pragacz showed the following formula:  
\begin{theorem}[Pragacz \cite{Pragacz1988}, Lemma 2.4; \cite{Pragacz1996}, Proposition 4.3 (ii); 
Fulton-Pragacz  \cite{Fulton-Pragacz1998}, p.41]     \label{thm:Fulton-Pragacztau^*tau_*}  
For a polynomial $f(X_{1}, \ldots, X_{n})     \in H^{*}(X)[X_{1}, \ldots, X_{n}]$, we have
\begin{equation*} 
      \tau^{*} \circ \tau_{*}(f(x_{1}, \ldots, x_{n}))  
      =  \sum_{w \in S_{n}} 
                 w \cdot \left [   \dfrac{f(x_{1}, \ldots, x_{n})} 
                                                { \prod_{1 \leq i < j \leq n}  (x_{i} - x_{j})} 
                            \right ].    
\end{equation*}  
\end{theorem} 
\noindent 
Thus the Gysin map $\tau_{*}$  is  given by  a   certain {\it symmetrizing operator} called 
the {\it Jacobi symmetrizer}  in Fulton-Pragacz \cite[\S 4.1]{Fulton-Pragacz1998}.  
As is well-known, the flag bundle $\F \ell (E)$ can be constructed as a sequence of 
projective bundles, and  Fulton-Pragacz proved this formula by the induction on 
the rank of $E$.

\subsubsection{Application of the Bressler-Evens formula}    \label{subsubsec:ApplicationBressler-EvensFormula}   
Most of these Gysin formulas are formulated in the ordinary cohomology rings
or Chow rings, and proved by many different ways.  
We remark that  the Bressler-Evens formulas (Theorem \ref{thm:Bressler-EvensThm1.8} and Corollary  $\ref{cor:Bressler-EvensThm1.8Cor}$) enable us to show these Gysin formulas  by a unified manner.   

For Theorem \ref{thm:Akyildiz-Carrell},  one can argue as follows: 
Put $H   =  G \cap P$ as in \S \ref{subsec:Bressler-EvensFormula}.  
Then by the {\it classification theorem of principal bundles}, we have 
a {\it classifying map}  $h: G/H  \longrightarrow BH$ and its lift 
$\tilde{h}:  G/T \longrightarrow BT$,  and the following 
diagram is commutative:     
  \begin{equation*} 
 \begin{CD} 
             G_{\C}/B \cong G/T    @>{\tilde{h} \simeq \iota}>>       BT  \\
                @V{\pi}VV                                      @VV{\rho = \rho (T, H)}V \\
             G_{\C}/P \cong G/H     @>{h \simeq \iota_{H}}>>               BH.  \\
\end{CD} 
\end{equation*} 
Note that the above classifying map $h$ (resp. $\tilde{h}$) coincides with  
the fiber inclusion $\iota_{H}: G/H  \hooklongrightarrow BH$ (resp. $\iota: G/T \hooklongrightarrow BT$) up to homotopy.   
Then by Theorem \ref{thm:Bressler-EvensThm1.8} and the base-change property of 
Gysin maps,   we compute 
\begin{equation*} 
\begin{array}{lllll} 
    \pi^{*} \circ \pi_{*} \,  (c(f)) &
     =  \pi^{*} \circ \pi_{*} \circ \iota^{*} (f)     
    = \pi^{*} \circ \iota_{H}^{*} \circ \rho_{*}  (f)   
   = \iota^{*} \circ \rho^{*} \circ \rho_{*} (f)  \medskip \\
                                     & = \iota^{*} \left (   
                                                             \displaystyle{\sum_{w \in W_{\Theta}}}   
                                                                 w \cdot \left [ 
                                                                                   \dfrac{f}{\prod_{\alpha \in \Delta_{\Theta}^{+}}   c_{1}(L_{-\alpha}) }   
                                                                             \right ] 
                                                       \right )     \medskip   \\
                                        & =  c \left ( 
                                                        \displaystyle{\sum_{w \in W_{\Theta}}} 
                                                                   w \cdot \left [ 
                                                                                    \dfrac{f}{\prod_{\alpha \in \Delta_{\Theta}^{+}} \alpha }  
                                                                               \right ] 
                                                  \right ),    
\end{array} 
\end{equation*} 
as required.   Here we used the convention that $c_{1}(L_{\alpha}) = -\alpha$
for a root $\alpha \in \Delta$.

For Theorem \ref{thm:Fulton-Pragacztau^*tau_*}, one can  argue as follows:  
 By the {\it classification theorem of complex vector bundles}, we have the {\it classifying map} 
        $h: X  \longrightarrow BU(n)$, 
and its lift $\tilde{h}:  \F \ell (E)  \longrightarrow BT^{n}$, 
and the following  diagram is commutative:     
 \begin{equation*} 
 \begin{CD} 
          \mathcal{F} \ell (E)   @>{\tilde{h}}>>       BT^{n}  \\
                @V{\tau}VV                                 @VV{\rho}V \\
            X                        @>{h}>>                 BU(n).  \\
\end{CD} 
\end{equation*} 
Let $\chi_{i}: T^{n} \longrightarrow  U(1)$ be the character which takes an element 
$t  = \mathrm{diag} \, (t_{1},  \ldots, t_{n})  \in T^{n}$ to the $i$-th entry 
$t_{i} \in U(1) \; (i = 1, 2, \ldots, n)$.  The line bundles $L_{\chi_{i}}$ over $BT^{n}$ 
can be constructed as in \S \ref{subsec:Bressler-EvensFormula}.  
Put $y_{i} := -c_{1}(L_{\chi_{i}})  = c_{1}((L_{\chi_{i}})^{\vee})  \in H^{2}(BT^{n}) \; (i  = 1, 2, \ldots, n)$ 
(notice our convention). 
Then the positive root system for  $G= U(n)$ is given by 
$\Delta^{+}   = \{ y_{i} - y_{j} \; (1 \leq i < j \leq n) \}$ as a subset of 
$H^{2}(BT^{n})$.   The Weyl group  $W_{U(n)}$ of $U(n)$ can be identified 
with the symmetric group $S_{n}$ by the usual manner.  
Let $\gamma^{n} \longrightarrow BU(n)$ be the {\it universal} or {\it canonical} 
vector bundle over $BU(n)$ (see Milnor-Stasheff \cite[\S 14, p.161]{Milnor-Stasheff1974}). 
Then the associated flag bundle $\F \ell (\gamma^{n}) \longrightarrow BU(n)$ can be 
identified with the Borel fibration $BT^{n}   \overset{\rho}{\longrightarrow} BU(n)$. 
As noted in \S \ref{subsubsec:GysinFormulasFl(E)->X},  there is the tautological 
sequence of flag of subbundles $S_{0}  \subset S_{1} \subset S_{2} \subset \cdots \subset S_{n-1} \subset 
S_{n} = \rho^{*}(\gamma^{n})$ over $BT^{n}$.   The usual line bundles $L_{i} \; (i = 1, 2, \ldots, n)$ over $BT^{n}$ are defined by $L_{i} := S_{i}/S_{i-1}
\; (i = 1, 2, \ldots, n)$.   
Then it is easily verified that the line bundle $L_{i}$  
can be identified with the line bundle $L_{\chi_{i}}$.  
Therefore as for the Chern roots of $E$, we have 
\begin{equation*} 
    x_{i} = c_{1}(L^{i})   =  c_{1}(L_{n + 1 -i})  =  c_{1}(\tilde{h}^{*}(L_{n+1-i})) 
          = \tilde{h}^{*} (c_{1}(L_{n + 1 - i}))   = \tilde{h}^{*}(-y_{n+1-i}). 
\end{equation*} 
Then by Theorem \ref{thm:Bressler-EvensThm1.8} and the base-change property of Gysin maps,  we compute\footnote{
We also used the well-known fact that the map $\tilde{h}^{*}:   H^{*}(BT^{n}) \twoheadlongrightarrow 
H^{*}(\F \ell (E))$ is surjective.  
}
\begin{equation*} 
\begin{array}{llll} 
   \tau^{*} \circ \tau_{*} (f(x_{1}, \ldots, x_{n}))  & =  \tau^{*} \circ \tau_{*} \circ \tilde{h}^{*} (f(-y_{n}, \ldots, -y_{1}))  \medskip \\
      & =        \tilde{h}^{*} \circ \rho^{*} \circ \rho_{*} (f(-y_{n}, \ldots, -y_{1}))   \medskip \\
     & =  \tilde{h}^{*} \left (   
                                         \displaystyle{\sum_{w \in W_{U(n)}}} 
                                                   w \cdot \left [  
                                                                         \dfrac{ f(-y_{n}, \ldots, -y_{1})} 
                                                                                {\prod_{\alpha \in \Delta^{+}}  c_{1}(L_{-\alpha}) } 
                                                              \right ]  
                            \right )   \medskip   \\
    & =   \tilde{h}^{*} \left (   
                                         \displaystyle{\sum_{w \in S_{n}}} 
                                                   w \cdot \left [  
                                                                         \dfrac{ f(-y_{n}, \ldots, -y_{1})} 
                                                                                {\prod_{1 \leq i <  j \leq n} 
                                                                                     (y_{i} - y_{j})   } 
                                                              \right ]  
                            \right )   \medskip   \\
    & =   \displaystyle{\sum_{w \in S_{n}}} 
                                                   w \cdot \left [  
                                                                         \dfrac{ f(x_{1}, \ldots, x_{n})} 
                                                                                {\prod_{1 \leq i <  j \leq n} 
                                                                                     (x_{i} - x_{j})   } 
                                                              \right ],  
\end{array} 
\end{equation*} 
as required.  

\subsubsection{Thom-Porteous formula}   \label{subsubsec:Thom-PorteousFormula}  
Finally, we briefly review the {\it Thom-Porteous formula} (see Porteous \cite[p.298, Proposition 1.3]{Porteous1971})  as an application of Gysin formulas.   Here we adopt the 
formulation as in  Fulton \cite[\S 14.4]{Fulton1998}, 
Fulton-Pragacz \cite[\S 2.1]{Fulton-Pragacz1998},  
Pragacz \cite{Pragacz1988}, \cite{Pragacz1996}  which is slightly different from Porteous' original one.   
Let $E  \overset{p_{E}}{\longrightarrow}  X$ 
and $F \overset{p_{F}}{\longrightarrow} X$ be complex vector bundles of ranks 
$e$ and $f$ on a variety $X$.  Let $\varphi: E \longrightarrow F$ be a vector 
bundle homomorphism.  For each point $x$, denote by $\varphi_{x}: 
E_{x} = p_{E}^{-1}(x) \longrightarrow F_{x}  =p_{F}^{-1}(x)$ the linear map 
on the fiber.   Then we set 
\begin{equation*} 
  D_{r} (\varphi) :=  \{  x  \in X \;  | \; \rank \varphi_{x} \leq r \}   \subset X, 
\end{equation*} 
which is called the $r$th {\it degeneracy locus}  of $\varphi$ ($r = 0, 1, \ldots, \mathrm{min} \, (e, f)$).   It is known that if the map $\varphi$ is {\it sufficiently generic}, 
the subvariety $D_{r}(\varphi)$ has codimension $(e-r) (f-r)$, and defines 
a cohomology class $[D_{r}(\varphi)]   \in H^{2(e-r)(f-r)} (X)$.  
Then Thom  \cite{Thom1957} 
observed that there must be a polynomial in the Chern classes of $E$ and $F$ which is 
equal to $[D_{r}(\varphi)]$.   Thom posed a problem  to find  such a polynomial, and 
later Porteous gave the answer (see  Fulton-Pragacz \cite[Chapter II, (2.1), (2.5)]{Fulton-Pragacz1998}, 
Pragacz \cite[p.414]{Pragacz1988}):    
\begin{equation}  \label{eqn:Thom-PorteousFormulaI}  
  [D_{r}(\varphi)]   =  \det \, (c_{f - r - i + j} (F  - E))_{1 \leq i, j \leq e-r}. 
\end{equation} 
Here $c(F - E)$ is defined to be $c(F)/c(E)$.  Notice that the right-hand side 
of (\ref{eqn:Thom-PorteousFormulaI}) is equal to  the {\it relative} 
version of  the Schur polynomial $s_{((e-r)^{(f-r)})}(F - E)$, where 
$((e-r)^{(f-r)})$ 
 is a {\it rectangular  partition}  with $(f-r)$ rows and $(e-r)$ columns  (for the notation, 
see Fulton-Pragacz \cite[\S 3.2]{Fulton-Pragacz1998}),  and the above formula 
becomes as follows:  
\begin{equation}   \label{eqn:Thom-PorteousFormulaII} 
          [D_{r}(\varphi)]   =  s_{((e-r)^{(f-r)})} (F - E). 
\end{equation}  
We shall give an outline of the proof of the above formula for reader's convenience: 
 Let $\pi_{F}:  G^{f-r} (F) \longrightarrow X$ be the Grassmann bundle parametrizing 
rank $(f-r)$ quotient bundles of $F$.  On $G^{f-r}(F)$, we have the tautological
exact sequence of vector bundles: 
\begin{equation*} 
  0  \longrightarrow S_{F} \hooklongrightarrow \pi_{F}^{*}(F)  \twoheadlongrightarrow Q_{F} 
\longrightarrow 0.  
\end{equation*} 
Then the vector bundle homomorphism 
$\pi_{F}^{*}(E)   \overset{\pi_{F}^{*}\varphi}{\longrightarrow}   \pi^{*}_{F}(F) \twoheadlongrightarrow Q_{F}$ over $G^{f-r}(F)$ gives a cross-section $s_{\varphi}  
\in \Gamma (\mathrm{Hom} \, (\pi^{*}_{F}(E) , Q_{F}))  \cong 
\Gamma (\pi^{*}_{F} (E)^{\vee}  \otimes Q_{F})$.  Denote by $Z(s_{\varphi}) \subset G^{f-r}(F)$
the zero locus of  $\varphi$.  Then for an element $W \in G^{f-r}(F)$ with $\pi_{F}(W) = x  \in X$,   one sees immediately that $W  \in Z(s_{\varphi})$ implies  $\Im \, \varphi_{x} \subset W$, 
and hence $\rank \, \varphi_{x} \leq \dim W = r$.  Thus we have $x \in D_{r}(\varphi)$.  
From this, the set $Z(s_{\varphi})$ maps onto 
$D_{r}(\varphi)$.  Then under appropriate conditions, the class $[Z(s_{\varphi})]$ is given by 
the top Chern class $c_{e (f-r)}(\pi^{*}_{F}(E)^{\vee} \otimes Q_{F})$.  Therefore we have 
the following formula: 
\begin{equation*} 
     \pi_{F *}  (c_{e (f-r)} (\pi^{*}_{F}(E)^{\vee} \otimes Q_{F}))  = [D_{r}(\varphi)],   
\end{equation*} 
and we have to compute the left-hand side of the above equation.  
This can be done by making use of Gysin formulas.  Let $x_{1}, \ldots, x_{f}$ 
(resp. $a_{1}, \ldots, a_{e}$)  be the Chern roots of $F$ (resp. $E$)  as in \S \ref{subsubsec:GysinFormulasFl(E)->X}. The Chern roots of $Q_{F}$ are 
$x_{1}, \ldots, x_{f-r}$. By the splitting principle, 
the top Chern class  $c_{e(f-r)} (\pi^{*}_{F}(E)^{\vee} \otimes Q_{F})$ is given by the 
product $\prod_{i=1}^{f-r}  \prod_{j=1}^{e} (x_{i} - a_{j})$. 
 On the other hand,  by a similar argument as in the previous subsection \S \ref{subsubsec:GysinFormulasFl(E)->X},  the Gysin map $\pi_{F *}:  H^{*}(G^{f-r}(F)) 
\longrightarrow H^{*}(X)$ is described by the following  symmetrizing operator
(see also Pragacz \cite[Lemma 2.5]{Pragacz1988}, \cite[Proposition 4.2]{Pragacz1996}): 
\begin{equation*} 
    \pi_{F *} (g(x_{1}, \ldots, x_{f}))
    = \sum_{\overline{w} \in S_{f}/S_{f-r} \times S_{r}} 
    w \cdot \left [  
                      \dfrac{g (x_{1}, \ldots, x_{f})}  
                                     {\prod_{1 \leq i \leq f-r} \prod_{f-r+1 \leq j \leq f}  (x_{i} - x_{j}) } 
                \right ] 
\end{equation*} 
for a polynomial $g(X_{1}, \ldots, X_{f})  \in H^{*}(X)[X_{1}, \ldots, X_{f}]^{S_{f-r} \times S_{r}}$.  
From this description, one can compute $\pi_{F *} (\prod_{i=1}^{f-r} \prod_{j=1}^{e} (x_{i} - a_{j}))$, 
and obtain the formula (\ref{eqn:Thom-PorteousFormulaII}).\footnote{
As in Ikeda-Naruse \cite[\S 5.1]{Ikeda-Naruse2009}, Molev-Sagan \cite[\S 2]{Molev-Sagan1999}, let us introduce the following notation (cf. \S \ref{subsec:UFSSPQF}): 
Set  
\begin{equation*} 
     (t|a)^{k} := \prod_{i=1}^{k} (t - a_{i})  = (t -a_{1})(t -a_{2}) \cdots (t - a_{k})
\end{equation*} 
for any integer $k \geq 0$ (Here $a = (a_{1}, \ldots, a_{e}) = (a_{1}, \ldots, a_{e}, 0, 0, \ldots)$
is the Chern roots of $E$).  
 Then one can rewrite 
\begin{equation*} 
     \prod_{i=1}^{f-r} \prod_{j=1}^{e} (x_{i} - a_{j})  = \prod_{i=1}^{f-r} (x_{i}|a)^{e}. 
\end{equation*} 
Then one computes $\pi_{F *}(\prod_{i=1}^{f-r} (x_{i}|a)^{e})$ by the 
above symmetrizing operator description of $\pi_{F *}$, and  
obtains the {\it factorial Schur polynomial} $s_{((e-r)^{(f-r)})} (\x_{f}|a)$, 
which is equal to $s_{((e-r)^{(f-r)})}  (F - E)$.  
}       We remark that 
the $K$-theoretic analogue of this formula is also given by  Buch \cite[Theorem 2.3]{Buch2002(Duke)}.   In \S \ref{subsec:Thom-PorteousFormulaComplexCobordism}, 
we shall generalize the Thom-Porteous formula  for  cohomology 
to the complex cobordism theory.

\section{Universal Hall-Littlewood functions}  \label{sec:UHLF} 

As mentioned in the introduction (see also Example \ref{ex:FGL}), 
Quillen  \cite{Quillen1969}  showed that the complex cobordism theory $MU^{*}(-)$ (with the associated formal group law $F_{MU}$)  has the following  {\it universal}   property: for any complex-oriented cohomology theory $h^{*}(-)$ (with the associated formal group law $F_{h}$), 
there exists a homomorphism of rings $\theta: MU^{*} \longrightarrow h^{*}$ 
such that  $F_{h}(X, Y) = (\theta_{*} F_{MU}) (X, Y) 
= X  + Y + \sum_{i, j \geq 1} \theta (a^{MU}_{i, j}) X^{i}Y^{j}$.  
Thus it will be sufficient to consider the case when $h = MU$, 
for general case follows immediately from the universal one 
by the specialization $a^{MU}_{i, j} \longmapsto  \theta (a^{MU}_{i, j}) \; (i, j \geq 1)$.  
Recall that, by Quillen again, the coefficient ring $MU_{*} = MU^{-*}$ is isomorphic to the 
Lazard  ring  $\L$.   
In our previous paper \cite{Nakagawa-Naruse2016}, we introduced  the {\it universal 
Schur} ($S$-) {\it functions}   $s^{\L}_{\lambda}(\x_{n})$  for $\lambda$ partitions, and  
 the {\it universal 
Schur $P$- and $Q$-functions}   $P^{\L}_{\nu} (\x_{n})$, 
$Q^{\L}_{\nu} (\x_{n})$  for $\nu$ strict partitions.  
In this section, we introduce the {\it universal Hall-Littlewood functions} 
$H^{\L}_{\lambda}(\x_{n}; t)$ which will be expected to interpolate the univerasal Schur $S$-functions 
and the universal Schur $P$-functions.  
Since these  functions   
will be of independent interest in terms of, e.g., algebraic combinatorics, 
so apart from geometry,  we shall deal with these functions  purely algebraically, 
and slightly changes the notation  concerning the formal group law  
in this section.

\subsection{Lazard ring $\L$ and the universal formal group law}    \label{subsec:UFGL}   
We begin with collecting  the basic facts about the Lazard ring. 
We use the convention  as  in Levine-Morel's book \cite{Levine-Morel2007}.  
In \cite{Lazard1955},  Lazard considered a universal commutative formal group law 
of rank one $(\L, F_{\L})$,  where the ring $\L$, called the {\it Lazard ring}, 
is isomorphic to the polynomial ring 
in countably infinite number of variables with  integer coefficients,    
and $F_{\L} = F_{\L}(u,v)$ is  the {\it universal formal group law} (for a construction 
and basic properties of $\L$, see Levine-Morel \cite[\S 1.1]{Levine-Morel2007}):
\begin{equation*} 
F_{\L} (u,v) =   u + v + \sum_{i,j \geq 1} a^{\L}_{i,j} u^{i}  v^{j}   \in    \L[[u,v]].
\end{equation*} 
This is a formal power series in $u$,  $v$ with coefficients $a^{\L}_{i,j}$ of formal variables
which satisfies the axiom of the  formal group law (see  \S \ref{subsec:COGCT}).
For the universal formal group law, we shall use the notation (see Levine-Morel \cite[\S 2.3.2]{Levine-Morel2007}) 
\begin{equation*} 
\begin{array}{llll} 
  &  u  +_{\L}    v =  F_{\L}(u,  v)   \quad  & \text{(formal sum)}, \medskip \\ 
  &  \overline{u} =   [-1]_{\L} (u)  = \chi_{_{\L}}(u)  & \text{(formal  inverse of} \;   u).    \medskip 
\end{array}
\end{equation*}
Note that  $\overline{u}\in \L[[u]]$ is a formal power series in $u$ with initial term $-u$, 
 and   first few terms appear  in  Levine-Morel  \cite[p.41]{Levine-Morel2007}.  
The $n$-series $[n]_{F_{\L}}(u) $ introduced in \S \ref{subsubsec:FGL} shall be denoted simply 
by $[n]_{\L} (u)$ in the sequel.  
In what follows, we regard $\L$ as a   graded   algebra over $\Z$, 
and the grading of $\L$ is given  by $\deg \, (a^{\L}_{i, j}) =  1 - i - j  \; (i, j \geq 1)$ 
(see Levine-Morel \cite[p.5]{Levine-Morel2007}).  
Be aware that in topology, it is customary to give   $a^{\L}_{i, j}$ the {\it cohomological} 
degree $2(1 - i - j)$.  

\subsection{Universal factorial Schur $S$-, $P$-, and $Q$-functions}   \label{subsec:UFSSPQF}  
In this subsection, we recall the definitions of the {\it universal factorial 
Schur $S$-, $P$-, and $Q$-functions}  following  Nakagawa-Naruse \cite[\S 4]{Nakagawa-Naruse2016}.  
Besides the variables $\x = (x_{1}, x_{2}, \ldots)$, we 
prepare another set of variables $\b = (b_{1}, b_{2}, \ldots)$.  
We provide the variables
$\x=(x_{1},  x_{2},  \ldots)$ and $\b = (b_{1},  b_{2},  \ldots)$ with degree
$\deg  \,  (x_{i}) = \deg \,   (b_{i})= 1 $ for $i = 1, 2, \ldots$.    
In what follows, when considering polynomials or formal power series 
$f(x_{1}, x_{2}, \ldots)$ with coefficients in $\L$ (or $\L[[\b]]$), 
we shall call the degree with respect to $x_{1}, x_{2}, \ldots, b_{1}, b_{2}, \ldots$, and 
$a_{i, j}^{\L}  \; (i, j \geq 1)$  the {\it total} degree of $f(x_{1}, x_{2}, \ldots)$.  

For an integer $k\geq 1$, we define a generalization of the 
ordinary $k$-th power $t^{k}$  
 by 
\begin{equation*} 
   [t| \b]^{k}  := \displaystyle{\prod_{i=1}^{k}}  (t +_{\L}   b_{i})  
               = (t +_{\L}   b_{1})(t  +_{\L}   b_{2}) \cdots (t +_{\L}   b_{k}) 
\end{equation*} 
and its variant by 
\begin{equation*} 
   [[t| \b  ]]^{k}  :=(t +_{\L}  t)[t| \b ]^{k-1}  
  = (t +_{\L}   t) (t  +_{\L}   b_{1})(t +_{\L}   b_{2}) \cdots (t +_{\L}   b_{k-1}),
\end{equation*}  
where we  set $[t| \b ]^{0} =   [[t| \b ]]^{0} := 1$.
For a partition,\footnote{
For notation and terminology on partitions which will be used throughout this paper, 
we mainly follow those in Macdonald's book \cite[Chapter I]{Macdonald1995}. 
} i.e., a non-increasing sequence of non-negative integers 
 $\lambda=(\lambda_1,\ldots,\lambda_r) \;
 (\lambda_{1} \geq \lambda_{2} \geq \cdots \geq \lambda_{r} \geq 0)$,\footnote{
It is customary not to distinguish two partitions which differ only by a string of 
zeros at the end.  
} 	
we set
\begin{equation*} 
  [\x| \b ]^{\lambda}  :=  \displaystyle{\prod_{i=1}^{r}}  [x_i|\b]^{\lambda_i}  \quad 
\text{and}  \quad 
[[\x| \b ]]^{\lambda}  :=  \displaystyle{\prod_{i=1}^{r}} [[x_i|\b]]^{\lambda_i}.
\end{equation*} 
Let $\mathcal{P}_{n}$ denote  the set of all partitions of length $\leq n$. 
For a positive integer $n$,   we set $\rho_{n}=(n, n-1, \ldots, 2, 1)$.
For partitions $\lambda, \mu \in \mathcal{P}_{n}$, $\lambda + \mu$ 
is a partition  of length $\leq n$ defined by $(\lambda + \mu)_{i} := \lambda_{i} + \mu_{i} \; (1 \leq i \leq n)$.  With this notation, the {\it universal factorial Schur $(S$-$)$ function}
$s^{\L}_{\lambda}(\x_{n}|\b)  = s^{\L}_{\lambda}(x_{1}, \ldots, x_{n}| \b)$  
corresponding to a partition $\lambda = (\lambda_{1}, \ldots, \lambda_{n})  \in \mathcal{P}_{n}$ 
is defined to be 
\begin{equation}   \label{eqn:Definitions^L(x_n|b)}  
s^{\L}_{\lambda} (\x_{n}|\b) = s^{\L}_{\lambda} (x_{1},  \ldots, x_{n}|\b):=
\displaystyle
\sum_{w  \in S_{n}} w  \cdot \left[ \frac{[\x |\b]^{\lambda + \rho_{n-1}}}
                                                {\prod_{1\leq i<j\leq n}(x_{i}  +_{\L}   \overline{x}_{j})}
                                       \right]. 
\end{equation} 
We also define 
\begin{equation}   \label{eqn:Definitions^L(x_n)}   
s^{\L}_{\lambda} (\x_{n}) 
= s^{\L}_{\lambda} (x_{1}, \ldots, x_{n}) :=s^{\L}_{\lambda} (x_{1}, \ldots, x_{n}|\bm{0}). 
\end{equation} 
The function $s^{\L}_{\lambda}(\x_{n})$ will be  called just the {\it universal Schur function}.  
\begin{rem}   \label{rem:Fel'dman'sGysinFormula} 
The non-equivariant version  $s^{\L}_{\lambda}(\x_{n})$  is  already defined by Fel'dman 
\cite[Definition 4.2]{Fel'dman2003}.  These are called the {\it generalized Schur polynomials} there.  In that paper,  the author also established a Gysin formula for these generalized Schur 
polynomials $($see \cite[Theorem 4.5]{Fel'dman2003}$)$.\footnote{
Fel'dman's  formula is a  generalization of  Fulton-Pragacz' formula  \cite[(4.2)]{Fulton-Pragacz1998} (see also (\ref{eqn:GysinFormulas_lambda(Q)s_mu(S)}) in this paper) 
to the complex cobordism theory  as well as to general partial flag bundles .  
However, we think that his formula should be modified correctly. 
In the later section \S \ref{subsubsec:Fel'dman'sGysinFormula}, we shall establish a Gysin formula for the universal 
Schur functions, thus correcting Fel'dman's formula.  
}  
\end{rem}  
\noindent
Since 
\begin{equation*} 
  x_{i}  +_{\L}  \overline{x}_{j}=  (x_{i} -x_{j})(1+\text{higher degree terms in } x_{i} \text{ and } x_{j}
\text{ with coefficients in }\L),  
\end{equation*}
the function $s^{\L}_{\lambda}(\x_{n}|\b)$ is a  formal power series with coefficients in $\L$ 
in the variables $x_{1}, \ldots, x_{n}$ 
and  $b_{1}, b_{2}, \ldots, b_{\lambda_{1} + n-1}$. 
It is also a  homogeneous  formal power series  
of  total     degree  $|\lambda|  = \sum_{i=1}^{n} \lambda_{i}$, the size of $\lambda$.  
In (\ref{eqn:Definitions^L(x_n|b)}), 
if we put $a^{\L}_{i, j} = 0$ for all $i, j \geq 1$ and $b_{i} = -a_{i} \; (i = 1, 2, \ldots)$, 
where $a = (a_{1}, a_{2}, \ldots, )$ is another sequence of parameters, 
 the functions $s^{\L}_{\lambda}(\x_{n}|\b)$ 
reduce to the  factorial Schur polynomials  usually denoted by 
$s_{\lambda}(\x_{n}|a)$ (for its definition, see  Ikeda-Naruse \cite[\S 5.1]{Ikeda-Naruse2009}, Macdonald \cite[I,  \S 3,  Examples 20]{Macdonald1995},  Molev-Sagan \cite[p.4431]{Molev-Sagan1999}).  
If we put $a^{\L}_{1, 1} = \beta$  and  $a^{\L}_{i, j} = 0$ for all $(i, j) \neq (1, 1)$,\footnote{
Notice that the sign convention of $\beta$ is opposite from the one given in Example 
\ref{ex:FGL}.   In the rest of this paper, we shall use this sign convention that fits in 
with the listed references here.  
}  
then $s^{\L}_{\lambda}(\x_{n}|\b)$ reduce to the factorial Grothendieck 
polynomials $G_{\lambda}(\x_{n}|\b)$ (for its definition, see Ikeda-Naruse \cite[(2.12), (2.13)]{Ikeda-Naruse2013},   McNamara \cite[Definition 4.1]{McNamara2006}).  
Thus our functions $s^{\L}_{\lambda}(\x_{n}|\b)$  are  generalizations of these polynomials 
and hence {\it universal} in this sense.  
 Note that unlike the usual factorial Schur and  Grothendieck 
polynomials,  the function
$s_{\emptyset}^{\L}(\x_{n}|\b)$ corresponding to the empty partition $\emptyset 
= (0^{n})$ is not equal to $1$.  For instance, we have 
\begin{equation*} 
      s^{\L}_{\emptyset} (\x_{2}|\b) =   \dfrac{x_{1} +_{\L} b_{1}}{x_{1} +_{\L} \overline{x}_{2}}  
                                               +  \dfrac{x_{2} +_{\L} b_{1}}{x_{2} +_{\L} \overline{x}_{1}}
         =  1 + a^{\L}_{1, 2}x_{1}x_{2} + a^{\L}_{1, 1} a^{\L}_{1, 2}b_{1} x_{1}x_{2} + \cdots 
        \neq 1.  
\end{equation*} 

As mentioned in Remark \ref{rem:Fel'dman'sGysinFormula},  in order to formulate 
Fel'dman's Gysin formula for the universal Schur functions,  
we need to extend the above definition to arbitrary 
sequences  
of non-negative integers.  
Thus,  for a sequence $I = (I_{1}, I_{2}, \ldots, I_{n})$ of non-negative integers, 
we define $s^{\L}_{I}(\x_{n})$ to be 
\begin{equation} 
     s^{\L}_{I} (\x_{n})  :=  
\sum_{w  \in S_{n}}    w  \cdot \left[ 
                                                \frac{    {\x}^{I + \rho_{n-1}}}
                                                {\prod_{1\leq i<j\leq n}(x_{i}  +_{\L}   \overline{x}_{j})}
                                       \right]. 
\end{equation}

Next let us recall the definition of the universal factorial Schur $P$- and $Q$-functions. 
Denote by $\mathcal{SP}_{n}$ the set of 
all strict partitions of length $\leq n$, i.e., a  sequence of positive integers 
$\nu = (\nu_{1}, \ldots, \nu_{k})$ with $k \leq n$ 
 such that $\nu_{1}  > \cdots > \nu_{k} > 0$. Then, 
for a strict partition $\nu =(\nu_{1},  \ldots, \nu_{k})  \in \mathcal{SP}_{n}$,  
the {\it unviersal factorial Schur $P$- and $Q$-functions} are  defined to be  
\begin{equation}    \label{eqn:DefinitionP^L(x_n|b)Q^L(x_n|b)}  
\begin{array}{rlll} 
& P^{\L}_{\nu}(\x_{n}|\b)  =  P^{\L}_{\nu} (x_{1}, \ldots, x_{n} |\b) 
:=  
\displaystyle\frac{1}{(n-k)!}
\sum_{w  \in  S_{n}} w \cdot 
  \left[
           [\x | \b]^{\nu}
            \prod_{i=1}^{k}  
            \prod_{j=i+1}^{n}   \frac{x_{i}  +_{\L}  x_{j}}{x_{i} +_{\L}    \overline{x}_{j}} 
   \right],  \medskip \\
& Q^{\L}_{\nu}(\x_{n}|\b) =   Q^{\L}_{\nu} (x_{1}, \ldots, x_{n} |\b)
:=    
\displaystyle\frac{1}{(n-k)!}
\sum_{w  \in  S_{n}}w  \cdot 
  \left[
           [[\x|\b]]^{\nu}
               \prod_{i=1}^{k} 
               \prod_{j=i+1}^n \frac{x_{i} +_{\L}   x_{j}}{x_{i} +_{\L}  \overline{x}_{j}}
  \right],    \medskip 
\end{array} 
\end{equation} 
where the symmetric group $S_{n}$ acts only on the $x$-variables $x_{1}, \ldots, x_{n}$
by permutations.   

We also define 
\begin{equation*} 
\begin{array}{rllll} 
P^{\L}_{\nu}(\x_{n}) & = P^{\L}_{\nu} (x_{1}, \ldots, x_{n}) 
:= P^{\L}_{\nu}  (x_{1},  \ldots, x_{n}| \bm{0}), \medskip \\
Q^{\L}_{\nu}(\x_{n}) & = Q^{\L}_{\nu} (x_{1},  \ldots, x_{n}) 
:= Q^{\L}_{\nu} (x_{1}, \ldots, x_{n}| \bm{0}).  \medskip 
\end{array} 
\end{equation*} 
\noindent 
The functions  $P^{\L}_{\nu} (\x_{n}|\b)$ and  $Q^{\L}_{\nu} (\x_{n}|\b)$
are   formal power series with coefficients in $\L$ 
in the variables $x_{1},  \ldots, x_{n}$ and    $b_{1}, b_{2}, \ldots, b_{\nu_{1}}$ for $P^{\L}_{\nu}$ 
(resp. $b_{1}, b_{2} \ldots, b_{\nu_{1} - 1}$ for $Q^{\L}_{\nu}$).  
These are  homogeneous formal power series of 
total    degree  $|\nu|$.   
In  (\ref{eqn:DefinitionP^L(x_n|b)Q^L(x_n|b)}), 
if we put $a^{\L}_{i, j} = 0$ for all $i, j \geq 1$ and $b_{i} = -a_{i} \; (i = 1, 2, \ldots)$, 
the functions 
$P^{\L}_{\nu} (\x_{n}|\b)$, $Q^{\L}_{\nu}(\x_{n}|\b)$ reduce to 
the  {\it factorial Schur 
$P$- and $Q$-polynomials}   $P_{\nu}(\x_{n}|a)$, $Q_{\nu}(\x_{n}|a)$ 
(for their definitions, see   Ikeda-Mihalcea-Naruse \cite[\S 4.2]{IMN2011}), 
Ikeda-Naruse \cite[Definition 8.1]{Ikeda-Naruse2009},   Ivanov \cite[Definitions 2.10 and 2.13]{Ivanov2004}).  
If we put $a^{\L}_{1, 1} = \beta$ and $a^{\L}_{i, j}  = 0$ for all $(i, j) \neq (1, 1)$, then $P^{\L}_{\nu}(\x_{n}|\b)$, $Q^{\L}_{\nu}(\x_{n}|\b)$ reduce to the {\it $K$-theoretic factorial Schur 
$P$- and $Q$-polynomials}  $GP_{\nu}(\x_{n}|\b)$, $GQ_{\nu}(\x_{n}|\b)$ 
due to Ikeda-Naruse  \cite[Definition 2.1]{Ikeda-Naruse2013}.  
Thus our functions $P^{\L}_{\nu}(\x_{n}|\b)$, $Q^{\L}_{\nu}(\x_{n}|\b)$
are  generalizations of these polynomials  and hence {\it universal} in this sense.  

\subsection{Universal Hall-Littlewood functions}     \label{subsec:UHLF}   
In this subsection, we introduce the {\it universal Hall-Littlewood functions} which 
interpolate the $S$-functions and the $P$-functions.  We use the notation 
as in Pragacz \cite{Pragacz2015}.  
Let  $\lambda = (\lambda_{1},  \ldots, \lambda_{n})  \in \mathcal{P}_{n}$ be a
partition of  length $\ell (\lambda)  \leq n$.    
Consider the maximal ``intervals'' $I_{1}, I_{2}, \ldots, I_{d}$  in 
$[n] := \{ 1, 2, \ldots, n \}$, where the sequence $\lambda$ is ``constant''.  
Thus we have  a decomposition 
\begin{equation*} 
   [n]  = I_{1} \sqcup I_{2} \sqcup \cdots \sqcup I_{d} \quad  (\text{disjoint union}). 
\end{equation*}  
Here and in what follows, we keep the following convention: 
when we refer to such a decomposition, we always arrange ``intervals''  $I_{1}, I_{2}, \ldots$ 
in increasing order, 
that is, $\mathrm{max} \,  I_{r}  <  \mathrm{min}  \, I_{r+1}$ for each $r$.  
Let  $m_{r}$ be the  ``length'' of the interval $I_{r}$ for $r = 1, 2, \ldots, d$, 
namely the cardinality of $I_{r}$ so that $\sum_{ r =1}^{d} m_{r} = n$.  
We write  $n (i)$ the  number of the interval containing $i$  for $i \in [n]$, namely if $i $ is in $I_{r}$, then $n (i) = r$.  
Notice that, since $\lambda  = (\lambda_{1}, \ldots, \lambda_{n})$ is a partition, 
$n(i) < n(j)$ is equivalent to $\lambda_{i} > \lambda_{j}$ for $i, j  \in [n]$.  
We define a subgroup $S_{n}^{\lambda}$  of $S_{n}$ as the  stabillizer of $\lambda$. 
Thus 
\begin{equation*} 
    S_{n}^{\lambda} =  \prod_{i=1}^{d} S_{m_{i}} =   S_{m_{1}} \times S_{m_{2}} \times \cdots \times S_{m_{d}}.  
\end{equation*}    
Denote by  $\ell_{\L}  (x)   \in \L[[x]]$ the  {\it logarithm}\footnote{
See e.g.,   Kono-Tamaki \cite[Lemma 6.27]{Kono-Tamaki2006}, 
Levine-Morel \cite[Lemma 4.1.29]{Levine-Morel2007}, 
 Quillen \cite{Quillen1969},  Ravenel \cite[Appendix A2]{Ravenel2004}.  
If we put $a^{\L}_{i, j} = 0$ for all $i, j \geq 1$,  then  $F_{\L}(u, v)$ reduces to 
the additive formal group law $F_{a}(u, v) = u + v$, 
 and both $\ell_{\L} (x)$ and $\ell_{\L}^{-1}(x)$  reduce to $x$. 
If we put $a^{\L}_{1, 1} = \beta$ and $a^{\L}_{i, j} = 0$ for all $(i, j) \neq (1,1)$, 
then $F_{\L}(u, v)$ reduces to the multiplicative formal group law 
$F_{m}(u, v) = u + v + \beta uv$.    In this case, 
 $\ell_{\L}(x)$ reduces to $\beta^{-1} \log \, (1 + \beta x) 
= \sum_{i=0}^{\infty}  \dfrac{(-\beta)^{i}}{i + 1}  x^{i + 1}$, 
and $\ell_{\L}^{-1}(x)$ reduces to $\beta^{-1} (e^{\beta x} - 1) 
= \sum_{i=0}^{\infty}  \dfrac{\beta^{i}}{(i + 1)!}  x^{i + 1}$.  
} of  $F_{\L}$, i.e.,  
a unique formal power series with leading term $x$ 
such that   
\begin{equation*} 
\ell_{\L}(a +_{\L} b)  =   \ell_{\L} (a) + \ell_{\L} (b).   
\end{equation*} 
Using the logarithm $\ell_{\L}(x)$, one can rewrite the $n$-series $[n]_{\L}(x)$ for a non-negative integer $n$, aforementioned in \S \ref{subsec:UFGL},    as  
\begin{equation*}
   \ell_{\L} ([n]_{\L}(x))  =   \ell_{\L}(\underbrace{x +_{\L}  + \cdots +_{\L} x}_{n}) 
   =  \underbrace{\ell_{\L}(x) + \cdots + \ell_{\L}(x)}_{n} = n \cdot \ell_{\L}(x), 
\end{equation*} 
in other words, $[n]_{\L} (x)  = \ell_{\L}^{-1}  (n \cdot \ell_{\L} (x))$.  
This formula allows us to define 
\begin{equation*} 
   [t]_{\L}(x) :=  \ell_{\L}^{-1} (t  \cdot \ell_{\L} (x))
\end{equation*} 
 for an  indeterminate $t$.  This is a natural extension of  $t \cdot x$ as well as 
the $n$-series  $[n]_{\L} (x)$.\footnote{
$[t]_{\L} (x)$ reduces to  just $t \cdot x$ for the  additive formal group law
 $F_{a}(u, v) = u + v$, 
and $\beta^{-1} \{ (1 + \beta x)^{t} - 1 \}  =  \sum_{i=1}^{\infty}  \dfrac{(t)_{i}} {i!}   \beta^{i-1} x^{i}$ for  the multiplicative formal group law $F_{m}(u, v) = u + v + \beta uv$. 
Here  $(t)_{i}$ means $t(t-1)(t-2) \cdots (t-i+1)$.   
}  
\begin{defn}  [Universal Hall-Littlewood function]   \label{df:DefinitionUHLF}  
With the above notation, for a partition  $\lambda
\in \mathcal{P}_{n}$, 
we define 
\begin{equation}   \label{eqn:DefinitionUHLF}    
  H^{\L}_{\lambda}(\x_{n}; t) 
  :=  \sum_{\overline{w}  \in S_{n}/S^{\lambda}_{n}}  
           w \cdot \left [   \x^{\lambda}  
                                  \prod_{
                                                   1 \leq i < j \leq n, \; 
                                                               n (i) <    n (j)  
                                           }  
                                    \dfrac{x_{i} +_{\L}  [t]_{\L} (\overline{x}_{j}) } {x_{i} +_{\L}  \overline{x}_{j}}  
                       \right ].   
\end{equation}
\end{defn}  
\noindent
If we put $a^{\L}_{i, j} = 0$ for all $i, j \geq 1$,  the functions $H^{\L}_{\lambda}(\x_{n}; t)$ 
reduce to the usual Hall-Littlewood polynomials denoted by $P_{\lambda}(x_{1}, \ldots, x_{n}; t)$ 
in Macdonald's book \cite[Chapter III, \S 2, (2.2)]{Macdonald1995}.  
For the usual Hall-Littlewood polynomial   $P_{\lambda}(x_{1}, \ldots, x_{n}; t)$, 
it is known  that 
\begin{equation*} 
   P_{\lambda}(x_{1}, \ldots, x_{n}; 0)  = s_{\lambda}(x_{1}, \ldots, x_{n})
\end{equation*} 
under the specialization $t = 0$ (see Macdonald \cite[Chapter III, \S 2, (2.3)]{Macdonald1995}), 
 and 
\begin{equation*} 
    P_{\lambda} (x_{1}, \ldots, x_{n}; 1) = m_{\lambda} (x_{1}, \ldots, x_{n}),  
\end{equation*} 
under the specialization $t = 1$ (see Macdonald \cite[Chapter III, \S 2, (2.4)]{Macdonald1995}).  
Here $m_{\lambda}$ denotes the monomial symmetric polynomial corresponding to 
$\lambda$.  
Moreover,  for a strict partition $\nu$ of length 
$\ell (\nu)  \leq n$,  one obtains that 
\begin{equation}   \label{eqn:Hall-LittlewoodSchurP} 
   P_{\nu}(x_{1}, \ldots, x_{n}; -1)  =  P_{\nu}(x_{1}, \ldots, x_{n})
\end{equation} 
under the specialization $t = -1$  (see Macdonald \cite[Chapter III, \S 8, Examples 1.]{Macdonald1995}).  

For the universal Hall-Littlewood functions $H^{\L}_{\lambda}(\x_{n}; t)$,  it follows 
immediately from (\ref{eqn:DefinitionUHLF}) that $H^{\L}_{\lambda}(\x_{n}; 1)  = m_{\lambda}(\x_{n})$
under the specialization $t = 1$.  Let us next consider the specialization $t = -1$. 
Let $\nu = (\nu_{1}, \ldots,  \nu_{k})  \in \mathcal{S P}_{n}$  be a strict partition with length $\ell (\nu) = k  \leq n$.  Then we have a decomposition 
$[n] =  I_{1} \sqcup \cdots \sqcup I_{k}  \sqcup I_{k+1}$, where 
$I_{r}  = \{ r \} \; (r = 1, \ldots, k)$ and $I_{k+1}  = \{ k+1, \ldots, n \}$.  
Therefore we have  
  $m_{r} = 1 \; (r = 1, \ldots, k)$,  $m_{k+1}    = n-k$ and     
  $n(i)  =  i \; (i = 1, \ldots, k)$, $n(i)   = k+1 \; (i = k+1, \ldots, n)$.   
The stabilizer  of $\nu$ is given by $S^{\nu}_{n}  = (S_{1})^{k}  \times S_{n-k}$.   
Therefore  it follows from Definition  \ref{df:DefinitionUHLF} and (\ref{eqn:DefinitionP^L(x_n|b)Q^L(x_n|b)}), we have 
\begin{equation*} 
\begin{array}{llll} 
    H^{\L}_{\nu}(\x_{n}; -1)  & =  \displaystyle{
                                                   \sum_{\overline{w}  \in S_{n}/(S_{1})^{k} \times S_{n-k}}
                                                             }   
                                         w \cdot \left [ 
                                                           \x^{\nu}  \prod_{
                                                                                        1 \leq i < j \leq n,  \;  
                                                                                          1 \leq  i \leq k}
                                                                        \dfrac{x_{i}  +_{\L}  x_{j} } {x_{i} +_{\L}  \overline{x}_{j}}     
                                                    \right ]    \medskip \\
                                     & =   \dfrac{1}{(n-k)!}  \displaystyle{\sum_{w  \in S_{n}}}  
                                         w  \cdot \left [ 
                                                           \x^{\nu}  \prod_{i = 1}^{k}  \prod_{j = i + 1}^{n}  \dfrac{x_{i} +_{\L} x_{j}} {x_{i} +_{\L} \overline{x}_{j}}   
                                                      \right ]    = P^{\L}_{\nu}(\x_{n}).    \medskip 
\end{array} 
\end{equation*} 

We now consider the specialization $t = 0$.  Let 
$\lambda = (\lambda_{1}, \ldots, \lambda_{n})  \in \mathcal{P}_{n}$ be a partition 
with length $\ell (\lambda) \leq n$. 
Then we have  a decomposition $[n] = I_{1} \sqcup I_{2} \sqcup \cdots \sqcup I_{d}$ 
as above.   Letting  $m_{r}$ be the cardinality of $I_{r}$ for $r = 1, \ldots, d$,  
one can rewrite  $\lambda = (n_{1}^{m_{1}} \; n_{2}^{m_{2}} \;  \ldots \;  n_{d}^{m_{d}}) 
\;  (n_{1} > n_{2} > \cdots > n_{d}  \geq  0)$.  We put 
$\nu (r) := m_{1} + \cdots + m_{r}$ for $r = 1, 2, \ldots, d$ and 
$\nu (0) := 0$.  
Then the specialization $t  = 0$ gives 
\begin{equation}   \label{eqn:DefinitionNUSF}   
\begin{array}{llll} 
  H^{\L}_{\lambda}(\x_{n}; 0) 
    & =  \displaystyle{
                                \sum_{\overline{w}  \in S_{n}/S^{\lambda}_{n}}
                           }   
           w \cdot \left [   \x^{\lambda}  
                                  \prod_{
                                                             1 \leq i < j \leq n,  \; 
                                                              n (i) <    n (j)  
                                           }    
                                    \dfrac{x_{i}  } {x_{i} +_{\L}  \overline{x}_{j}}  
                       \right ]       \medskip \\
       &  =  \displaystyle{
                                    \sum_{\overline{w}  \in S_{n}/S^{\lambda}_{n}}
                              }  
           w \cdot \left [   \prod_{i=1}^{n} x_{i}^{\lambda_{i}}  
                                  \prod_{
                                                            1 \leq i < j \leq n, \; 
                                                                n(i) < n(j)    
                                           }   
                                    \dfrac{x_{i}  } {x_{i} +_{\L}  \overline{x}_{j}}  
                       \right  ]     \medskip   \\
    & =   \displaystyle{
                                    \sum_{\overline{w}  \in S_{n}/S^{\lambda}_{n}}
                              }  
           w \cdot \left [     
                                 \dfrac{
                                          \prod_{r=1}^{d}   \left ( 
                                                                              \prod_{m_{1} + \cdots + m_{r-1} 
                                                                                           < i \leq m_{1} + \cdots + m_{r}                 
                                                                                      } 
                                                                               x_{i}^{n_{r} + n - (m_{1}  + \cdots + m_{r})}  
                                                                 \right )
                                           }
                                            {
                                                \prod_{
                                                                           1 \leq i < j \leq n, \; 
                                                                                 n(i) < n(j)      
                                                         }
                                               (x_{i} +_{\L}  \overline{x}_{j}) 
                                            }  
                       \right  ]     \medskip   \\
   & =   \displaystyle{
                                    \sum_{\overline{w}  \in S_{n}/S^{\lambda}_{n}}
                              }  
           w \cdot \left [     
                                 \dfrac{
                                          \prod_{r=1}^{d}   \left ( 
                                                                              \prod_{\nu (r-1)    
                                                                                           < i \leq \nu (r)                
                                                                                      } 
                                                                               x_{i}^{n_{r} + n - \nu (r)}  
                                                                 \right )
                                           }
                                            {
                                                \prod_{
                                                                           1 \leq i < j \leq n,  \;  
                                                                                     n(i) < n(j)      
                                                         }
                                               (x_{i} +_{\L}  \overline{x}_{j}) 
                                            }  
                       \right  ].    \medskip   \\
\end{array}  
\end{equation} 
We shall consider  this function in later section \S \ref{sec:NUFSF}.

\section{Applications of Gysin formulas to the Schur functions}
In \S \ref{subsec:VariousGysinFormulas}, we reviewed various  Gysin formulas 
in the ordinary cohomology (or Chow) theory, 
and in the previous section \S \ref{sec:UHLF}, we introduced  {\it universal} 
analogues of  the ordinary Schur  $S$-, $P$-, $Q$-, and Hall-Littlewood functions.  
In this section, we pursue  the Gysin formulas in the complex cobordism theory  
which relate Gysin maps for flag bundles to  these universal Schur functions.  
The following \S \ref{subsec:GysinFormulasVariousSchurFunctions}  is devoted to 
the recollection of various known Gysin formulas in the ordinary cohomology 
theory, and their generalization to the complex cobordism theory 
will be treated in  \S \ref{subsec:UGFUSF}.   Our main tool for establishing 
Gysin formulas in the complex cobordism theory is the Bressler-Evens 
formula reviewed in \S \ref{subsec:Bressler-EvensFormula}.

\subsection{Gysin formulas for various Schur functions}     \label{subsec:GysinFormulasVariousSchurFunctions}   
We use the same notation as in \S \ref{subsubsec:GysinFormulasFl(E)->X}.  
Let $E  \overset{p}{\longrightarrow}   X$ be a  complex vector bundle of rank $n$ (over a variety).  
Denote by    $\tau = \tau_{E}: \mathcal{F} \ell (E) \longrightarrow X$  the associated flag bundle  parametrizing successive flags of quotient bundles  of $E$ of rank $n-1, \ldots, 2, 1$.  
The usual Schur polynomial 
$s_{\lambda}(X_{1}, \ldots, X_{n})$ 
corresponding to  a partition $\lambda \in \mathcal{P}_{n}$ is a symmetric polynomial 
in the $n$-variables $X_{1}, \ldots, X_{n}$, and therefore it can be written 
as a polynomial  in the elementary symmetric polynomials $e_{i}(X_{1}, \ldots, X_{n})$'s.  
Let $x_{1}, \ldots, x_{n}$  be the  Chern roots of $E$  as in \S \ref{subsubsec:GysinFormulasFl(E)->X}.  Then the Chern classes $c_{i}(E)$'s can be identified with 
$e_{i}(x_{1}, \ldots, x_{n})$'s as usual, and hence  $s_{\lambda}(x_{1}, \ldots, x_{n})$ 
can be expressed  as a polynomial  in  $c_{i}(E)$'s.  Let us define a cohomology class 
$s_{\lambda}(E)  \in H^{2 |\lambda|} (X)$ to be 
$\tau^{*}(s_{\lambda}(E)) := s_{\lambda}(x_{1}, \ldots, x_{n})  \in H^{2|\lambda|}(\F \ell (E))$.\footnote{
It is well-known that the induced homomorphism $\tau^{*}:  H^{*}(X) \hooklongrightarrow H^{*}(\F \ell (E))$ is injective, and hence the cohomology class $s_{\lambda}(E)$ 
is well-defined.   
}
Then  the following formula is known: 
 \begin{prop} [Pragacz \cite{Pragacz1988}, Lemma 2.3;  Fulton-Pragacz \cite{Fulton-Pragacz1998},  (4.1)]    \label{prop:Jacobi-TrudiIdentity}   
      The image of the monomial  
     ${\bf x}^{\lambda + \rho_{n-1}} = x_{1}^{\lambda_{1}+n-1} x_{2}^{\lambda_{2} + n-2} \cdots x_{n}^{\lambda_{n}}$
  under the Gysin homomorphism $\tau_{*}: H^{*}(\mathcal{F}  \ell (E))  
         \longrightarrow H^{*}(X)$ is given by
\begin{equation}  \label{eqn:Jacobi-TrudiIdentityII} 
       \tau_{*}({\bf x}^{\lambda + \rho_{n-1}} )
        =  s_{\lambda} (E).        
\end{equation}  
\end{prop}  
\noindent
The formula (\ref{eqn:Jacobi-TrudiIdentityII}) is called the {\it Jacobi-Trudi identity}, from which 
 Fulton-Pragacz \cite{Fulton-Pragacz1998} derived some useful formulas for Grassmann bundles, 
which we recall a bit later.

Furthermore, the analogous Gysin formulas which relate the Hall-Littlewood polynomials and more general flag bundles are considered in  Pragacz \cite{Pragacz2015}. 
Let us recall these  formulas. 
We use the same notation as in \S \ref{subsec:UHLF} (see also Pragacz \cite{Pragacz2015}).  
Let   $\lambda = (\lambda_{1}, \ldots, \lambda_{n})  \in \mathcal{P}_{n}$ 
be a partition  of  length $\leq n$.  Then we have a decomposition  
\begin{equation*} 
    [n] = \{ 1, 2, \ldots, n \} = I_{1} \sqcup I_{2} \sqcup \cdots \sqcup I_{d},
\end{equation*} 
 where   $\lambda$ is  ``constant'' on each $I_{r} \; (r = 1, 2, \ldots,  d)$. 
Denote by  $m_{r}$ the length of the interval $I_{r}$ for $r = 1, 2, \ldots, d$ so that $\sum_{r =1}^{d} m_{r} = n$, and 
$n(i)$  the  number of the interval containing $i$  for $i  \in [n]$.  
We put $\nu (p) = \sum_{r=1}^{p} m_{r}$ for $p = 1, \ldots, d$ and $\nu (0) = 0$. 
 The stabilizer of $\lambda$ is denoted by $S_{n}^{\lambda}$.  
Associated to a complex vector bundle 
$E \longrightarrow X$, one can define a ``$(d-1)$-step flag bundle'' with steps of lengths $m_{r}$
   \begin{equation*} 
      \eta_{\lambda}:  \mathcal{F} \ell^{\lambda}  (E) \longrightarrow X,   
  \end{equation*} 
parametrizing flags of   quotient bundles   of $E$ of ranks 
\begin{equation*} 
   n-m_{d} = \nu (d-1),   n-m_{d} - m_{d-1} = \nu (d-2),   \ldots,  
  n-m_{d} - m_{d-1} - \cdots - m_{2} = \nu (1).   
\end{equation*} 
If $\lambda = \emptyset$, the empty partition, then $\F \ell^{\emptyset} (E)$ is 
understood to be the base space $X$.  
Here, for later discussion,  we shall fix the notation about  {\it partial flag bundles} 
associated to a complex vector bundle $E  \overset{p}{\longrightarrow}  X$ of rank $n$ (see Fulton's book 
\cite[\S 9.1, 10.6]{Fulton1997}):  For a sequence of integers $0 < r_{1} < r_{2} < \cdots < r_{k} < n$, 
let us denote by $\F \ell_{r_{1}, r_{2}, \ldots, r_{k}}(E)$ a
 partial flag bundle consisting of flags of 
subbundles $0 \subsetneq S_{1} \subsetneq S_{2} \subsetneq \cdots \subsetneq S_{k} \subsetneq  E$ with $\rank S_{i} = r_{i} \; (i = 1, \ldots, k)$.  
 Since giving a flag of subbundles of $E$  as above 
is equivalent to giving a flag of quotient bundles $E \twoheadrightarrow Q^{k} \twoheadrightarrow Q^{k-1} \twoheadrightarrow \cdots \twoheadrightarrow Q^{1} \twoheadrightarrow 0$ with
$\rank Q^{i} =  n - r_{k + 1 -i}  \; (i = 1, \ldots, k)$, this partial flag bundle is also denoted by 
$\F \ell^{n-r_{1}, n-r_{2}, \ldots, n-r_{k}}(E)$.    Moreover,   dualizing  the sequence of 
vector bundles 
\begin{equation*} 
    0      \hookrightarrow S_{1} \hookrightarrow  S_{2}  \hookrightarrow 
  \cdots \hookrightarrow     S_{k}  \hookrightarrow E  
       \twoheadrightarrow Q^{k}     
       \twoheadrightarrow  Q^{k-1} 
       \twoheadrightarrow  \cdots  \twoheadrightarrow  
                                      Q^{1}  \twoheadrightarrow 0 
\end{equation*} 
gives  the sequence  of vector bundles
\begin{equation*} 
\begin{array}{llll}   
     0      \hookrightarrow (Q^{1})^{\vee}    \hookrightarrow  (Q^{2})^{\vee}  
           \hookrightarrow 
  \cdots \hookrightarrow     (Q^{k})^{\vee}   &  \hookrightarrow   E^{\vee}      \medskip \\
       &   \twoheadrightarrow    (S_{k})^{\vee}      
       \twoheadrightarrow   (S_{k-1})^{\vee}   
       \twoheadrightarrow  \cdots  \twoheadrightarrow  
                                      (S_{1})^{\vee}  \twoheadrightarrow 0.      \medskip 
\end{array}   
\end{equation*}
Therefore we have the canonical isomorphism 
\begin{equation*} 
     \F \ell_{r_{1}, r_{2}, \ldots, r_{k}}(E) = \F \ell^{n-r_{1}, n-r_{2}, \ldots, n-r_{k}}(E)
     \cong  \F \ell_{n-r_{k}, n-r_{k-1}, \ldots, n-r_{1}} (E^{\vee})  
    = \F \ell^{r_{k}, r_{k-1}, \ldots, r_{1}}(E^{\vee}).   
\end{equation*}  
 Using this notation, we can write 
\begin{equation*} 
\begin{array}{llll}  
   \F \ell^{\lambda}(E)   
& = 
    \F \ell^{\nu (d-1),  \nu (d-2), \ldots, \nu (1)}(E)   
   = \F \ell_{n - \nu (d-1), n - \nu (d-2),  \ldots, n - \nu (1)}(E)    \medskip \\
  & = \F \ell_{\nu (1),  \nu (2), \ldots, \nu (d-1)}(E^{\vee})  
     = \F \ell^{n - \nu (1),  n - \nu (2), \ldots, n - \nu (d-1)}(E^{\vee}).   \medskip  
\end{array} 
\end{equation*}

\begin{ex}  [See Pragacz \cite{Pragacz2015}, Example 2]    \label{ex:PartialFlagBundles}  
\quad 

\begin{enumerate} 
\item   Let $\nu  = (\nu_{1}, \dots,  \nu_{k})   \in   \mathcal{S P}_{n}$ be 
a strict partition with length $\ell (\nu)  = k \leq   n$.  
Then as we saw at the end of \S $\ref{subsec:UHLF}$,  we have $d = k + 1$, and 
\begin{equation*} 
       (m_{1}, \ldots, m_{k}, m_{k+1}) = (\underbrace{1, 1, \ldots, 1}_{k}, n-k), 
    \quad   S^{\nu}_{n}  =  (S_{1})^{k} \times S_{n-k}.   
\end{equation*} 
Then the corresponding flag bundle $\eta_{\nu}:  \F \ell^{\nu}(E)  \longrightarrow  X$ 
is often denoted by $\tau^{k} = \tau^{k}_{E}:   \F \ell^{k, k-1, \ldots, 2, 1}(E) \longrightarrow X$,  and parametrizes  flags of successive quotient bundles  of $E$ of ranks $k, k-1, \ldots, 2, 1$.  As a special case of this example, the flag bundle corresponding to the 
partition $\rho_{n-1} = (n-1, n-2, \ldots, 2, 1, 0)$ is the full flag bundle
$\tau: \F \ell^{n-1, n-2, \ldots, 2,1}(E) = \F \ell (E) \longrightarrow X$.

\item  Let $\lambda = (\underbrace{a, \ldots, a}_{q},  \underbrace{b, \ldots, b}_{n-q})  =  (a^{q} \; b^{n-q})  \in \mathcal{P}_{n}$ be a partition of two rows with  $a > b \geq  0$.   Then we have a decomposition 
$[n]  = I_{1} \sqcup I_{2}$, where $I_{1} = [1, q] = \{  1, \ldots, q \}$ and 
$I_{2} = [q+1, n] = \{ q + 1, \ldots, n \}$.  Therefore we have 
\begin{equation*} 
   (m_{1}, m_{2}) = (q, n-q),  \quad  \text{and} \quad   S^{\lambda}_{n}  =   S_{q} \times S_{n-q}. 
\end{equation*} 
Thus the corresponding flag bundle is the Grassmann bundle 
$\pi:  G^{q}(E)  \longrightarrow X$  parametrizing rank  $q$ quotient bundles of $E$. 
\end{enumerate}  
\end{ex}  

Now we recall the useful formula for Grassmann bundle derived from the Jacobi-Trudi identity mentioned above.  
Let $\pi: G^{q}(E) \longrightarrow X$ be the Grassmann bundle parametrizing 
rank $q$ quotient bundles of $E$ as in Example \ref{ex:PartialFlagBundles} (2).  
On $G^{q}(E)$, we have the  tautological exact sequence of vector bundles: 
\begin{equation*} 
    0   \longrightarrow S \hooklongrightarrow \pi^{*}(E)   \twoheadlongrightarrow Q 
\longrightarrow 0, 
\end{equation*} 
where $\rank \, S  = n-q$ and $\rank \, Q = q$.   Set $r:= n - q$.   
Then by the repeated applications of the Jacobi-Trudi identity (\ref{eqn:Jacobi-TrudiIdentityII}), 
Fulton-Pragacz  \cite[(4.2)]{Fulton-Pragacz1998}  showed  the following formula:\footnote{
Note that in case the sequence $(\lambda_{1} - r,  \ldots, \lambda_{q} - r, \, 
\mu_{1}, \ldots, \mu_{r})$ is {\it not} a partition, the right-hand side is either 
$0$ or $\pm s_{\nu}(E)$ for some partition (see Fulton-Pragacz \cite[p.42, Footnote]{Fulton-Pragacz1998}).    
}
For any partitions $\lambda = (\lambda_{1} , \ldots, \lambda_{q})$,  
$\mu = (\mu_{1}, \ldots, \mu_{r})$,  
\begin{equation}   \label{eqn:GysinFormulas_lambda(Q)s_mu(S)}   
  \pi_{*} (s_{\lambda}(Q)  \cdot s_{\mu} (S))  
= s_{\lambda_{1}-r,  \ldots, \lambda_{q} - r,  \, \mu_{1}, \ldots, \mu_{r}}   (E).  
\end{equation} 
In \S \ref{subsubsec:Fel'dman'sGysinFormula}  below, we shall give a  
generalization of the above formula (\ref{eqn:GysinFormulas_lambda(Q)s_mu(S)})   as a 
special case of the corrected version of Fel'dman's Gysin formula.

Let us recall  Gysin formulas for more general flag bundles.  
The following proposition is a generalization of Theorem \ref{thm:Fulton-Pragacztau^*tau_*}
to the partial flag bundle $\eta_{\lambda}: \F \ell^{\lambda}(E) \longrightarrow X$, 
which was proved by Pragacz \cite{Pragacz2015}  as a particular case of Brion \cite[Proposition 2.1]{Brion1996}:   
\begin{prop}  [Pragacz  \cite{Pragacz2015}, Proposition 5]     \label{prop:Brion-Pragacz}  
For  an $S_{n}^{\lambda}$-invariant polynomial  $f(X_{1}, \ldots, X_{n})$   
$\in  H^{*}(X)[X_{1}, \ldots,  X_{n}]^{S_{n}^{\lambda}}$,  we have
\begin{equation*} 
    (\eta_{\lambda})_{*} (f(x_{1}, \ldots, x_{n}) ) 
=   \sum_{\overline{w} \in S_{n}/S_{n}^{\lambda}}
            w \cdot  \left [   
                                       \dfrac{f (x_{1}, \ldots, x_{n})} 
                                     { \prod_{1 \leq i < j \leq n, \; n(i) <  n(j) } (x_{i} - x_{j})   } 
                        \right ].   
\end{equation*}  
Here the element $f(x_{1}, \ldots, x_{n})$ is regarded as an element in $H^{*}(\F \ell^{\lambda}(E))
\hookrightarrow H^{*}(\F \ell (E))$.\footnote{
Strictly speaking, this formula should be considered in $H^{*}(\F \ell (E))$ via the pull-back 
$H^{*}(X)   \overset{\eta_{\lambda}^{*}}{\hooklongrightarrow} H^{*}(\F \ell^{\lambda} (E))) \hooklongrightarrow 
H^{*}(\F \ell (E))$ (cf. Corollary \ref{cor:Bressler-EvensThm1.8Cor}).   In what follows, 
we often use such abbreviation to simplify the presentation. 
}
\end{prop}  
\noindent
From this,  Pragacz showed (implicitly) the following formula:  
\begin{cor}  \label{cor:Pragacz2015}    \label{cor:GysinFormulasH-LPolynomials}  
For the Gysin homomorphism   $(\eta_{\lambda})_{*}:  H^{*}(\mathcal{F} \ell^{\lambda} (E)) \longrightarrow H^{*}(X)$, 
the following formula holds$:$ 
\begin{equation*} 
     (\eta_{\lambda})_{*} \left ( \x^{\lambda}  \prod_{1 \leq i < j \leq n,  \; n(i) <  n(j) }  (x_{i} - tx_{j})  \right )  = P_{\lambda}(E; t),  
\end{equation*} 
where the cohomology class $P_{\lambda}(E; t)$ is defined from the Hall-Littlewood polynomial
$P_{\lambda}(x_{1}, \ldots, x_{n}; t)$ in the same way as $s_{\lambda}(E)$ at the beginning of 
this subsection.  
\end{cor}  

By Corollary \ref{cor:GysinFormulasH-LPolynomials}, one can deduce 
Gysin formula for Schur $P$-polynomials (see Pragacz \cite[Examples 2  and 11, Corollary 6]{Pragacz2015}):  
\begin{cor}   [Pragacz \cite{Pragacz2015}, Corollary 6] 
In the setting as in Example $\ref{ex:PartialFlagBundles}$ $\mathrm{(1)}$, 
the following formula holds$:$
\begin{equation*} 
    (\tau^{k})_{*}  \left  (
                                            \x^{\nu}  \prod_{1 \leq i \leq k, \;   1 \leq i < j \leq n}   
                                              (x_{i}  -  t x_{j})
                      \right )  
        =  P_{\nu} (E; t).  
\end{equation*} 
\end{cor} 
\noindent
Since we know that  $P_{\nu}(\x_{n}; -1) =  P_{\nu}(\x_{n})$ (see (\ref{eqn:Hall-LittlewoodSchurP})),  
we obtain the following: 
\begin{cor}    \label{cor:GysinFormulaSchurP-Polynomial}  
With the above notation,  the following formula holds$:$
\begin{equation*} 
    (\tau^{k})_{*}  \left  (
                                            \x^{\nu}  \prod_{1 \leq i \leq k, \;   1 \leq i < j \leq n}   
                                              (x_{i} + x_{j})
                      \right )  
        =  P_{\nu} (E).  
\end{equation*} 
\end{cor}  


\subsection{Universal Gysin formulas for the universal Schur functions}  
\label{subsec:UGFUSF}  
\subsubsection{Application of  the Bressler-Evens formula}   \label{subsubsec:ApplicationBressler-EvensFormula(Cobordism)}    
Since the Bressler-Evens formula (Theorem \ref{thm:Bressler-EvensThm1.8}) 
is formulated in complex-oriented generalized cohomology theories, 
it can be  applied especially  to the complex cobordism theory $MU^{*}(-)$. 
We use the same notation as in \S \ref{subsubsec:GysinFormulasFl(E)->X}.  
For a complex vector bundle $E  \overset{p}{\longrightarrow}  X$ of rank $n$,  one can associate
the $MU^{*}$-theory Chern classes\footnote{
For the historical reason, they are also called {\it Conner-Floyd Chern classes} 
(see Adams \cite[Part I, \S4]{Adams1974}, Conner-Floyd \cite[Corollary 8.3]{Conner-Floyd1966}).   
}  $c^{MU}_{i}(E)  \in MU^{2i}(X) \; (i = 1, 2, \ldots, n)$ and $c^{MU}_{0}(E) := 1$.  
Put $x_{i} = x_{i}^{MU}  := c^{MU}_{1}(L^{i}) \in MU^{2}(\F \ell (E)) \; (i = 1, 2, \ldots, n)$
(the $MU^{*}$-theory Chern roots of $E$).  
Then the $MU^{*}$-cohomology of $\F \ell (E)$ is given as follows (see e.g., Hornbostel-Kiritchenko \cite[Theorem 2.6]{Hornbostel-Kiritchenko2011}): 
\begin{equation*} 
   MU^{*}(\F \ell (E))  = MU^{*}(X)[x_{1}, \ldots, x_{n}]/(\prod_{i=1}^{n}  (1 + x_{i})  = c^{MU}(E)), 
\end{equation*} 
where $c^{MU}(E) = \sum_{i=0}^{n}  c^{MU}_{i}(E)$ is the total Chern class of $E$.   
Then  we obtain the following result, which is  a {\it universal } analogue 
of Theorem \ref{thm:Fulton-Pragacztau^*tau_*}:  
\begin{theorem}      \label{thm:Nakagawa-Narusetau^*tau_*}  
With the same notation as in Theorem $\ref{thm:Fulton-Pragacztau^*tau_*}$,  we have 
\begin{equation*} 
      \tau^{*} \circ  \tau_{*}(f(x_{1}, \ldots, x_{n})) 
      =  \sum_{w \in S_{n}}  w \cdot 
                          \left [   
                                      \dfrac{f (x_{1}, \ldots, x_{n})} 
                                          { \prod_{1 \leq i < j \leq n} (x_{i}  +_{\L}   \overline{x}_{j})} 
                          \right ]   
\end{equation*}  
for a polynomial $f(X_{1}, \ldots, X_{n}) \in MU^{*}(X)[X_{1}, \ldots, X_{n}]$.  
\end{theorem} 
\noindent
Thus the Gysin map $\tau_{*}$ in the complex cobordism theory is  also given by 
a certain symmetrizing operator  which is a  universal  analogue of 
the Jacobi symmetrizer  (see Theorem \ref{thm:Fulton-Pragacztau^*tau_*}).

By Theorem \ref{thm:Nakagawa-Narusetau^*tau_*}  and
(\ref{eqn:Definitions^L(x_n|b)}),  
we obtain the following corollary,  which is a universal analogue of 
Proposition \ref{prop:Jacobi-TrudiIdentity}:  
\begin{cor} [Characterization of the universal Schur functions]  \label{cor:Nakagawa-Naruse(CharacterizationUSF)}   
 The image of the monomial  
                    $ {\bf x}^{\lambda + \rho_{n-1}} = x_{1}^{\lambda_{1}+n-1} x_{2}^{\lambda_{2} + n-2} \cdots x_{n}^{\lambda_{n}}$
  under the Gysin homomorphism $\tau_{*}:  MU^{*}(\mathcal{F} \ell (E))  
         \longrightarrow MU^{*}(X)$ is given by 
 \begin{equation*}  \label{eqn:GysinFormula2} 
       \tau_{*}({\bf x}^{\lambda + \rho_{n-1}} )
        =  s^{\L}_{\lambda} (E).         
       \end{equation*}  
Here the characteristic class $s^{\L}_{\lambda}(E)  \in  MU^{2|\lambda|} (X)$ 
is defined by the same  manner as $s_{\lambda} (E) \in H^{2|\lambda|}(X)$.  
\end{cor} 
\begin{rem}      \label{rem:tau_*(x^{lambda+rho_{n-1}})}  
Since $\tau^{*} (s^{\L}_{\lambda}(E))  = s^{\L}_{\lambda}(\x_{n}) \in MU^{*}(\F \ell (E))$, 
and $\tau^{*}$ is injective, the above formula can be written as 
\begin{equation*} 
         \tau_{*} (\x^{\lambda + \rho_{n-1}})   = s^{\L}_{\lambda}(\x_{n}). 
\end{equation*} 
More generally,  the symmetrizing operator description of $\tau_{*}$ $($Theorem $\ref{thm:Nakagawa-Narusetau^*tau_*})$ yields formally  the following formula$:$ 
\begin{equation}   \label{eqn:tau_*([x|b]^{lambda+rho_{n-1}})}   
      \tau_{*}([\x|\b]^{\lambda + \rho_{n-1}})  =  s^{\L}_{\lambda}(\x_{n}|\b). 
\end{equation} 
Here   $\b = (b_{1}, b_{2}, \ldots)$ is  a sequence of 
certain elements in $MU^{*}(X)$, which behaves as scalars with respect to 
$\tau_{*}$.  
\end{rem}

The universal analogue of Proposition \ref{prop:Brion-Pragacz}  
can also be obtained.  Under the same setting  as in 
Proposition \ref{prop:Brion-Pragacz},    we have a  classifying map  $h: X \longrightarrow BU(n)$
and its lift $\tilde{h}:  \F \ell^{\lambda}(E) \longrightarrow B(U(m_{1}) \times U(m_{2}) \times \cdots \times U(m_{d}))$,  
 and   the following diagram is commutative: 
\begin{equation*} 
 \begin{CD} 
          \mathcal{F}  \ell^{\lambda} (E)   @>{\tilde{h}}>>       B(U(m_{1}) \times U(m_{2}) \times \cdots \times U(m_{d}))  \\
                @V{\eta_{\lambda}}VV                                 @VV{\sigma}V \\
            X                        @>{h}>>                 BU(n).  \\
\end{CD} 
\end{equation*} 
By Corollary \ref{cor:Bressler-EvensThm1.8Cor} and the base-change property of 
Gysin maps, we obtain immediately the 
following generalization of  Proposition \ref{prop:Brion-Pragacz}:    
\begin{theorem}    \label{thm:Nakagawa-Naruse(GeneralFlagBundles)}  
For  an $S_{n}^{\lambda}$-invariant polynomial  $f (X_{1}, \ldots, X_{n})  
 \in  MU^{*}(X)[X_{1}, \ldots,  X_{n}]^{S_{n}^{\lambda}}$,  we have
\begin{equation*} 
   (\eta_{\lambda})_{*} (f(x_{1}, \ldots, x_{n}) ) 
=   \sum_{\overline{w} \in S_{n}/S_{n}^{\lambda}}  
                    w \cdot  \left [   
                                      \dfrac{f (x_{1}, \ldots, x_{n})} 
                                 {\prod_{1 \leq i < j \leq n, \; n(i) <  n(j) } (x_{i} +_{\L} \overline{x}_{j}) }
                               \right ].   
\end{equation*}  
\end{theorem}  
From this, we have the following corollary which is a universal analogue of 
Corollary \ref{cor:GysinFormulasH-LPolynomials}:  
\begin{cor} [Characterization of the universal Hall-Littlewood  functions]    \label{cor:Nakagawa-Naruse(CharacterizationUHLF)}  
For the Gysin homomorphism   $(\eta_{\lambda})_{*}:  MU^{*}(\mathcal{F} \ell^{\lambda} (E)) \longrightarrow MU^{*}(X)$, 
the following formula holds$:$ 
\begin{equation*} 
   (\eta_{\lambda})_{*} \left (  \x^{\lambda}   
     \prod_{1 \leq i < j \leq n, \;  n(i) <  n(j) }  (x_{i} +_{\L}  [t]_{\L} (\overline{x}_{j}))   \right )  
=  H^{\L}_{\lambda} (E; t).    
\end{equation*} 
\end{cor}  
For a strict partition $\nu  \in \mathcal{S P}_{n}$,  we saw that  $H^{\L}_{\nu}(\x_{n}; -1)  = P^{\L}_{\nu}(\x_{n})$ at the end of \S \ref{subsec:UHLF}.  From this and 
Corollary \ref{cor:Nakagawa-Naruse(CharacterizationUHLF)},  we obtain 
the following corollary, which is a universal analogue of Corollary \ref{cor:GysinFormulaSchurP-Polynomial}: 
\begin{cor} 
With the same notation as in Corollary $\ref{cor:GysinFormulaSchurP-Polynomial}$, 
we have 
\begin{equation*} 
    (\tau^{k})_{*}  \left  (
                                            \x^{\nu}  \prod_{1 \leq i \leq k, \;  1 \leq i < j \leq n}   
                                              (x_{i} +_{\L}   x_{j})
                      \right )  
        =  P^{\L}_{\nu} (E).  
\end{equation*} 
\end{cor}

\subsubsection{Fel'dman's Gysin formula}     \label{subsubsec:Fel'dman'sGysinFormula}   
As promised in the footnote after Remark \ref{rem:Fel'dman'sGysinFormula}, 
we shall reformulate Fel'dman's Gysin formula \cite[Theorem 4.5]{Fel'dman2003}
in the complex cobordism theory, and  prove it.   To this end, 
we use the universal Schur function $s^{\L}_{I}(\x_{n})$ for a sequence 
of non-negative integers $I$ defined in \S \ref{subsec:UFSSPQF}.  
For two sequences $I = (I_{1}, I_{2}, \ldots, I_{q})$,  $J = (J_{1}, J_{2}, \ldots, J_{r})$
of non-negative integers,  
denote by $IJ$ their {\it juxtaposition}, i.e., 
\begin{equation*}  
   IJ := (I_{1}, I_{2}, \ldots, I_{q}, J_{1}, J_{2}, \ldots, J_{r}). 
\end{equation*} 

Let  $\lambda = (\lambda_{1}, \ldots, \lambda_{n})  \in \mathcal{P}_{n}$ be a partition 
of length $\leq n$,  and rewrite it as $\lambda =   (n_{1}^{m_{1}} \; n_{2}^{m_{2}} \; \cdots n_{d}^{m_{d}})$, 
$n_{1} > n_{2} >  \cdots > n_{d} \geq 0$,  according to the decomposition 
$[n] = I_{1} \sqcup I_{2} \sqcup \cdots \sqcup I_{d}$.   
Then  consider the $(d-1)$-step partial flag bundle
\begin{equation*} 
    \eta_{\lambda}:  \F \ell^{\lambda}(E)  \longrightarrow X
\end{equation*} 
associated with a complex vector bundle $E \longrightarrow X$ (for the notation, 
see \S  \ref{subsec:GysinFormulasVariousSchurFunctions}). 
We are concerned with the induced Gysin map in the complex cobordism theory: 
\begin{equation*} 
  (\eta_{\lambda})_{*}:   MU^{*}(\F \ell^{\lambda} (E))  \longrightarrow  MU^{*}(X). 
\end{equation*} 
On $\F \ell^{\lambda}(E)$,  we have the tautological sequence of flag 
of sub  and quotient bundles: 
\begin{equation*} 
\begin{array}{llll}  
  & S_{0} = 0  \hooklongrightarrow S_{1} \hooklongrightarrow  \cdots \hooklongrightarrow 
         S_{r} \hooklongrightarrow 
 \cdots \hooklongrightarrow S_{d-1}  \hooklongrightarrow S_{d} = \eta_{\lambda}^{*} (E),   \medskip \\
    &    \eta_{\lambda}^{*} (E) = Q^{d}  \twoheadlongrightarrow  Q^{d-1} \twoheadlongrightarrow 
       \cdots  \twoheadlongrightarrow Q^{r}  \twoheadlongrightarrow \cdots 
     \twoheadlongrightarrow Q^{1}  \twoheadlongrightarrow Q^{0} = 0,  \medskip 
\end{array}  
\end{equation*} 
where,  $Q^{r} :=  \eta_{\lambda}^{*} (E)/S_{d-r} \; (r = 1, 2, \ldots, d)$.   
Define the vector bundles over $\F \ell^{\lambda} (E)$ by 
\begin{equation*} 
    E_{r} := \Ker \, (Q^{r}   \twoheadlongrightarrow Q^{r-1})  \quad (r = 1, 2, \ldots, d). 
\end{equation*} 
Then   we have $\eta_{\lambda}^{*}(E)  \cong \bigoplus_{r=1}^{d} E_{r}$, and 
the $MU^{*}$-theory Chern roots of $E_{r}$ are given by $x_{\nu (r-1) + 1},  \ldots, x_{\nu (r)}$
($r = 1, 2, \ldots, d$). 
Here $x_{1}, \ldots, x_{n}$ are  the $MU^{*}$-theory Chern roots of $E$ (see the beginning of 
\S  \ref{subsubsec:ApplicationBressler-EvensFormula(Cobordism)}), and $\nu (r) := \sum_{i=1}^{r} m_{i}$
for $r = 1, 2, \ldots, d$ and $\nu (0) := 0$ (see the end of \S  \ref{subsec:UHLF}).  
With this notation,  our version of Fel'dman's Gysin formula is stated as follows: 
\begin{theorem} [cf. Fel'dman  \cite{Fel'dman2003}, Theorem 4.5]   
  For sequences of non-negative integers $I^{(r)} = (I^{(r)}_{1}, I^{(r)}_{2}, \ldots, I^{(r)}_{m_{r}})
\; (r = 1, 2, \ldots, d)$,   we have 
\begin{equation*} 
      (\eta_{\lambda})_{*} \left ( \prod_{r=1}^{d}  s^{\L}_{I^{(r)}  + ((n - \nu (r))^{m_{r}}) } (E_{r})   \right )    =  s^{\L}_{I^{(1)} I^{(2)} \cdots I^{(d)}} (E).  
\end{equation*} 
Here the partition $((n - \nu (r))^{m_{r}})$ means 
$( \underbrace{n - \nu (r),  n - \nu (r), \ldots, n - \nu (r)}_{m_{r}})$.  
\end{theorem} 

\begin{proof} 
As in Fel'dman's proof,  it suffices to prove the assertion for the case of 
a Grassmann bundle. Therefore we may assume that 
the partition $\lambda$ is of the form 
$\lambda = (a^{q} \,  b^{n-q})$ with $a > b \geq 0$, and 
consider the associated Grassmann bundle $\pi:  G^{q}(E)  \longrightarrow X$ 
(see Example \ref{ex:PartialFlagBundles} (2)).  
On $G^{q}(E)$, we have the tautological exact sequence of vector bundles: 
\begin{equation*} 
    0  \longrightarrow S  \hooklongrightarrow  \pi^{*}(E) \twoheadlongrightarrow Q 
   \longrightarrow 0. 
\end{equation*} 
Note that   $E_{1} = Q$, and $E_{2} = \Ker \, (\pi^{*}(E)  \twoheadlongrightarrow Q) \cong S$.  
For given two sequences of non-negative integers $I = (I_{1}, I_{2}, \ldots, I_{q})$, 
$J = (J_{1}, J_{2}, \ldots, J_{n-q})$,  consider the universal Schur function 
corresponding  to the juxtaposition  $I   J = (I_{1},\ldots, I_{q}, J_{1}, \ldots, J_{n-q})$:  
\begin{equation*} 
   s^{\L}_{I J} (x_{1}, \ldots, x_{q}, x_{q+1}, \ldots, x_{n}) 
   =  \sum_{w \in S_{n}}  w \cdot \left [ 
                                       \dfrac{ \x^{IJ + \rho_{n-1}}  }  
                                               {\prod_{1 \leq i < j \leq n}  (x_{i} +_{\L} \overline{x}_{j}) } 
                                          \right ].  
\end{equation*} 
We decompose the denominator and the numerator  inside the bracket in the right-hand side as 
\begin{equation*} 
\begin{array}{llll} 
 &   \displaystyle{\prod_{1 \leq i < j \leq n}}   (x_{i} +_{\L} \overline{x}_{j}) 
   = \prod_{1 \leq i  < j \leq q} (x_{i} +_{\L} \overline{x}_{j})  \times 
     \prod_{q + 1 \leq i  < j \leq n}  (x_{i}  +_{\L} \overline{x}_{j})  \times 
     \prod_{1 \leq i \leq q,  \;  q + 1 \leq j \leq n}  (x_{i} +_{\L} \overline{x}_{j}),  \medskip \\ 
 &  \x^{IJ + \rho_{n-1}}   
 =  \displaystyle{\prod_{i=1}^{q}}  x_{i}^{I_{i} + n-i}  \times \prod_{i=q+1}^{n}  x_{i}^{J_{i-q} + n-i}   
 = (x_{1} \cdots x_{q})^{n-q}  \times \prod_{i=1}^{q}  x_{i}^{I_{i} + q - i}  \times 
      \prod_{i=q+1}^{n}  x_{i}^{J_{i-q}  + n-i}.  \medskip 
\end{array}    
\end{equation*} 
Since $\prod_{1 \leq i \leq q, \;  q + 1 \leq j  \leq n}  (x_{i} +_{\L} \overline{x}_{j})$
and  $(x_{1} \cdots x_{q})^{n-q}$  are both $S_{q} \times S_{n-q}$-invariant, 
we compute 
\begin{equation*} 
\begin{array}{llll} 
  &  s^{\L}_{I J} (x_{1}, \ldots, x_{q}, x_{q+1}, \ldots, x_{n}) 
     =   \displaystyle{\sum_{w \in S_{n}}}    w \cdot \left [ 
                                       \dfrac{ \x^{IJ + \rho_{n-1}}  }  
                                               {\prod_{1 \leq i < j \leq n}  (x_{i} +_{\L} \overline{x}_{j}) } 
                                          \right ]   \medskip \\
   & =   \displaystyle{\sum_{\overline{w} \in S_{n}/S_{q} \times S_{n-q}}}    
      w \cdot \left [  
                              \dfrac{(x_{1} \cdots x_{q})^{n-q}} 
                                      {\prod_{1 \leq i \leq q, \; q + 1 \leq j \leq n}  
                                                       (x_{i} +_{\L} \overline{x}_{j})  } 
                               \right.     \medskip \\
   & \hspace{3cm}   \times     \left.   \displaystyle{\sum_{(u, v)  \in  S_{q} \times S_{n-q}}}   
                                             u \cdot \left [   
                                                            \dfrac{ \prod_{i=1}^{q} x_{i}^{I_{i}+q-i}} 
                                                                     { \prod_{1 \leq i  < j \leq q} (x_{i} +_{\L} \overline{x}_{j})  }  
                                                          \right ] 
                                \times 
                                            v \cdot \left [ 
                                                                \dfrac{ \prod_{i=q+1}^{n} x_{i}^{J_{i-q} + n-i}}
                                                                            {\prod_{q + 1 \leq i  < j \leq n}  (x_{i}  +_{\L} \overline{x}_{j})   }
                                                       \right ]    
                  \right ]   \medskip   \\
& =  \displaystyle{\sum_{\overline{w} \in S_{n}/S_{q} \times S_{n-q}}}    
      w \cdot \left [   
            \dfrac{(x_{1} \cdots x_{q})^{n-q} 
                                s^{\L}_{I}(x_{1}, \ldots, x_{q})  s^{\L}_{J}(x_{q+1}, \ldots, x_{n}) }
                                       { \prod_{1 \leq i \leq q, \; q + 1 \leq j \leq n}  
                                                       (x_{i} +_{\L} \overline{x}_{j})   } 
                   \right ]    \medskip \\
   & = \pi_{*}  ((x_{1} \cdots x_{q})^{n-q}  s^{\L}_{I} (x_{1}, \ldots, x_{q})   \, s^{\L}_{J}(x_{q+1}, \ldots, x_{n}) )   \medskip   \\
   & = \pi_{*} ( s^{\L}_{I + ((n-q)^{q})}  (x_{1}, \ldots, x_{q})   \, s^{\L}_{J}(x_{q+1}, \ldots, x_{n})).  \medskip 
\end{array} 
\end{equation*} 
In the final step, we used the formula ``$(x_{1} \cdots x_{n})^{c}  \times s^{\L}_{I}(\x_{n})
=  s^{\L}_{I  + (c^{n})}  (\x_{n})$'' for any non-negative integer $c$, 
which follows immediately from the 
definition of $s^{\L}_{I}(\x_{n})$. 
Thus   we have 
\begin{equation*}     \label{eqn:pi_*(s^L_Is^L_J)}  
     \pi_{*} (s^{\L}_{I + ((n-q)^{q})}  (x_{1}, \ldots, x_{q})  s^{\L}_{J}(x_{q+1}, \ldots, x_{n}))
   =   s^{\L}_{I J}  (x_{1}, \ldots, x_{n}).     
\end{equation*} 
Rewriting this expression in terms of vector bundles, we have the desired 
formula:  
\begin{equation*} 
                 \pi_{*} (s^{\L}_{I + ((n-q)^{q})} (E_{1})  \,  s^{\L}_{J} (E_{2}))  = s^{\L}_{I J} (E).   
\end{equation*}
\end{proof}  

\section{New universal factorial Schur functions}     \label{sec:NUFSF}  
\subsection{Definition of the new universal factorial Schur functions}  
 In \S  \ref{subsec:UHLF},  we defined the universal Hall-Littlewood function
$H^{\L}_{\lambda}(\x_{n};t)$ for a partition $\lambda  \in  \mathcal{P}_{n}$.  
We saw that  for a strict partition $\nu \in \mathcal{S P}_{n}$, 
 this function reduces to the 
universal Schur $P$-function $P^{\L}_{\nu}(\x_{n})$  under the specialization $t = -1$.  
As announced in \S \ref{subsec:UHLF},  we  shall  consider  the specialization $t = 0$  in this subsection. 
In fact, unlike the usual Hall-Littlewood polynomial $P_{\lambda}(x_{1}, \ldots, x_{n}; t)$, 
the universal Hall-Littlewood function $H^{\L}_{\lambda}(\x_{n}; t)$
need not coincide with the universal Schur function $s^{\L}_{\lambda}(\x_{n})$ 
under the specialization $t = 0$.  Thus one obtains another {\it universal} 
analogue of the Schur polynomial, which we  call  the ``new universal 
(factorial) Schur function''.  
We shall use the same notation as in \S \ref{subsec:UHLF}. 
 Let $\lambda = (\lambda_{1}, \ldots, \lambda_{n})  \in \mathcal{P}_{n}$ be a 
 partition  of  length $\leq n$.  
Then we rewrite 
$\lambda = (n_{1}^{m_{1}} \;  n_{2}^{m_{2}} \;  \ldots \;  n_{d}^{m_{d}})$, $n_{1} > n_{2} > \cdots > n_{d} \geq 0$ as before.   
Put $\nu (r) = \sum_{i=1}^{r} m_{i}$ for $r = 1, 2, \ldots, d$  and $\nu (0) = 0$.  
Define 
\begin{equation}    \label{eqn:Monomial(NUFSF)}  
  (\x | \b)^{[\lambda]}  
   := \displaystyle{\prod_{r=1}^{d}}  \left ( \prod_{\nu (r-1) < i \leq \nu (r)}   [x_{i}|\b]^{n_{r} + n - \nu (r)}   \right ).  
\end{equation}  
When the parameters $\b = \bm{0}$, then we write  simply $\x^{[\lambda]}$ for 
$(\x|\bm{0})^{[\lambda]}$.   
With the above notation, we make the following definition:  
\begin{defn}[New universal factorial Schur functions]  \label{df:DefinitionNUFSF}  
For a partition $\lambda  = (\lambda_{1}, \ldots, \lambda_{n})  \\  \in \mathcal{P}_{n}$, 
we define  
\begin{equation*} 
     \mathbb{S}^{\L}_{\lambda} (\x_{n}|\b)  = \mathbb{S}^{\L}_{\lambda}(x_{1}, \ldots, x_{n}|\b)  
 := \sum_{\overline{w} \in S_{n}/S_{n}^{\lambda}}  w \cdot \left [  \dfrac{(\x|\b)^{[\lambda]} }{\prod_{
                       1 \leq i < j \leq n,  \; 
                            n(i)  <   n(j)   
          } (x_{i}  +_{\L} \overline{x}_{j})}    \right ]. 
\end{equation*}  
We also define 
\begin{equation*} 
          \mathbb{S}^{\L}_{\lambda}(\x_{n}) = \mathbb{S}^{\L}_{\lambda}(x_{1}, \ldots, x_{n}) 
         := \mathbb{S}^{\L}_{\lambda}(x_{1}, \ldots, x_{n}| \bm{0}).  
\end{equation*} 
\end{defn}
\noindent
It follows immediately from Definition  \ref{df:DefinitionNUFSF} and (\ref{eqn:DefinitionNUSF}) that 
when $t = 0$, the universal Hall-Littlewood function $H^{\L}_{\lambda}(\x_{n}; t)$ 
reduces to the new universal Schur function $\mathbb{S}^{\L}_{\lambda}(\x_{n})$.   
Moreover, 
one sees directly  from the definition that 
if we specialize the universal formal group law $F_{\L}(u, v)$ to the additive 
one $F_{a}(u, v) = u + v$ (and $b_{i}  = -a_{i} \; (i = 1, 2, \ldots)$),  
 then the new universal factorial Schur function 
$\s^{\L}_{\lambda}(\x_{n}|\b)$ reduces to the factorial Schur function 
$s_{\lambda}(\x_{n}|a)$ (see the proof of Macdonald \cite[Chapter III, \S 1,  (1.5)]{Macdonald1995}).   
On the other hand, it is not obvious from Definition \ref{df:DefinitionNUFSF} 
that $\s^{\L}_{\lambda}(\x_{n}|\b)$ reduces to the factorial Grothendieck 
polynomial $G_{\lambda}(\x_{n}|\b)$ under the specialization from 
$F_{\L}(u, v)$ to the multiplicative one $F_{m}(u, v) = u + v + \beta uv$. 
We shall show that this is the case after establishing  some Gysin formulas 
for the new universal factorial Schur functions (see \S \ref{subsec:GysinFormulasNUSF}).  
Thus the new universal factorial Schur functions
are  also   {\it universal} analogues of  the usual  Schur functions.  
Note that, by definition,  $\s^{\L}_{\emptyset} (\x_{n}|\b) = 1$ 
for   the empty partition $\emptyset$, whereas $s^{\L}_{\emptyset}(\x_{n}|\b) 
\neq 1$ (see  \S \ref{subsec:UFSSPQF}).

\begin{ex}      \label{ex:ExamplesNUFSF}  
\quad

\begin{enumerate} 
\item  Let $\lambda \in \mathcal{P}_{n}$ be a partition such that 
$\lambda_{1} > \lambda_{2} > \cdots > \lambda_{n-1} > \lambda_{n} \geq 0$. 
Then we have a decomposition $[n] = I_{1} \sqcup I_{2} \sqcup \cdots \sqcup I_{n}$, 
where $I_{i} = \{ i \}$. Therefore we have 
\begin{equation*} 
    (m_{1}, m_{2}, \ldots, m_{n}) = (\underbrace{1, 1, \ldots, 1}_{n})  \quad \text{and}   \quad 
    S^{\lambda}_{n}  =  \underbrace{S_{1} \times  \cdots  \times  S_{1}}_{n} = \{ 1 \}. 
\end{equation*} 
In this case,  one sees immediately from $(\ref{eqn:Monomial(NUFSF)})$ 
that $(\x |\b)^{[\lambda]} =   [\x |\b]^{\lambda + \rho_{n-1}}$,  and 
hence 
\begin{equation*} 
        s^{\L}_{\lambda}(\x_{n}|\b)       =  \s^{\L}_{\lambda}(\x_{n}|\b). 
\end{equation*} 
Thus for such a  distinct partition $\lambda$,  the  function $s^{\L}_{\lambda}(\x_{n}|\b)$ 
coincides with the   function $\s^{\L}_{\lambda}(\x_{n}|\b)$.  
However, for a partition with equal  parts,  the difference  does occur $($see the end of 
\S $\ref{subsec:GysinFormulasNUSF})$.

\item   
Let $k \geq 1$ be a positive integer, and 
consider the case where $\lambda$ is ``one-row'', namely  
$\lambda = (k)  = (k \; 0^{n-1})$.   Then we have a decomposition 
$[n] = I_{1} \sqcup I_{2}$, where $I_{1} = \{ 1 \}$,  $I_{2} = [2, n]  = \{ 2, \ldots, n \}$. 
Thus   we have 
\begin{equation*} 
    (m_{1}, m_{2}) = (1,  n-1)    \quad \text{and}  \quad   S^{\lambda}_{n}  = S_{1}  \times S_{n-1}, 
\end{equation*} 
and one sees immediately that $\x^{[\lambda]} = x_{1}^{k + n-1}$.   
Therefore the 
function $\s^{\L}_{k}(\x_{n}) := \s^{\L}_{(k)}(\x_{n})$ corresponding to 
the one-row $(k)$ is given by  
\begin{equation}   \label{eqn:NUSF(One-Row)}   
    \s^{\L}_{k}(\x_{n})  =  \sum_{\overline{w} \in S_{n}/S_{1} \times S_{n-1}}  
                  w \cdot \left [  
                                         \dfrac{x_{1}^{k + n-1}} 
                                           {\prod_{j=2}^{n} (x_{1}  +_{\L}   \overline{x}_{j}) } 
                              \right ]  
                         =  \sum_{i=1}^{n}   \dfrac{x_{i}^{k + n-1}}  
                                                          {\prod_{j \neq i} (x_{i} +_{\L} \overline{x}_{j}) }.  
\end{equation} 
 Here we assumed that $k \geq 1$ is a positive integer.  However, the right-hand side 
of $(\ref{eqn:NUSF(One-Row)})$ makes sense for  $k \geq 1-n$,  which could be  non-positive, 
or actually negative if $n \geq 2$.  Therefore one can formally define $\s^{\L}_{k}(\x_{n})$ 
for each integer $k \geq 1 -n$ by the right-hand side of the above expression. 
For instance, one has 
\begin{equation*} 
    \s^{\L}_{0}(\x_{n})  =  \sum_{i=1}^{n}  \dfrac{x_{i}^{n-1}}  
                                                             {\prod_{j \neq i} (x_{i} +_{\L} \overline{x}_{j})}, 
\end{equation*}  
which differs from $1$.\footnote{
Note that   $\s^{\L}_{\emptyset} (\x_{n}) = 1$ by definition. 
}    Moreover, as we will see in the next subsection \S $\ref{subsec:GeneratingFunctionNUSF}$, 
one can define $\s^{\L}_{k}(\x_{n})$ for any integer $k \in \Z$. 
\end{enumerate} 

\end{ex}  

\begin{rem} 
In our previous paper \cite[\S 4.5]{Nakagawa-Naruse2016}, we investigated 
various properties of the universal factorial Schur functions 
$s^{\L}_{\lambda}(\x_{n}|\b)$ such as 
``vanishing property'',  
``basis theorem''. 
The new universal factorial Schur functions $\s^{\L}_{\lambda}(\x_{n}|\b)$  also 
have the similar properties. We shall discuss this problem 
elsewhere.  
\end{rem}

\subsection{Gysin formulas for the new universal Schur functions}   \label{subsec:GysinFormulasNUSF}  
In this subsection, we shall establish  Gysin formulas for the new universal 
factorial Schur functions.   Using these formulas, we are able to compare the 
new universal factorial Schur functions (``new'' functions for short) $\s^{\L}_{\lambda}(\x_{n}|\b)$ 
with the universal factorial Schur functions (``old'' functions for short) 
 $s^{\L}_{\lambda}(\x_{n}|\b)$
introduced in \S \ref{subsec:UFSSPQF}.

\subsubsection{Gysin formulas for the new universal Schur functions} 
We begin with the following theorem: 
From Theorem  \ref{thm:Nakagawa-Naruse(GeneralFlagBundles)}, 
we   have  immediately  
\begin{theorem} [Characterization of the new universal Schur functions]  \label{thm:Nakagawa-Naruse(CharacterizationNUSF)}   
For the Gysin homomorphism  
 $(\eta_{\lambda})_{*}:  MU^{*}(\mathcal{F} \ell^{\lambda}(E)) \longrightarrow  MU^{*}(X)$,  the following formula holds$:$
 \begin{equation*}   \label{eqn:eta_lambda(x^[lambda])}     
    (\eta_{\lambda})_{*} (\x^{[\lambda]})  =  \s^{\L}_{\lambda}(\x_{n}), 
 \end{equation*}   
where $x_{1}, \ldots, x_{n}$ are the $MU^{*}$-theory Chern roots of $E$ as before.  
More gererally,   the following factorial version  holds  for a 
sequence  $\b = (b_{1}, b_{2}, \ldots )$   of certain elements in $MU^{*}(X)$$:$  
\begin{equation}    \label{eqn:eta_lambda((x|b)^[lambda])}   
         (\eta_{\lambda})_{*}((\x|\b)^{[\lambda]})  = \s^{\L}_{\lambda}(\x_{n}|\b). 
\end{equation}  

\end{theorem}

\begin{ex}    \label{ex:GysinFormulasNUSF}  
\quad 

\begin{enumerate} 
\item Let us consider the ``one-row'' case $\lambda = (k)$ as in Example 
$\ref{ex:ExamplesNUFSF}$ $(2)$.  
Thus the corresponding flag bundle is the projective bundle
 $\pi_{1}:  G^{1}(E) \longrightarrow X$, and Theorem  $\ref{thm:Nakagawa-Naruse(CharacterizationNUSF)}$ implies the following formula$:$ 
\begin{equation*} 
     \pi_{1 *}(x_{1}^{k + n-1})  =    \s^{\L}_{k}(\x_{n}).  
\end{equation*}

\item   Let $\lambda = (\underbrace{a, \ldots, a}_{q},  \underbrace{b, \ldots, b}_{n-q})  =  (a^{q} \; b^{n-q})  \in \mathcal{P}_{n}$ be a partition of two rows with  $a > b \geq  0$ as in 
Example $\ref{ex:PartialFlagBundles}$ $(2)$.  Then the corresponding flag 
bundle is the Grassmann bundle $\pi:  G^{q}(E)   \longrightarrow X$, and one sees 
directly that $\x^{[\lambda]}  
=  x_{1}^{a + n-q} \cdots x_{q}^{a + n-q}  x_{q + 1}^{b} \cdots x_{n}^{b}$, 
and Theorem $\ref{thm:Nakagawa-Naruse(CharacterizationNUSF)}$ implies 
\begin{equation*} 
    \pi_{*} (x_{1}^{a + n-q} \cdots x_{q}^{a + n-q}  x_{q + 1}^{b} \cdots x_{n}^{b}) 
   = \s^{\L}_{(a^{q} \;  b^{n-q})}  (\x_{n}).   
\end{equation*}

\item   Let $\lambda \in \mathcal{P}_{n}$ be a distinct partition 
such that 
$\lambda_{1} > \lambda_{2} > \cdots > \lambda_{n-1} > \lambda_{n} \geq 0$ 
as in Example  $\ref{ex:ExamplesNUFSF}$ $(1)$. 
Then  the corresponding partial flag bundle is the full flag bundle
$\tau: \F \ell (E)  \longrightarrow X$.  In this case,  one obtains that $\x^{[\lambda]} =  \x^{\lambda + \rho_{n-1}}$, and Corollary   $\ref{cor:Nakagawa-Naruse(CharacterizationUSF)}$ 
and  Theorem $\ref{thm:Nakagawa-Naruse(CharacterizationNUSF)}$  imply   
\begin{equation*} 
        s^{\L}_{\lambda} (\x_{n}) 
    = \tau_{*}(\x^{\lambda + \rho_{n-1}})   
     =  \s^{\L}_{\lambda} (\x_{n}),    
\end{equation*} 
which was  already seen  in Example $\ref{ex:ExamplesNUFSF}$ $(1)$ $($with $\b = \bm{0}$$)$.  


\end{enumerate} 
\end{ex}

\subsubsection{Comparison of ``new'' functions $\s^{\L}_{\lambda}(\x_{n}|\b)$ 
with  ``old'' functions  $s^{\L}_{\lambda}(\x_{n}|\b)$
}

More generally, the above Gysin  formulas (Theorem \ref{thm:Nakagawa-Naruse(CharacterizationNUSF)}) 
can be formulated as those between two partial flag bundles of the form   
``$\pi^{\mu}_{\lambda}:  \F \ell^{\mu}(E)  \twoheadlongrightarrow \F \ell^{\lambda}(E)$'', 
where the partition $\mu$ is a ``refinement'' of   $\lambda$.   
Let us explain what the word ``refinement'' means:  
We use the notation in the beginning of this subsection (see also \S \ref{subsec:UHLF} 
and \S \ref{subsec:GysinFormulasVariousSchurFunctions}).  For two positive integers $a$, $b$ with $a < b$, denote by 
$[a, b]$ the set of integers $i$ such that $a \leq i \leq b$.
Let $\lambda, \mu  \in \mathcal{P}_{n}$ be two partitions of length $\leq n$.  
Then as explained in \S \ref{subsec:UHLF},  
one obtains two   decompositions   
\begin{equation*} 
  [1, n]  = I_{1} \sqcup  I_{2} \sqcup  \cdots \sqcup I_{d},  \quad 
  [1, n]  = J_{1}  \sqcup  J_{2} \sqcup  \cdots \sqcup J_{e}
\end{equation*} 
of the interval $[1, n]$ corresponding to $\lambda$, $\mu$ respectively.  
Suppose that  the decomposition $[1, n] = J_{1} \sqcup \cdots \sqcup J_{e}$ is 
a refinement of the decomposition $[1, n] = I_{1} \sqcup \cdots \sqcup I_{d}$ 
in the usual sense, i.e.,  $I_{1} = J_{1} \sqcup \cdots \sqcup J_{k_{1}}$, 
$I_{2} = J_{k_{1} + 1} \sqcup \cdots \sqcup J_{k_{2}}$, and so on,  for some positive 
integers $1 \leq k_{1} < k_{2} <  \cdots \leq e$.    
Then we say  that $\mu$ is a ``refinement'' of $\lambda$.\footnote{
Note that even though $\mu$ is a ``refinement'' of $\lambda$, 
this does not necessarily mean $\mu \subset \lambda$.  
}
By the construction of the associated partial flag bundle $\F \ell^{\lambda}(E)$ 
(see \S \ref{subsec:GysinFormulasVariousSchurFunctions}), 
we have a natural projection from $\F \ell^{\mu} (E)$ to $\F \ell^{\lambda}(E)$,  denoted by $\pi^{\mu}_{\lambda}$.  
Furthermore, for  three partitions $\lambda, \mu$, and $\nu \in \mathcal{P}_{n}$, 
if $\nu$ is a refinement of $\mu$, and $\mu$ is a refinement of $\lambda$, 
then $\nu$ is also a refinement of $\lambda$, and the 
relation $\pi^{\mu}_{\lambda} \circ \pi^{\nu}_{\mu} = \pi^{\nu}_{\lambda}$
holds.   For example, 
the partition $\rho_{n-1} =  (n-1, n-2, \ldots, 1, 0)  \in \mathcal{P}_{n}$ 
is a refinement of any partition $\lambda \in \mathcal{P}_{n}$, 
and any partition $\lambda \in \mathcal{P}_{n}$ is a 
refinement of the empty partition $\emptyset = (0^{n})$. 
In particular, we are concerned with   the following three   projections  
and the induced Gysin homomorphisms in complex cobordism:  
\begin{equation}  \label{eqn:ProjectionsPartialFlagBundles}    
\begin{array}{llll} 
   &   \pi^{\rho_{n-1}}_{\emptyset} = \tau:  \F \ell^{\rho_{n-1}}(E)  = \F \ell (E)    \twoheadlongrightarrow 
            \F \ell^{\emptyset} (E) = X, \medskip   \\
&  \pi^{\lambda}_{\emptyset} = \eta_{\lambda}:  
         \F \ell^{\lambda}(E)  \twoheadlongrightarrow \F \ell^{\emptyset} (E) = X, \medskip \\
 & \pi^{\rho_{n-1}}_{\lambda}:  \F \ell^{\rho_{n-1}}(E) = \F \ell (E)  \twoheadlongrightarrow 
   \F \ell^{\lambda}(E).   \medskip 
\end{array} 
\end{equation}  
The theorem below (Theorem \ref{thm:GysinFormulasNUFSF}) 
will be useful in describing the difference between the ``new'' 
function $\s^{\L}_{\lambda}(\x_{n}|\b)$ and the ``old'' one 
$s^{\L}_{\lambda}(\x_{n}|\b)$.  
Note that    $\tau$ is written as a composite $\eta_{\lambda}  \circ  \pi^{\rho_{n-1}}_{\lambda}$, and hence  $\tau_{*}  =   (\eta_{\lambda})_{*} \circ 
(\pi^{\rho_{n-1}}_{\lambda})_{*}$.  
\begin{theorem}    \label{thm:GysinFormulasNUFSF}   
\quad 
\begin{enumerate} 
\item  For the Gysin map $(\pi^{\rho_{n-1}}_{\emptyset})_{*} 
= \tau_{*}:  MU^{*}(\F \ell (E))  \longrightarrow  MU^{*}(X)$,  
the following formula holds$:$ 
\begin{equation*} 
    (\pi^{\rho_{n-1}}_{\emptyset})_{*} ([\x|\b]^{\lambda + \rho_{n-1}})  
   =  s^{\L}_{\lambda}(\x_{n}|\b).    
\end{equation*} 

\item  For the Gysin map $(\pi^{\lambda}_{\emptyset})_{*} = (\eta_{\lambda})_{*}: 
MU^{*}(\F \ell^{\lambda}(E))  \longrightarrow   MU^{*}(X)$, 
the following formula holds$:$ 
\begin{equation*} 
    (\pi^{\lambda}_{\emptyset})_{*}((\x|\b)^{[\lambda]})  = \s^{\L}_{\lambda} (\x_{n}|\b).
\end{equation*} 

\item For the Gysin map $(\pi^{\rho_{n-1}}_{\lambda})_{*}:  MU^{*}(\F \ell (E))  \longrightarrow MU^{*}(\F \ell^{\lambda} (E))$,  the following formula holds$:$ 
\begin{equation*} 
  (\pi^{\rho_{n-1}}_{\lambda})_{*} ([\x|\b]^{\lambda + \rho_{n-1}}) 
        =  (\x|\b)^{[\lambda]}  \cdot \prod_{r = 1}^{d}  s^{\L}_{\emptyset} (x_{\nu (r-1) + 1}, \ldots, x_{\nu (r)} |\b[+[n_{r} + n - \nu (r)]]).  
\end{equation*} 
Here  $\b = (b_{1}, b_{2}, \ldots )$  is a sequence of elements in $MU^{*}(X)$, and 
$\b [+m] := (b_{m+1}, b_{m+2}, \ldots)$ for a positive integer $m$.  
\end{enumerate} 
\end{theorem}

\begin{proof} 
(1)  and (2):  
The Gysin maps $(\pi^{\rho_{n-1}}_{\emptyset})_{*} = \tau_{*}$
and $(\pi^{\lambda}_{\emptyset})_{*}  =  (\eta_{\lambda})_{*}$ 
were  already considered in (\ref{eqn:tau_*([x|b]^{lambda+rho_{n-1}})}) 
and (\ref{eqn:eta_lambda((x|b)^[lambda])}) respectively.    

(3)  We shall show the assertion when the partition $\lambda  \in \mathcal{P}_{n}$ is 
of the form $\lambda = (a^{q} \,  b^{n-q}) $ (see Example \ref{ex:GysinFormulasNUSF} (2)).  
In this case, one sees 
\begin{equation*} 
\begin{array}{lll} 
     & [\x|\b]^{\lambda + \rho_{n-1}}  
                 =  \displaystyle{\prod_{i=1}^{q}}   [x_{i}|\b]^{a + n-i}  \times 
                                                   \prod_{i=q+1}^{n} [x_{i}|\b]^{b + n-i}   \medskip \\
                &  =  \displaystyle{\prod_{i=1}^{q}}  [x_{i}|\b]^{a + n-q} \times \prod_{i=q+1}^{n} [x_{i}|\b]^{b} 
                           \times \prod_{i=1}^{q} [x_{i}|\b[+(a + n-q)]]^{q-i} 
                           \times  \prod_{i=q+1}^{n} [x_{i} |\b[+b]]^{n-i}    \medskip \\
                & = (\x|\b)^{[\lambda]}  \times 
               \displaystyle{\prod_{i=1}^{q}}   [x_{i}|\b[+(a + n-q)]]^{q-i} 
                           \times  \prod_{i=q+1}^{n} [x_{i} |\b[+b]]^{n-i}.      \medskip 
\end{array}  
\end{equation*} 
Since $(\x |\b)^{[\lambda]}$ is $S_{q}  \times S_{n-q}$-invariant, 
we have by the Bressler-Evens formula (Theorem \ref{thm:Bressler-EvensThm1.8}), 
\begin{equation*} 
\begin{array}{llll} 
    &  (\pi^{\rho_{n-1}}_{\lambda})_{*} ([\x|\b]^{\lambda + \rho_{n-1}})    
    =   \displaystyle{\sum_{w \in S_{q}  \times S_{n-q}}}  w \cdot 
      \left [ 
                   \dfrac{[\x|\b]^{\lambda + \rho_{n-1}}} 
                           {\prod_{1 \leq i <  j \leq q}  (x_{i} +_{\L} \overline{x}_{j})  
                              \times \prod_{q +1 \leq i  < j \leq n}   (x_{i} +_{\L} \overline{x}_{j})
                           }  
     \right ]   \medskip \\ 
& =  (\x|\b)^{[\lambda]}  \times 
      \displaystyle{\sum_{w \in S_{q}  \times S_{n-q}}}  w \cdot 
      \left [ 
                   \dfrac{   \prod_{i=1}^{q}  [x_{i}|\b[+(a + n-q)]]^{q-i} 
                           \times  \prod_{i=q+1}^{n} [x_{i} |\b[+b]]^{n-i}.       } 
                           {\prod_{1 \leq i <  j \leq q}  (x_{i} +_{\L} \overline{x}_{j})  
                              \times \prod_{q +1 \leq i  < j \leq n}   (x_{i} +_{\L} \overline{x}_{j})
                           }  
     \right ]   \medskip \\ 

  & =  (\x|\b)^{[\lambda]}  \times s^{\L}_{\emptyset}  (x_{1}, \ldots, x_{q}|\b[+(a + n-q)]) 
      \times s^{\L}_{\emptyset}(x_{q+1}, \ldots, x_{n}|\b[+b]),  
\end{array}   
\end{equation*} 
and the theorem holds.

For an arbitrary partition $\lambda \in \mathcal{P}_{n}$, write 
$\lambda$ as the following form: $\lambda = (n_{1}^{m_{1}} \; n_{2}^{m_{2}} \; \cdots \; n_{d}^{m_{d}})$, $n_{1} > n_{2} > \cdots > n_{d} \geq 0$. 
Here $m_{i} > 0$ for $i = 1, 2, \ldots, d$, and $\sum_{i=1}^{d} m_{i} = n$.  
Put $\nu (r) = \sum_{i=1}^{r} m_{i}$ for $r = 1, 2, \ldots, d$ and 
$\nu (0) = 0$.   Then  analogous  computation to the above  case  leads to the result of 
$(\pi^{\rho_{n-1}}_{\lambda})_{*}([\x|\b]^{\lambda + \rho_{n-1}})$.   
Note that in the case where $\lambda = (a^{q} \,  b^{n-q})$, 
the ``parameter shift'' is given by $a + n-q = n_{1} + n - \nu (1)$ and $b = n_{2} + n - \nu (2)$.  
\end{proof}

With the aide of Theorem \ref{thm:GysinFormulasNUFSF}, 
we observe that

\begin{equation}  \label{eqn:DifferenceOldNew}   
\begin{array}{llll} 
    s^{\L}_{\lambda}(\x_{n}|\b)  &  =  \tau_{*}([\x|\b]^{\lambda + \rho_{n-1}}) 
        =  (\eta_{\lambda})_{*} \circ (\pi^{\rho_{n-1}}_{\lambda})_{*} ([\x|\b]^{\lambda + \rho_{n-1}})     \medskip \\
     & = (\eta_{\lambda})_{*}  \left  ( 
                   (\x|\b)^{[\lambda]} 
   \cdot     \displaystyle{\prod_{r = 1}^{d}}    s^{\L}_{\emptyset} (x_{\nu (r-1) + 1}, \ldots, x_{\nu (r)} |\b[+[n_{r} + n - \nu (r)]]) 
           \right ),     \medskip \\
     \s^{\L}_{\lambda}(\x_{n}|\b) &  =  (\eta_{\lambda})_{*} ((\x|\b)^{[\lambda]}).    \medskip 
\end{array}  
\end{equation} 
Thus  the difference between the  ``new'' function $\s^{\L}_{\lambda}(\x_{n}|\b)$ and 
the ``old''   one  $s^{\L}_{\lambda}(\x_{n}|\b)$  is given  by the product of  
``old'' functions corresponding  to   the the empty partition $\emptyset$.  
From this,  ``new'' function $\s^{\L}_{\lambda}(\x_{n}|\b)$ reduce to $G_{\lambda}(\x_{n}|\b)$ 
under the specialization from $F_{\L}(u, v)$ to $F_{m}(u, v)$  because of the fact $G_{\emptyset}(\x_{n}|\b) = 1$ (see e.g., Ikeda-Naruse \cite[\S 2.4]{Ikeda-Naruse2013}).



\subsection{Generating function for $\s^{\L}_{k}(\x_{n})$}    \label{subsec:GeneratingFunctionNUSF}  
As we mentioned in Example \ref{ex:ExamplesNUFSF} (2),  
the new universal Schur functions $\s^{\L}_{k}(\x_{n}) := \s^{\L}_{(k)}(\x_{n})$ 
for $k \geq 1$  
can be extended to non-positive integers. Thus one has 
the functions $\s^{\L}_{k}(\x_{n})$ for all integers $k \in \Z$. 
In this subsection, we shall give the generating function for these.  
For this purpose, we make use of Quillen's result.  
Recall from Quillen \cite{Quillen1969} 
that the 
{\it normalized invariant differential form}\footnote{
If we put $a^{\L}_{i, j} = 0$ for all $i, j \geq 1$, then $\omega_{\L}(t)$ reduces to 
$dt$.  If we put $a^{\L}_{1, 1} = \beta$ and $a^{\L}_{i, j} = 0$ for all $(i, j) \neq (1, 1)$, 
then $\omega_{\L} (t)$ reduces to $\dfrac{dt}{1 + \beta t}$.  
} $\omega_{\L}(t)$ associated 
with the universal formal group law $F_{\L}(u, v)  = u + v + \sum_{i, j \geq 1} a^{\L}_{i, j} u^{i} v^{j}$ is defined by 
\begin{equation*} 
   \omega_{\L} (t) :=  \dfrac{d t}{  \dfrac{\partial F_{\L}}{\partial v} (t, 0)  
                                    }  
                  = \dfrac{dt}{  1 + \sum_{i \geq 1}  a^{\L}_{i, 1} t^{i}
                                     }. 
\end{equation*} 
The logarithm $\ell_{\L} (x)$ of $F_{\L}$  is then determined by the equations 
\begin{equation*} 
     \ell_{\L}' (t) \, d  t  =   \omega_{\L} (t),   \quad \ell_{\L} (0) = 0. 
\end{equation*} 
Then Quillen  gave the following formula (see also Damon \cite[p.650, Proposition]{Damon1973}): 
\begin{theorem}  [Quillen \cite{Quillen1969}, Theorem 1]   \label{thm:QuillenResidueFormula} 
Let $E \longrightarrow X$ be a complex vector bundle of rank $n$, let 
$\pi_{1}:  G^{1}(E) \cong P (E^{\vee})  \longrightarrow X$ be the associated projective 
bundle of lines in the dual $E^{\vee}$ of $E$,  and let $Q  \cong \mathcal{O} (1)$ 
be the canonical quotient line bundle on $G^{1}(E) \cong P(E^{\vee})$.  
Then the Gysin homomorphism $\pi_{1 *}:  MU^{*}(G^{1}(E)) \longrightarrow MU^{*}(X)$ 
is given by the residue formula
\begin{equation}   \label{eqn:QuillenResidueFormula}  
     \pi_{1 *} (f(\xi))    =  \displaystyle{\Res_{t = 0}} 
                                                         \dfrac{ f(t) \,  \omega_{\L}(t)}  
                                                         {\prod_{i=1}^{n}  (t +_{\L}  \overline{x}_{i})}
          = \Res_{t = 0}  \dfrac{  f(t) \, dt} 
                                                   {(1 + \sum_{i \geq 1} a^{\L}_{i, 1}t^{i})
                                                    \prod_{i=1}^{n} (t +_{\L} \overline{x}_{i})}. 
\end{equation} 
Thus  $\pi_{1 *}(f(\xi))$ is given by the coefficient of $t^{-1}$ in the Laurent series 
\begin{equation*} 
        \dfrac{  f(t)} 
                {(1 + \sum_{i \geq 1} a^{\L}_{i, 1}t^{i})
                   \prod_{i=1}^{n} (t +_{\L} \overline{x}_{i})}. 
\end{equation*} 
Here $f(t) \in MU^{*}(X)[t]$,  $\xi := c^{MU}_{1}(Q) \in MU^{2}(G^{1}(E))$, and 
$x_{1}  = \xi$, $x_{2}, \ldots, x_{n}$ are the $MU^{*}$-theory Chern roots of $E$.   
\end{theorem}  
On the other hand,  in the same setting as in Examples \ref{ex:ExamplesNUFSF} and 
\ref{ex:GysinFormulasNUSF}, 
together with Theorem \ref{thm:Nakagawa-Naruse(GeneralFlagBundles)}, we obtain 
\begin{equation}    \label{eqn:PushForwardFormula(ProjectiveBundle)}   
       \pi_{1 *} (f(\xi))  =   \sum_{i=1}^{n}   
                                          \dfrac{ f(x_{i})} 
                                       { \prod_{j \neq i}  (x_{i} +_{\L} \overline{x}_{j})}.     
\end{equation} 
The equivalence of (\ref{eqn:QuillenResidueFormula}) and (\ref{eqn:PushForwardFormula(ProjectiveBundle)}) is shown by Vishik \cite[Lemma  5.36]{Vishik2007}.
Therefore if we take $f(t) = t^{k + n-1} \; (k \geq 1)$, 
we have by Example \ref{ex:ExamplesNUFSF} (2),  
\begin{equation} 
\begin{array}{llll}  
      \s^{\L}_{k}(\x_{n})   &   =  \displaystyle{\sum_{i=1}^{n}}   
      \dfrac{x_{i}^{k + n-1}} {\prod_{j \neq i}  (x_{i} +_{\L} \overline{x}_{j}) }   
=  \pi_{1 *}  (\xi^{k + n-1})    \medskip \\
&  = 
\Res_{t = 0}  \dfrac{ t^{k + n -1} \, dt}  
                                             {(1 + \sum_{i \geq 1} a^{\L}_{i, 1}t^{i}) \prod_{i=1}^{n} (t +_{\L} \overline{x}_{i})}.   \medskip 
\end{array}   
\end{equation} 
Thus  the coefficient of $t^{-1}$ in  the formal Laurent series 
\begin{equation*} 
        \dfrac{t^{k + n - 1}} 
                                {(1 + \sum_{i \geq 1} a^{\L}_{i, 1} t^{i}) \prod_{i=1}^{n}
                                     (t +_{\L}  \overline{x}_{i})} 
\end{equation*} 
is equal to $\s^{\L}_{k}(\x_{n})$. 

\begin{rem} 
   It should be remarked that the above residue formula is closely related to the 
recent work of Darondeau-Pragacz \cite{Darondeau-Pragacz2015}.  
The following comment might be useful  when one wants to relate the residue formula 
to their work. We borrow the notation from \cite[\S 0]{Darondeau-Pragacz2015}.  
For a Laurent  series $f(t)$  in one indeterminate $t$,  we shall denote by 
$[t^{k}]( f(t))$ the coefficient  of $t^{k}$ in $f (t)$.   The the above residue formula 
means 
\begin{equation*} 
\begin{array}{llll}  
   \s^{\L}_{k}(\x_{n})   &   = [t^{-1}]  \left (  
                             \dfrac{t^{k + n - 1}} 
                                {(1 + \sum_{i \geq 1} a^{\L}_{i, 1} t^{i}) \prod_{i=1}^{n}
                                (t +_{\L}  \overline{x}_{i})}  
                                      \right )    \medskip \\
                               & =  [t^{n-1}]  \left (  
                                      t^{k + n-1}  \cdot   
                             \dfrac{t^{n}} 
                                {(1 + \sum_{i \geq 1} a^{\L}_{i, 1} t^{i}) \prod_{i=1}^{n}
                                (t +_{\L}  \overline{x}_{i})}  
                                      \right ).     \medskip \\
                                 
\end{array}  
\end{equation*} 
If one wants to consider the ordinary cohomology theory,  we put $a^{\L}_{i, j} = 0$ 
for all $i, j \geq 1$.  Then the above formula leads to   
\begin{equation*} 
       s_{k}(\x_{n})   =   [t^{n-1}]  \left (
                                              t^{k + n-1} \cdot \dfrac{t^{n}}{\prod_{i=1}^{n}  (t - x_{i})} 
                                          \right ). 
\end{equation*} 
Here the left-hand side is the usual Schur polynomial corresponding to the one row $(k)$, 
namely, the $k$-th homogeneous complete symmetric polynomial.   
We  see easily that the rational  function 
 $\dfrac{t^{n}}{\prod_{i=1}^{n} (t - x_{i})}$  in the right-hand side  is the ``reversed Segre polynomial'' $s_{1/t}(E)$
of $E$ if we think  the variables $x_{1}, \ldots, x_{n}$  the Chern roots of  
a complex vector bundle $E$.  
This  is  the fundamental formula given  in \cite[\S 0, p.2]{Darondeau-Pragacz2015}.\footnote{Since Darondeau-Pragacz uses the projective bundle 
of lines $G_{1}(E)  = P(E)   \longrightarrow  X$, not $G^{1} (E)  \cong P(E^{\vee})$,  
in their formulation,  one has to change $E$  in their result  to its dual $E^{\vee}$.   
} 
\end{rem}

 From the above  interpretation, we can obtain the 
 generating function for $\s^{\L}_{k}(\x_{n}) \; (k \in \Z)$.  We argue as 
follows:  
Set 
\begin{equation*} 
     F_{n}(t) :=  \dfrac{t^{n}} 
          {(1 + \sum_{i \geq 1} a^{\L}_{i, 1} t^{i}) \prod_{i=1}^{n} (t +_{\L}  \overline{x}_{i})}.  
\end{equation*} 
Then  it follows from the above interpretation that the coefficient of $t^{-k}$ of $F_{n}(t)$ is equal  to $\s^{\L}_{k}(\x_{n})$ for each positive integer $k \geq 1$. 
By defining $\s^{\L}_{k}(\x_{n})$ for $k \leq 0$ by the same procedure, 
one obtains 
\begin{equation*} 
         F_{n}(t)  = \sum_{k   \in \Z}  \s^{\L}_{k}(\x_{n}) t^{-k}. 
\end{equation*} 
Thus we have  the following: 
\begin{theorem}  \label{thm:GeneratingFunctionNUSF} 
The generating function for $\s^{\L}_{k}(\x_{n}) \; (k \in \Z)$ is given 
by
\begin{equation}   \label{eqn:GeneratingFunctionNUSF}   
       \sum_{k \in \Z}  \s^{\L}_{k}(\x_{n}) u^{k}  
 = F_{n}(t)|_{t = u^{-1}} 
=  \dfrac{u^{-n}} 
          {(1 + \sum_{i \geq 1} a^{\L}_{i, 1} u^{-i}) \prod_{i=1}^{n} (u^{-1} +_{\L}  \overline{x}_{i})}.  
\end{equation} 
\end{theorem} 

\begin{rem} 
\quad 

\begin{enumerate} 
\item 
In Hudson-Matsumura \cite[Definition 3.1]{Hudson-Matsumura2016},  
the Segre class  in the algebraic cobordism theory  $\mathscr{S}_{m}(E) \; (m \in \Z)$ of  a complex vector bundle $E$ 
is defined  by using the push-forward image  from the projective bundle 
$G^{1}(E) \cong P(E^{\vee})$.\footnote{
For the $K$-theoretic analogue of the Segre classes of a vector bundle, 
see Buch  \cite[Lemma 7.1]{Buch2002(Duke)}. 
}
   By definition, these classes coincide with 
our $\s^{\L}_{m}(E)$.  Furthermore, they obtained the generating function 
$\mathscr{S}(E; u) := \sum_{m \in \Z}  \mathscr{S}_{m}(E) u^{m}$ 
of the Segre classes \cite[Theorem 3.6]{Hudson-Matsumura2016}  by a  different method from  ours. One can check directly  that their result coincides with our Theorem $\ref{thm:GeneratingFunctionNUSF}$.  As for their argument,   readers are recommended to consult  Hudson-Ikeda-Matsumura-Naruse \cite[\S 3.1]{HIMN2015}.  In that paper, the $K$-theoretic Segre classes are introduced by the same way as above, 
and the generating function of the stable Grothendieck polynomials 
is given \cite[Theorem 3.2, Appendix 8]{HIMN2015}.  Their result can also be obtained 
from Theorem $\ref{thm:GeneratingFunctionNUSF}$ by the specialization $a^{\L}_{1, 1} = \beta$
and $a^{\L}_{i, j} = 0$ for all $(i, j)  \neq (1, 1)$.    

\item  Since our formula $(\ref{eqn:GeneratingFunctionNUSF})$ is universal, 
          one can obtain the generating function of the ``$h^{*}$-theory Segre 
          classes''  of vector bundle by sepecializing the universal formal group  
         law $F_{\L}(u, v)$ to the formal group law $F_{h}(u, v)$ corresponding 
        to a given complex-oriented cohomology theory $h^{*}(-)$.   
        For instance, we  recently  obtained a concrete expression 
       of the ``Elliptic Schur functions'' corresponding to a cohomology 
      theory, denoted $\mathcal{S E}^{*} (-)$, whose formal group law is  that of a singular cubic curve in 
       Weierstrass form, called hyperbolic $($see Lenart-Zainoulline 
      \cite{Lenart-Zainoulline2014}, \cite{Lenart-Zainoulline2015}$)$.   
\end{enumerate}  
\end{rem}

\subsection{Application of the Gysin formulas for the new universal 
Schur functions --Thom-Porteous formula for the complex  cobordism theory--  
} 
\label{subsec:Thom-PorteousFormulaComplexCobordism}  
Using  the new universal Schur functions, one can formulate 
the Thom-Porteous formula (\ref{eqn:Thom-PorteousFormulaII}) in the 
universal setting.  We use the same notation as in \S \ref{subsubsec:Thom-PorteousFormula}. 
As explained in that subsection, in order to obtain the class determined by 
$D_{r}(\varphi)$, we have to compute the image 
$\pi_{F *} (c^{MU}_{e(f-r)} (\pi^{*}_{F}(E)^{\vee} \otimes Q_{F}))$.  
Let $x_{1}, \ldots, x_{f}$ (resp. $b_{1}, \ldots, b_{e}$) be  the $MU^{*}$-theory 
Chern roots of $F$ (resp. $E$).  
The Chern roots of $Q_{F}$ are $x_{1}, \ldots, x_{f-r}$. 
  Let $\lambda = ((e-r)^{(f-r)})$ be the rectangular partition with $(f-r)$ rows and 
$(e-r)$ columns as in \S \ref{subsubsec:Thom-PorteousFormula}. 
Then by the splitting principle, the top Chern class is given by 
\begin{equation*} 
   c^{MU}_{e(f-r)} (\pi^{*}_{F}(E)^{\vee} \otimes Q_{F})  
  =  \prod_{i=1}^{f-r} \prod_{j=1}^{e} (x_{i} +_{\L}   \overline{b}_{j})  
 = \prod_{i=1}^{f-r} [x_{i}|  \overline{\b}_{e}]^{e}   =   (\x|  \overline{\b}_{e})^{[\lambda]}.   
\end{equation*} 
Here $\overline{\b}_{e} = (\overline{b}_{1}, \ldots,  \overline{b}_{e}, 0, 0, \ldots)$.  
Therefore by   (\ref{eqn:eta_lambda((x|b)^[lambda])}) (see also Theorem \ref{thm:Nakagawa-Naruse(CharacterizationNUSF)}  and  Example \ref{ex:GysinFormulasNUSF} (2)), one obtains\footnote{
Since  we have 
   $\s^{\L}_{(e^{(f-r)})}  (\x_{f-r}|  \overline{\b}_{e})   
                            =  \prod_{i=1}^{f-r}  [x_{i}|  \overline{\b}_{e}]^{e}$,  
this formula can be written as 
\begin{equation*} 
    \pi_{F *}  (\s^{\L}_{(e^{(f-r)})}  (\x_{f-r}|  \overline{\b}_{e}))  
=  \s^{\L}_{((e-r)^{(f-r)})}(\x_{f}|  \overline{\b}_{e}).  
\end{equation*} 
} 
\begin{equation*} 
      \pi_{F *} (c^{MU}_{e(f-r)} (\pi^{*}_{F}(E)^{\vee} \otimes Q_{F})) 
    = \pi_{F *} ((\x|  \overline{\b}_{e})^{[\lambda]})  
    = \s^{\L}_{\lambda}(\x_{f} |  \overline{\b}_{e}).  
\end{equation*} 
\begin{theorem} [Thom-Porteous formula for the complex cobordism theory]   \label{thm:Thom-PorteousFormulaComplexCobordism}  
If the codimension of $D_{r}(\varphi)$ is $(e-r)(f-r)$,  then the class 
determined by $D_{r}(\varphi)$ is given by  the new universal factorial Schur 
function $\s^{\L}_{((e-r)^{(f-r)})}(\x_{f}|  \overline{\b}_{e})$.  

\end{theorem} 

\subsection{Application of the Gysin formulas for the new universal Schur functions 
--Class  of  Damon's resolution--}      \label{subsec:Damon'sResolution}  
As we mentioned in the introduction \S $\ref{subsec:GysinFormulasSchurFunctions}$, 
it is well-known that the usual Schur polynomials $s_{\lambda}(\x_{d})$, 
with  $\lambda$ contained in the rectangular partition $((n-d)^{d})$, 
represent the Schubert classes in the ordinary cohomology ring $H^{*}(G_{d}(\C^{n}))$
 of the complex Grassmannian $G_{d}(\C^{n})$ 
of $d$-dimensional linear subspaces in $\C^{n}$ 
 (see, e.g., Fulton \cite[\S 9.4]{Fulton1997}).  
In this subsection, by making use of the  Damon's resolution 
of the Schubert varieties (Damon \cite[\S 3, p.258]{Damon1974}), together with our 
Gysin formula (\ref{eqn:eta_lambda((x|b)^[lambda])}),       
we shall  show that the new universal Schur functions $\s^{\L}_{\lambda}(\y_{d})$, 
here $y_{1}, \ldots, y_{d}$ are the $MU^{*}$-theory Chern roots of   the 
dual bundle $S^{\vee}$ of the tautological vector bundle $S$ on $G_{d}(\C^{n})$, 
give the ``correct'' representatives of the ``Schubert classes'' in 
the complex cobordism ring $MU^{*}(G_{d}(\C^{n}))$.  
\subsubsection{Bundle of Schubert varieties}  
In fact, we can formulate our assertion in more general situation: 
Let $E \overset{p}{\longrightarrow}   X$ be a complex vector bundle of rank $n$ over a
variety $X$.  For a positive integer $1 \leq d \leq n-1$, 
consider the associated   Grassmann bundle $\pi:  G_{d}(E) \longrightarrow X$ of $d$-dimensional  subspaces in the fibers of $E$. 
On $G_{d}(E)$, one has the tautological exact sequence of vector bundles
\begin{equation*} 
    0  \longrightarrow S  \hooklongrightarrow \pi^{*}(E) \twoheadlongrightarrow 
   Q  \longrightarrow  0, 
\end{equation*} 
where $\rank S = d$ and $\rank Q = n-d$.  
Suppose that we are given a complete flag of subbundles
\begin{equation*} 
    F_{\bullet}:  0 = F_{0}  \subset F_{1} \subset \cdots  \subset F_{i}  \subset  \cdots 
    \subset F_{n-1}  \subset F_{n} = E, 
\end{equation*} 
where $\rank F_{i}  = i \; (i = 0, 1, 2, \ldots, n)$.  
For any partition $\lambda  = (\lambda_{1}, \ldots, \lambda_{d})  
\subset ((n-d)^{d})$,  
the {\it bundle of Schubert varieties}\footnote{cf. Damon \cite[\S 1, p.251]{Damon1974},  Hudson-Ikeda-Matsumura-Naruse \cite[\S 4.1]{HIMN2015},  Darondeau-Pragacz \cite[\S 1.1]{Darondeau-Pragacz2016}.  
}  
$\Omega_{\lambda} (F_{\bullet})  \subset G_{d}(E)$ is defined by the so-called 
``Schubert conditions'', namely 
\begin{equation*} 
  \Omega_{\lambda} (F_{\bullet})  
:=  \{  W    \in G_{d}(E)  \; | \;    W  \subset E_{x}  \; (x \in X), \;  
       \dim \, (W  \cap (F_{n - d + i - \lambda_{i}})_{x})   \geq i  \; (i = 1, 2, \ldots, d)    \}. 
\end{equation*}

\subsubsection{Damons's resolution of singularities of $\Omega_{\lambda} (F_{\bullet})$}  
Following Damon \cite[\S 3, p.258]{Damon1974}, we will construct a resolution of singularities 
(or  desingularization) of 
$\Omega_{\lambda} (F_{\bullet})$ (with a slight modification).  
For a partition $\lambda = (\lambda_{1}, \ldots, \lambda_{d})  \subset ((n-d)^{d})$, 
we have a decomposition $[d] = I_{1} \sqcup I_{2} \sqcup \cdots \sqcup I_{k}$ 
as in \S \ref{subsec:UHLF}.  Putting  $m_{i} := \sharp  I_{i} \; (i = 1, 2, \ldots, k)$, 
we can rewrite 
\begin{equation*} 
   \lambda = (n_{1}^{m_{1}} \; n_{2}^{m_{2}} \cdots n_{k}^{m_{k}}) \quad 
 (n-d \geq  n_{1} > n_{2} > \cdots > n_{k} \geq 0). 
\end{equation*} 
   We put 
 $\nu (p) := \sum_{i=1}^{p} m_{i} \; (p = 1, 2, \ldots, k)$  and $\nu (0) := 0$.  
Since the variety $\Omega_{\lambda} (F_{\bullet})$ is determined by the 
Schubert conditions corresponding to  ``outside corners'' of the 
Young diagram of $\lambda$ (see Fulton \cite[\S 9.4, Exercise 18]{Fulton1997}), 
one can also define $\Omega_{\lambda}(F_{\bullet})$ by 
\begin{equation}  \label{eqn:DefinitionSchubertVariety(EssentialBox)} 
\begin{array}{llll}   
  & \Omega_{\lambda} (F_{\bullet})  
=  \{  W    \in G_{d}(E)  \; | \;    W  \subset E_{x}  \; (x \in X),   \medskip \\
  &  \hspace{3.5cm}  
     \dim \, (W  \cap (F_{n - d + \nu (p) - n_{p}})_{x})   \geq  \nu (p) \; (p = 1, 2, \ldots, k)    \}.  \medskip 
\end{array}   
\end{equation}

Now consider the partial flag bundle $\varpi:  \F \ell_{\nu (1), \nu (2), \ldots, \nu (k-1)} (S) \longrightarrow G_{d}(E)$ associated with the tautological subbundle $S$ over $G_{d}(E)$. 
The fiber over  a point $W  \in G_{d}(E)$ consists of partial flags  in  a $d$-dimensional 
linear subspace $W  \subset E_{x}$ for some point $x \in X$, namely 
a nested sequence of linear subspaces of the form 
\begin{equation*} 
     W_{1}  \subset W_{2}  \subset \cdots \subset W_{k-1} \subset W_{k} = W, 
\end{equation*} 
where $\dim W_{p}  = \nu (p) \; (p = 1, \ldots, k)$.  
On $\F \ell_{\nu (1), \nu (2), \ldots, \nu (k-1)} (S)$, one has the tautological 
sequence of flag of sub and quotient bundles
\begin{equation*} 
\begin{array}{llll} 
    & D_{1}  \hooklongrightarrow D_{2} \hooklongrightarrow \cdots 
    \hooklongrightarrow D_{p} \hooklongrightarrow \cdots \hooklongrightarrow 
    D_{k-1}  \hooklongrightarrow D_{k} = \varpi^{*}(S),   \medskip \\
     &  \varpi^{*}(S)  \twoheadlongrightarrow Q_{1}  \twoheadlongrightarrow Q_{2} 
   \twoheadlongrightarrow \cdots \twoheadlongrightarrow Q_{p} \twoheadlongrightarrow
  \cdots \twoheadlongrightarrow  Q_{k-1} \twoheadlongrightarrow Q_{k}  = 0, 
\end{array}  
\end{equation*}
where $\rank D_{p} = \nu (p)$,  and $Q_{p}$ is defined by $Q_{p} := \varpi^{*}(S)/D_{p} \; (p = 1, 2, \ldots, k)$.  
Then the partial flag bundle $\F \ell_{\nu (1), \nu (2), \ldots, \nu (k-1)}(S)$ is constructed 
as a tower of Grassmann bundles (here we omit the pull-back notation of 
vector bundles as is customary): 
\begin{equation*} 
\begin{array}{lll}  
  & \F \ell_{\nu (1), \nu (2), \ldots, \nu (k-1)} (S)  =  G_{m_{k}}(Q_{k-1}) 
       \overset{\varpi_{k}}{\underset{=}{\longrightarrow}} G_{m_{k-1}}(Q_{k-2})  
       \overset{\varpi_{k-1}}{\longrightarrow}  \cdots   \medskip \\
 & \hspace{7cm}        \overset{\varpi_{3}}{\longrightarrow}  G_{m_{2}}(Q_{1}) 
      \overset{\varpi_{2}}{\longrightarrow} G_{m_{1}}(S) 
      \overset{\varpi_{1}}{\longrightarrow} G_{d}(E).    \medskip 
\end{array}   
\end{equation*} 
The natural projection 
$\varpi = \varpi_{1} \circ \varpi_{2} \circ \cdots \circ \varpi_{k}: \F \ell_{\nu (1), \nu (2), \ldots, \nu (k-1)}(S)  \longrightarrow G_{d}(E)$ sends a point $(0 \subset W_{1} \subset \cdots \subset W_{k-1} \subset  W_{k} = W)  \in \F \ell_{\nu (1), \nu (2), \ldots, \nu (k-1)} (S)$
to   the point $W  \in G_{d}(E)$.  
Note that the tautological exact sequence of vector bundles  over 
the Grassman bundle $G_{m_{p}}(Q_{p-1})$ is regarded as  
\begin{equation*} 
                            0  \longrightarrow  D_{p}/D_{p-1} \hooklongrightarrow    Q_{p-1}  \twoheadlongrightarrow Q_{p} \longrightarrow 0.  
\end{equation*} 

We then define  a subvariety $X_{k}$ of  $\F \ell_{\nu (1), \nu (2), \ldots, \nu (k-1)} (S)$  
by 
\begin{equation*} 
\begin{array}{llll}  
X_{k}  & :=  \{ (0 \subset W_{1} \subset \cdots \subset W_{k-1}  \subset W_{k} = W )  
      \in \F \ell_{\nu (1), \nu (2), \ldots, \nu (k-1)}  (S)   \medskip \\
       & \hspace{3cm}     | \;  W  \subset E_{x} \; (x \in X),  \;  
       W_{p}  \subset (F_{n - d +  \nu (p) - n_{p}})_{x} \; (p = 1, 2, \ldots, k)  \}.     \medskip 
\end{array} 
\end{equation*}  
By definition,  a point  $(0 \subset W_{1} \subset \cdots \subset W_{k-1} \subset W_{k} = W) 
\in X_{k}$  satisfies  the conditions 
$W_{p}  \subset (F_{n - d + \nu (p) - n_{p}})_{x}$, 
and hence 
$W  \cap (F_{n - d + \nu (p) - n_{p}})_{x}   \supset W_{p}$ for $p = 1, 2, \ldots, k$. 
Therefore its image $W  \in G_{d}(E)$ under the projection  $\varpi$ satifies the conditions 
$\dim \, (W \cap (F_{n - d + \nu (p) - n_{p}})_{x}) \geq \nu (p) \; (p = 1, 2, \ldots, k)$. 
Thus   $W$ is in   
$\Omega_{\lambda}(F_{\bullet})$ by the definition  (\ref{eqn:DefinitionSchubertVariety(EssentialBox)})  of $\Omega_{\lambda} (F_{\bullet})$.    
The map  $\varpi|_{X_{k}}:  X_{k} \longrightarrow  \Omega_{\lambda}(F_{\bullet})$
is a resolution of singularities of $\Omega_{\lambda}(F_{\bullet})$ 
constructed by Damon.\footnote{Damon's resolution  is one of the 
{\it small resolution} of singularities of a Schubert variety 
constructed by Zelevinsky \cite{Zelevinsky1983}.   
}

\subsubsection{
Class of Damon's resolution} 
Having constructed the resolution $X_{k}$   
of $\Omega_{\lambda} (F_{\bullet})$, we then define 
the {\it Damon class}  $\delta_{\lambda}  \in MU^{*}(G_{d}(E))$
associated to a partition $\lambda \subset ((n-d)^{d})$ as the 
push-forward image $\varpi_{*} ([X_{k}])$ of the class 
$[X_{k}]  \in  MU^{*}(\F \ell_{\nu (1), \nu (2), \ldots \nu (k)}(S))$.  
One can  compute the class $[X_{k}]$ explicitly by the  standard
fact about the top Chern class of a vector bundle (see e.g., 
Quillen \cite[\S 2]{Quillen1971}, Levine-Morel \cite[Lemma 6.6.7]{Levine-Morel2007}):  
The vector bundle homomorphism $D_{p}/D_{p-1}  \hooklongrightarrow  E \twoheadlongrightarrow E/F_{n - d + \nu (p) - n_{p}}$    over $\F \ell_{\nu (1), \nu (2), \ldots, \nu (k-1)}  (S)$   
defines a   section $s_{p}$ of the  vector bundle 
 $\Hom \, (D_{p}/D_{p-1}, E/F_{n - d + \nu (p) - n_{p}})   
\cong (D_{p}/D_{p-1})^{\vee}  \otimes E/F_{n - d + \nu (p) - n_{p}}$
for $p = 1, 2, \ldots, k$,  and  
the conditions $W_{p}   \subset (F_{n - d + \nu (p) - n_{p}})_{x}$ for $p = 1, 2, \ldots, k$
means that  this  section $s_{p}$  vanishes. 
Thus the variety $X_{k}  \subset \F \ell_{\nu (1), \nu (2), \ldots, \nu (k)}(S)$
is given by the zero locus of the section $s = \bigoplus_{p=1}^{k} s_{p}$ of the vector bundle 
$\bigoplus_{p=1}^{k}  (D_{p}/D_{p-1})^{\vee}  \otimes E/F_{n - d + \nu (p) - n_{p}}$. 
Therefore the class $[X_{k}]$ is given by the top Chern class of this vector 
bundle, that is, 
\begin{equation*} 
     [X_{k}]   = \prod_{p=1}^{k}  c^{MU}_{m_{p} + n_{p} + d  - \nu (p)}  ((D_{p}/D_{p-1})^{\vee}  \otimes E/F_{n - d + \nu (p) - n_{p}} ).    
\end{equation*} 
In order to proceed with the computation, 
we consider the full flag bundle $\F \ell_{1, 2, \ldots, d-1}(S)   = \F \ell (S) \longrightarrow 
G_{d}(E)$.   On $\F \ell (S)$, one has the tautological sequence of 
flag of subbundles 
\begin{equation*} 
  S_{1} \hooklongrightarrow S_{2} \hooklongrightarrow \cdots \hooklongrightarrow 
  S_{i}  \hooklongrightarrow \cdots \hooklongrightarrow  S_{d-1} \hooklongrightarrow S_{d} = S, 
\end{equation*} 
where $\rank S_{i} = i \; (i = 1, 2, \ldots, d)$.  
By the natural projection $\F \ell (S)  \longrightarrow \F \ell_{\nu (1), \nu (2), \ldots, \nu (k)} (S)$, the bundle $D_{p}$ over $\F \ell_{\nu (1), \nu (2), \ldots, \nu (k)}(S)$ 
is  pulled-back to $\bigoplus_{i=1}^{\nu (p)} S_{i}$.   Now we set 
\begin{equation*} 
    y_{i} :=  c^{MU}_{1} ((S_{i}/S_{i-1})^{\vee}) \quad (i = 1, 2, \ldots, d). 
\end{equation*} 
These are the $MU^{*}$-theory Chern roots 
of the dual bundle $S^{\vee}$.   Also  we set 
\begin{equation*} 
             b_{i} :=  c^{MU}_{1}(F_{n + 1 - i}/F_{n - i})   \in MU^{*}(X)   \; (i = 1, 2, \ldots, n). 
\end{equation*} 
These   are 
the $MU^{*}$-theory Chern roots of $E$.    
Then by the splitting principle,  we have 
\begin{equation*} 
\begin{array}{lll} 
   [X_{k}]   & =  \displaystyle{\prod_{p=1}^{k}}    c^{MU}_{m_{p} + n_{p} + d  - \nu (p)}  ((D_{p}/D_{p-1})^{\vee}  \otimes E/F_{n - d + \nu (p) - n_{p}} )  
                  =  \prod_{p=1}^{k}  \left (\prod_{\substack{\nu (p-1) < i \leq  \nu (p) \\
                                                                     1 \leq j \leq n_{p} + d - \nu (p)   }  
                                                            } (y_{i} +_{\L} b_{j})
                                             \right )  \medskip \\
            & =  \displaystyle{\prod_{p=1}^{k}}  \left ( \prod_{\nu (p-1) < i \leq \nu (p)} 
                                                                 [y_{i} |  \b_{n}]^{n_{p} + d - \nu (p)}  
                                                           \right )  
     =  (\y|  \b_{n})^{[\lambda]}.  
\end{array}
\end{equation*} 
Here $\b_{n} = (b_{1},  \ldots,  b_{n}, 0, 0, \ldots)$.  
Thus we have  
\begin{equation*} 
   \delta_{\lambda}  =  \varpi_{*}  ((\y| \b_{n})^{[\lambda]} ).  
\end{equation*}

Now, for a partition 
\begin{equation*} 
  \lambda = (n_{1}^{m_{1}} \; n_{2}^{m_{2}} \; \cdots \; n_{k}^{m_{k}}), \quad 
   n-d  \geq n_{1} > n_{2} > \cdots > n_{k} \geq 0,  
\end{equation*} 
consider the partial flag bundle 
$\eta_{\lambda}: \F \ell^{\lambda} (S^{\vee})  \longrightarrow G_{d}(E)$
associated with the vector bundle    $S^{\vee}  \longrightarrow G_{d}(E)$. 
Since  we have 
\begin{equation*} 
\begin{array}{llll} 
  \F \ell^{\lambda}(S^{\vee})  & = \F \ell^{d-m_{k}, d-m_{k}-m_{k-1},  \ldots, d-m_{k} - m_{k-1} - \cdots - m_{2}}(S^{\vee})   = \F  \ell_{m_{k}, m_{k} + m_{k-1}, \ldots, m_{k} + m_{k-1} + \cdots + m_{2}}(S^{\vee})   \medskip \\
  & \cong   \F \ell_{m_{1}, m_{1} + m_{2}, \ldots, m_{1} + m_{2} + \cdots + m_{k-1}}(S) 
     = \F \ell_{\nu (1), \nu (2), \ldots, \nu (k-1)} (S), 
\end{array} 
\end{equation*} 
the partial flag bundle 
 $\varpi:  \F \ell_{\nu (1), \nu (2), \ldots, \nu (k-1)}(S) \longrightarrow G_{d}(E)$
is identified with   the partial flag bundle 
$\eta_{\lambda}:  \F \ell^{\lambda}(S^{\vee})  \longrightarrow G_{d}(E)$. 
Therefore computing   $\varpi_{*} ((\y| \b_{n})^{[\lambda]})$ is equivalent 
to  computing  
$(\eta_{\lambda})_{*} ((\y| \b_{n})^{[\lambda]})$.  
Our Gysin formula (\ref{eqn:eta_lambda((x|b)^[lambda])}) then  
yields the following: 
\begin{theorem}   \label{thm:DamonClass}  
The  Damon  class $\delta_{\lambda}$ is represented  by 
the new universal factorial Schur function.  More precisely, we have 
\begin{equation*} 
     \delta_{\lambda}  
                          = \mathbb{S}^{\L}_{\lambda}(\y_{d} | \b_{n})  
  \in MU^{*}(G_{d}(E)).   
\end{equation*}   
\end{theorem}  
By Theorem  \ref{thm:DamonClass}, 
one may call the new universal factorial Schur function  $\mathbb{S}^{\L}_{\lambda}(\x_{n}|\b)$
the  universal factorial Schur function of {\it Damon type}.

\subsection{Concluding remarks and related work
}  
So far, we have considered   two types of  universal analogues of usual Schur polynomials, 
namely, the {\it universal factorial Schur  functions}   $s^{\L}_{\lambda}(\x_{n}|\b)$ 
and the {\it new universal factorial Schur functions} or the {\it universal factorial 
Schur functions of Damon type}   $\s^{\L}_{\lambda}(\x_{n}|\b)$.  
In closing this paper, we shall introduce  another type of universal analogue 
of Schur polynomials, which was essentially introduced by 
Hudson-Matsumura \cite[Definition 4.1]{Hudson-Matsumura2016}, 
and state briefly our  new results.  Details will be discussed in our 
forthcoming paper \cite{Nakagawa-Naruse2016II}.   
Hudson-Matsumura     introduced  the class  $\kappa_{\lambda}$ 
$($in the algebraic cobordism ring 
$\Omega^{*} (G_{d}(E))$$)$ using the 
{\it Kempf-Laksov resolution} $($see  op. cit. \cite[\S 4.2]{Hudson-Matsumura2016}, 
Kempf-Laksov \cite{Kempf-Laksov1974}$)$ of the Schubert varieties.  They call the class $\kappa_{\lambda}$ the {\it Kempf-Laksov class}.  
Using the Gysin formula $(\ref{eqn:eta_lambda((x|b)^[lambda])})$ 
and Theorem $\ref{thm:Nakagawa-Naruse(GeneralFlagBundles)}$, 
we   are able to show   that the Kempf-Laksov class $\kappa_{\lambda}$ is given 
explicitly   by 
$($here $r$ is the length of $\lambda$$)$ 
\begin{equation}     \label{eqn:Kempf-LaksovClass}   
\begin{array}{llll}  
   \kappa_{\lambda}   
     &  = (\eta_{\rho_{r}})_{*}([\y| \b_{n}]^{\lambda + \rho_{r-1} + (d - r)^{r}}) \medskip \\
     &        =   \displaystyle{\sum_{\overline{w} \in S_{d}/(S_{1})^{r} \times S_{d-r}}}  
                                  w \cdot \left [ 
                                                    \dfrac{ 
                                                                [\y | \b_{n}]^{\lambda + \rho_{r-1} + (d - r)^{r}} 
                                                              }
                                                             {\prod_{1 \leq  i \leq r} \prod_{i  <  j \leq d} 
                                                                   (y_{i} +_{\L} \overline{y}_{j}) }
                                             \right ]. 
\end{array}  
\end{equation}
Motivated by this formula, we make the following definition: 
\begin{defn} [Universal factorial Schur functions of Kempf-Laksov type\footnote{
``K-L type'' for short.  
}] 
For a partition $\lambda  \in \mathcal{P}_{n}$ with length $\ell (\lambda) = r \leq n$, 
we define 
\begin{equation*} 
\begin{array}{llll}  
   s^{\L, KL}_{\lambda}(\x_{n}|\b)  
     &  :=  \displaystyle{\sum_{\overline{w}  \in S_{n}/(S_{1})^{r} \times S_{n-r}}}   
      w \cdot \left [ 
                   \dfrac{ 
                                [\x | \b]^{\lambda + \rho_{r-1} + (n - r)^{r}} 
                            }
                           {\prod_{1 \leq i  \leq r}  \prod_{i  <  j \leq n} 
                                             (x_{i} +_{\L} \overline{x}_{j}) }
                                             \right ]     \medskip \\
       &=   \displaystyle{\sum_{\overline{w}  \in S_{n}/(S_{1})^{r} \times S_{n-r}}}   
      w \cdot \left [ 
                   \dfrac{ 
                                  \prod_{i=1}^{r}  [x_{i}|\b]^{\lambda_{i}  - i + n}  
                            }
                           {\prod_{1 \leq i  \leq r}   \prod_{i  <  j \leq n} 
                                             (x_{i} +_{\L} \overline{x}_{j}) }
                                             \right ].      \medskip 
\end{array}   
\end{equation*} 
We also define 
\begin{equation*} 
         s^{\L, KL}_{\lambda}(\x_{n}) :=   s^{\L, KL}_{\lambda}(\x_{n}| \bm{0}). 
\end{equation*} 
\end{defn} 

From Theorem \ref{thm:Nakagawa-Naruse(GeneralFlagBundles)},  we obtain 
the following theorem: 
\begin{theorem} [Characterization of the universal Schur functions of K-L type] 
\label{thm:Nakagawa-Naruse(CharacterizationUSFKL)}  
For the Gysin homomorphism $(\tau^{r})_{*}:   MU^{*}(\F \ell^{r, r-1, \ldots, 2, 1}(E)) 
\longrightarrow    MU^{*}(X)$,  the following formula holds$:$ 
\begin{equation*} 
    (\tau^{r})_{*} (\x^{\lambda + \rho_{r-1} + (n-r)^{r}}) 
         = s^{\L, KL}_{\lambda}(E). 
\end{equation*} 
Here $x_{1}, \ldots, x_{n}$ are the $MU^{*}$-theory Chern roots of $E$ 
as in \S $\ref{subsubsec:ApplicationBressler-EvensFormula(Cobordism)}$.   
\end{theorem}

On the other hand,   
 following Hudson-Matsumura \cite[Definition 3.1]{Hudson-Matsumura2016},   we define the  {\it $MU^{*}$-theory $k$-th Segre class}  $\mathscr{S}^{\L}_{k}(E)$ 
of a complex vector bundle $E$ to be 
\begin{equation}    \label{eqn:SegreClass(ComplexCobordism)}     
       \mathscr{S}^{\L}_{k}(E)  :=  \pi_{1 *} (x_{1}^{k + n -1})
\end{equation} 
 in the same setting as in Theorem \ref{thm:QuillenResidueFormula}. 
Let 
\begin{equation*} 
    \mathscr{S}^{\L}_{t}(E)  = \mathscr{S}^{\L} (E; t) 
    = \sum_{k \in \Z}  \mathscr{S}^{\L}_{k}(E) t^{k} 
\end{equation*} 
be the {\it Segre series} of $E$ in the complex cobordism theory.  
Then by Example \ref{ex:GysinFormulasNUSF} (1) and (\ref{eqn:SegreClass(ComplexCobordism)}),   we see immediately 
that $\s^{\L}_{k}(E) = \mathscr{S}^{\L}_{k}(E)$, and therefore 
the Segre series $\mathscr{S}^{\L}_{t}(E)$ is given by 
Theorem \ref{thm:GeneratingFunctionNUSF}.   
Using the Segre series $\mathscr{S}^{\L}_{t}(E)$,   the 
{\it universal push-forward  formula for full flag bundle of type $A$}  
 established by Darondeau-Pragacz \cite[pp.4--5]{Darondeau-Pragacz2015}\footnote{
Note that this type of Gysin formula is also obtained by Ilori \cite[p.623, Theorem]{Ilori1978}, 
Kaji-Terasoma \cite[Theorem 0.4 (Push-Forward Formula)]{Kaji-Terasoma2015}.   
}
can be generalized to the complex cobordism theory.   
\begin{theorem} [Darondeau-Pragacz formula in complex cobordism] 
    In the same setting as in Theorem  $\ref{thm:Nakagawa-Naruse(CharacterizationUSFKL)}$, 
  one has 
\begin{equation*} 
     (\tau^{r})_{*}(f(x_{1}, \ldots, x_{r}))  
   =  [t_{1}^{n-1} \cdots t_{r}^{n-1}]  \left ( f(t_{1}, \ldots, t_{r}) 
   \prod_{1 \leq i < j \leq r}  (t_{j} +_{\L} \overline{t}_{i}) 
    \prod_{1 \leq i \leq r}  \mathscr{S}^{\L}_{1/t_{i}}  (E)  
                                                \right ) 
\end{equation*} 
for a polynomial $f(X_{1}, \ldots, X_{r})  \in MU^{*}(X)[X_{1}, \ldots, X_{r}]$.  
\end{theorem}  
This formula, together with Theorem \ref{thm:Nakagawa-Naruse(CharacterizationUSFKL)}, 
enables us to obtain the generating function 
of the universal Schur functions of K-L type (see Nakagawa-Naruse \cite{Nakagawa-Naruse2016II}).



\end{document}